\documentclass[reqno,11pt]{amsart}
\usepackage{amsthm,amsfonts,amssymb,euscript,bbm,amsmath}
\usepackage{color}
\usepackage{empheq}
\usepackage[pdftex]{graphicx}
\usepackage{hyperref}
\newcommand{\bea}{\begin{eqnarray}}
\newcommand{\eea}{\end{eqnarray}}
\def\beaa{\begin{eqnarray*}}
\def\eeaa{\end{eqnarray*}}
\def\ba{\begin{array}}
\def\ea{\end{array}}
\def\be#1{\begin{equation} \label{#1}}
\def \eeq{\end{equation}}

\def\be{{\beta}}

\def\R{{\mathbb{R}}}
\def\C{{\mathbb{C}}}

\def\N{{\bf N}}

\def\Z{{\mathbb{Z}}}

\def\T{{\mathbb{T}}}

\newtheorem{theorem}{Theorem}[section]
\newtheorem{lemma}[theorem]{Lemma}
\newtheorem{proposition}[theorem]{Proposition}
\newtheorem{corollary}[theorem]{Corollary}

\newtheorem{remark}[theorem]{Remark}

\numberwithin{equation}{section}

\begin{document}

\author{Zaher Hani, Benoit Pausader, Nikolay Tzvetkov and Nicola Visciglia}
\address{Courant Institute of Mathematical Sciences\\ 251 Mercer Street, New York NY 10012}
\email{hani@cims.nyu.edu}
\address{Universit\'e Paris-Nord}
\email{pausader@math.univ-paris13.fr}
\address{Universit\'e Cergy-Pontoise}
\email{nikolay.tzvetkov@u-cergy.fr}
\address{Universita di Pisa}
\email{viscigli@dm.unipi.it}

\thanks{Z.~H. is supported by a Simons Postdoctoral Fellowship and NSF Grant DMS-1301647, B.~P.  is partially supported by NSF Grant DMS-1142293 and ANR project Scheq ANR-12-JS-0005-01, N.~T. is  supported by the ERC grant Dispeq., N. V. is partially supported by project FIRB Dinamiche Dispersive.}

%\subjclass{35Q55}
\keywords{Modified Scattering, Nonlinear Schr\"odinger equation, wave guide manifolds, energy cascade, weak turbulence}

\title[Modified Scattering on $\mathbb{R}\times\mathbb{T}^d$]{Modified scattering for the cubic Schr\"odinger equation on product spaces and applications}

%%%%%%%%%%%%%%%%%%%%%%%%%%%%%%%%%%%%%%%%%%%%%%%%%%%%%%%%%%%%%%%%%%%%%%%%%%%%%%%%%%%%%%%%%%%%%%%%%%%%%%%%%%%%%%%%%%%%%%%%%%%%%%%%%%%%%%%%%%%%%%%%%%%%%%%%%%%%%%%%%%%%%%%%%%%%%%%%%%%%%%%%%%%%%%%%%%%%%%%%%%%%%%%%%%%%%%%%%%%%%%%%%%%%%%%%%%%%%%%%%%%%%%%%%%%%%%%%%%%%%%%%%%%%%%%%%%%%%%%%%%%%%%%%%%%%%%%%%%%%%%%%%%%%%%%%%%%%%%%%%%%%%%%%%%%%%%%%%%%

\begin{abstract}
We consider the cubic nonlinear Schr\"odinger equation posed on the spatial domain $\R\times \T^d$. We prove modified scattering and construct modified wave operators for small initial and final data respectively ($1\leq d\leq 4)$. The key novelty comes from the fact that the modified asymptotic dynamics are dictated by the \emph{resonant system} of this equation, which sustains interesting dynamics when $d\geq 2$. As a consequence, we obtain global strong solutions 
%to the defocusing and focusing problems on $\R\times \T^d$ 
(for $d\geq 2$) with infinitely growing high Sobolev norms $H^s$.
\end{abstract}
\maketitle

%%%%%%%%%%%%%%%%%%%%%%%%%%%%%%%%%%%%%%%%%%%%%%%%%%%%%%%%%%%%%%%%%%%%%%%%%%%%%%%%%%%%%%%%%%%%%%%%%%%%%%%%%%%%%%%%%%%%%%%%%%%%%%%%%%%%%%%%%%%%%%%%%%%%%%%%%%%%%%%%%%%%%%%%%%%%%%%%%%%%%%%%%%%%%%%%%%%%%%%%%%%%%%%%%%%%%%%%%%%%%%%%%%%%%%%%%%%%%%%%%%%%%%%%%%%%%%%%%%%%%%%%%%%%%%%%%%%%%%%%%%%%%%%%%%%%%%%%%%%%%%%%%%%%%%%%%%%%%%%%%%%%%%%%%%%%%%%%%%%

\maketitle

\section{Introduction}

The purpose of this work is to study the asymptotic behavior of the cubic defocusing nonlinear Schr\"odinger equation posed on the wave-guide manifolds $\R\times \T^d$:
\begin{equation}\label{CNLS}
\left(i\partial_t +\Delta\right) U=\vert U\vert^2U
\end{equation}
where $U$ is a complex-valued function on the spatial domain $(x,y)\in \R\times \T^d$.
%, and $\lambda \in \{+1, -1\}$ is the sign of the nonlinearity ($+1$ for defocusing and $-1$ for focusing\footnote{In this paper, we will however be concerned with small data, so this distinction will not be relevant.}). 
In particular, we want to understand how this asymptotic behavior is related to a {\it resonant dynamic}, in a case when scattering does not occur.
Our results can be directly extended to the case of a focusing nonlinearity ($-|U|^2U$ in the left hand-side of \eqref{CNLS}) but 
we will however be concerned with small data, so this distinction on the nonlinearity will not be relevant. 
On the other hand the result of Corollary~\ref{Inf cascade cor} providing solutions blowing-up at infinite time is more striking in the defocusing case because in the focusing case one may have blow-up in finite time (via the quite different mechanism of self-focusing).

\subsection{Motivation and background}

The question of the influence of the geometry on the global behavior of solutions to the nonlinear Schr\"odinger equation
\begin{equation}\label{NLS}
\left(i\partial_t+\Delta\right) u=\lambda\vert u\vert^{p-1}u,\quad p>1
\end{equation} dates back at least to \cite{BrGa}. The first natural question is the issue of global existence of solutions, and many works have investigated this problem in different geometric settings \cite{BaCaDu,Bo2,Bo1,BoCrit,Bourgain2013,BuGeTz,BuGeTz3,BuGeTz2,CKSTTcrit,DoExt,Ha,HaPa,He,HeTaTz,HeTaTz2,IoPa,IoPa2,IoPaSt,KeMe,KiVi,KiViZh,PaTzWa,TaTz,Vi}. 
The conclusion that could be derived from these works is that the geometry of the spatial domain turned out to be of importance in the context of the best possible Strichartz inequalities or the sharp local in time well-posedness results (see e.g. \cite{Bo1,BuGeTz-mrl,BuGeTz}). However, the analysis in \cite{HeTaTz,HeTaTz2,IoPa,IoPa2,IoPaSt,PaTzWa} seems to indicate that, at least in the defocusing case\footnote{In the focusing case and for large data, it is likely that existence or nonexistence questions to elliptic problems also plays an important role.}, the only relevant geometric information for the global existence in the energy space is the ``local dimension'', i.e. the dimension of the tangent plane.

The next natural question concerns the asymptotic behavior. There the geometry must play a more important role. This is the question in which we are interested in this paper, focusing on the simpler case of noncompact quotients of $\mathbb{R}^d$.

When the domain is the Euclidean space, $\mathbb{R}^d$, this question is reasonably well-understood at least when the nonlinearity is defocusing and analytic ($p$ odd integer). In this case, global smooth solutions disperse and in many cases even scatter to a linear state 
(possibly after modulation by a real phase when $d=1$, $p=3$) \cite{Cazenave:book,CKSTTcrit,DeZh,Do1,Do,HaNa,KeMe,KiTaVi,Oz,RyVi,Vi,ZS}.

In contrast, much less is known for compact domains. The most studied example is that of the torus $\mathbb{T}^d$. In this case, many different long-time behaviors can be sustained even on arbitrarily small open sets around zero, ranging from KAM tori \cite{BourgainQuasi,EliaKuk,KukPo,Procesi}
to heteroclinic orbits \cite{CKSTTTor,GuKa} and coherent out-of-equilibrium frequency dynamics\footnote{Interestingly, all these long-time results derive from an analysis of resonant interactions that will play a central role in this work as well.} \cite{FGH}. 
One may also mention \cite{JB1,JB2,Bulut,tz_fourier}, where invariant measures for \eqref{CNLS} are constructed, when the problem is posed on  $\T^d$, $d=1,2$ or the $d$ dimensional ball for $d=2,3$ (with radial data). These works establish the existence of a large set of (not necessarily small) recurrent dynamics of \eqref{CNLS}.

In light of the above sharp contrast in behavior between $\R^d$ and its compact quotient $\T^d$, considerable interest has emerged in the past few years to study questions of long-time behavior on ``in between" manifolds, like the ones presented by the non-compact quotients of Euclidean space \cite{HaPa,HeTaTz2,IoPa,TaTz,TeTzVi,TzVi}. 

\medskip

In the generality of non-compact Riemannian $d-$manifolds $M$, it seems plausible that a key role is played by the parameter $\alpha$ for which solutions to the linear NLS equation (\eqref{NLS} with $\lambda=0$) with smooth compactly supported initial data decay like $t^{-\alpha/2}$. In light of the Euclidean theory on $\R^\alpha$, one can draw the following hypothetical heuristics: H1) when $p>1+4/\alpha$, global solutions (sufficiently small in the focusing case) scatter and no further information is needed about the geometry ``at infinity''; H2) if $p=1+4/\alpha$, global solutions scatter, but the geometry ``at infinity'' plays an important role in the analysis of certain sets of solutions (e.g. in the profile decomposition); H3) if $p\le 1+2/\alpha$, no nontrivial solution can scatter and H4) if $p=1+2/\alpha$, global solutions exhibit some ``modified scattering'' characterized by a correction to scattering on a {\it larger time-scale}. We are interested in this latter regime to which \eqref{CNLS} belongs.

In support of the heuristic H1) we cite the results in \cite{BaCaDu,IoPaSt,IoSt,TzVi,TzVi2}. The second heuristic H2) was put to test in \cite{HaPa} where the authors study the quintic NLS equation on $\R\times \T^d$. There, a strong relation is drawn between the large-data scattering theory for the quintic NLS equation and the system obtained from its \emph{resonant periodic frequency interactions}. The relevance of the result in \cite{HaPa} to our work here lies in the following two important messages: The first is that the asymptotic behavior of \eqref{NLS} on $\R^n \times \T^d$ can be understood through: i) the asymptotic dynamics of the same equation on Euclidean spaces, and ii) the asymptotic dynamics of a related resonant system corresponding to the resonant interactions between its periodic frequency modes. The second message from \cite{HaPa} is the insight that the resonant interactions in \eqref{CNLS} will play a vivid and decisive role in dictating the anticipated non-scattering asymptotic dynamics of \eqref{CNLS}. Indeed, as \cite{HaPa,TzVi} show that quintic interactions lead to scattering behavior for small data, and since non-resonant interactions in \eqref{CNLS} can be transformed, at least formally, into quintic interactions via a normal form transformation, it is up to the resonant interactions alone to drive the system away from scattering. This is the content of our main result.

The other interesting feature of the asymptotic dynamics of \eqref{CNLS} as opposed to previous modified scattering results, is that the modification dictated by its resonant system is not simply a phase correction term when $d\geq 2$, but rather a much more vigorous departure from linear dynamics. As we argue below, this will pose a new set of difficulties in comparison to previous modified scattering results in the literature, but, on the plus side, will lead us to several interesting and new types of asymptotic dynamics. 

\subsection{Statement of the results}

Consistent with the heuristics above, we show that the asymptotic dynamic of small solutions to \eqref{CNLS} is related to that of solutions of the resonant system
\begin{equation}\label{RSS}
\begin{split}
i\partial_\tau G(\tau)=&\quad\mathcal R[G(\tau), G(\tau), G(\tau)],\\
\mathcal F_{\R\times \T^d}\, \mathcal R[G, G, G](\xi, p)=&\quad\sum_{\substack{p_1+p_3=p+p_2\\\vert p_1\vert^2+\vert p_3\vert^2=\vert p\vert^2+\vert p_2\vert^2}}\widehat{G}(\xi,p_1)\overline{\widehat{G}(\xi,p_2)}\widehat{G}(\xi,p_3).
\end{split}
\end{equation}
Here $\widehat G(\xi, p)=\mathcal F_{\R\times \T^d} G (\xi, p)$ is the Fourier transform of $G$ at $(\xi, p)\in \R\times \Z^d$. Noting that the dependence on $\xi$ is merely parametric, the above system is none other than the resonant system for the cubic NLS equation on $\mathbb{T}^d$.
The equation \eqref{RSS} is globally well-posed thanks to Lemma~\ref{GlobSol} below.

More precisely, our main results are as follows. Below $N\geq 30$ is an arbitrary integer, and $S$ and $S^+$ denote Banach spaces whose norms are defined in \eqref{DefSNorm} later. They contain all the Schwartz functions.

\begin{theorem}\label{ModScatThm}
Let $1\leq d \leq 4$. There exists $\varepsilon=\varepsilon(N,d)>0$ such that if $U_0\in S^+$ satisfies
\begin{equation*}
\Vert U_0\Vert_{S^+}\le\varepsilon,
\end{equation*}
and if $U(t)$ solves \eqref{CNLS} with initial data $U_0$, then $U\in C((0,+\infty):H^N)$ exists globally and exhibits modified scattering to its resonant dynamics \eqref{RSS} in the following sense: there exists $G_0\in S$ such that if $G(t)$ is the solution of \eqref{RSS} with initial data $G(0)=G_0$, then
\begin{equation*}
\Vert U(t)-e^{it\Delta_{\mathbb{R}\times\mathbb{T}^d}}G(\pi\ln t)\Vert_{H^N(\mathbb{R}\times\mathbb{T}^d)}\to 0\qquad\hbox{ as }t\to+\infty.
\end{equation*}
Moreover
$$
\|U(t)\|_{L^\infty_x H^1_y}\lesssim (1+|t|)^{-\frac{1}{2}}.
$$
\end{theorem}

A similar statement holds as $t\to-\infty$, and a more precise one is contained in Theorem \ref{ModScatThm2}. It is worth pointing out that for $d=4$, even the global existence claim in the above theorem is new, due to the energy-supercritical nature of \eqref{CNLS} in this dimension. However, the main novelty is the modified scattering statement to a non-integrable asymptotic dynamics, given by \eqref{RSS}. 

In addition, we construct modified wave operators in the following sense:

\begin{theorem}\label{ExMWO}
Let $1\le d\le 4$. There exists $\varepsilon=\varepsilon(N,d)>0$ such that if $G_0\in S^+$ satisfies
\begin{equation*}
\Vert G_0\Vert_{S^+}\le\varepsilon,
\end{equation*}
and $G(t)$ solves \eqref{RSS} with initial data $G_0$, then there exists $U\in C((0,\infty):H^N)$ a solution of \eqref{CNLS} such that
\begin{equation*}
\begin{split}
\Vert U(t)-e^{it\Delta_{\mathbb{R}\times\mathbb{T}^d}}G(\pi\ln t)\Vert_{H^N(\mathbb{R}\times\mathbb{T}^d)}\to0\,\hbox{ as }\,t\to+\infty.
\end{split}
\end{equation*}
\end{theorem}
%%%%%%
\begin{remark}
It is worth mentioning that a 
slight modification of the proof of Theorems~\ref{ModScatThm} and \ref{ExMWO} shows that similar statements hold if $\T^d$ is replaced by the sphere $S^d, d=2,3$
(with a suitably modified resonant system).
Indeed, the largest part of the analysis is exploiting the $1d$ dispersion.
In the case of $S^d, d=2,3$  
the spectrum of the Laplace-Beltrami operator satisfies the non-resonant condition needed for the normal form analysis, and the  
$H^1$ well-posedness analysis on the sphere of  \cite{BuGeTz, BuGeTz3} provides the needed substitute of Lemma~\ref{DiscreteStricTorLem}.
A similar remark applies to the case of a partial harmonic confinement (cf. \cite{HaTh}). 
On the other hand, the extension of our analysis to an irrational torus is less clear because of the appearance of small denominators in the normal form analysis.
\end{remark}
%%%%%%%%
As a consequence of Theorem \ref{ExMWO}, all the behaviors that can be isolated for solutions of the resonant system \eqref{RSS} have counterparts in the asymptotic behavior of solutions of \eqref{CNLS}. Most notably, given the existence of unbounded Sobolev orbits for \eqref{RSS} as proved in \cite{Hani} for $d\geq 2$ (cf. Theorem \ref{growth on Z^d theo} for an explicit construction with quantitative lower bounds on the growth), we have the following.

\begin{corollary}[Existence of infinite cascade solutions]\label{Inf cascade cor}
Let $d\geq 2$ and $s\in \mathbb{N}$, $s\ge 30$. Then for every $\varepsilon>0$ there exists a global solution $U(t)$ of \eqref{CNLS} such that
\begin{equation}\label{Inf cascade}
\Vert U(0)\Vert_{H^s(\mathbb{R}\times\mathbb{T}^d)}\le\varepsilon,\qquad\limsup_{t\to+\infty}\Vert U(t)\Vert_{H^s(\mathbb{R}\times\mathbb{T}^d)}=+\infty.
\end{equation}
More precisely, there exists a sequence $t_k\to +\infty$ for which 
$$
\Vert U(t_k)\Vert_{H^s(\mathbb{R}\times\mathbb{T}^d)}\gtrsim \exp(c(\log\log t_k)^\frac{1}{2}).
$$
\end{corollary}

\begin{remark}\label{NoCascade}
These infinite cascades \emph{do not occur} when $d=1$ on $\mathbb{R}\times\mathbb{T}$ (nor when $d=0$ on $\mathbb{R}$). At least not for small smooth localized solutions. In fact, (see \eqref{Explicit1d}), the asymptotic dynamic of small solutions to \eqref{CNLS} is fairly similar on $\mathbb{R}$ and on $\mathbb{R}\times\mathbb{T}$, in sharp contrast with the case $d\ge2$.
\end{remark}

%%%%%%%%%%%%%%%%%%%%%
%%%%%%%%%%%%%%%%%%%%%

Corollary~\ref{Inf cascade cor}  gives a partial solution to a problem posed by Bourgain \cite[page 43-44]{Bo7} concerning the possible long time growth of the $H^s$, $s>1$ norms for the solutions of the cubic nonlinear Schr\"odinger equation. This growth of high Sobolev norms is regarded as a proof of the (direct)\emph{energy cascade phenomenon} in which the energy of the system (here the kinetic energy) moves from low frequencies (large scales) towards arbitrarily high frequencies (small scales). Heuristically, the solution in Corollary \ref{Inf cascade cor} can be viewed as initially oscillating at scales that are $O(1)$, but at later times exhibits oscillations at arbitrarily smaller length-scales. This energy cascade is a main aspect of the out-of-equilibrium dynamics predicted for \eqref{CNLS} by the vast literature of physics and numerics falling under the theory of \emph{weak (wave) turbulence} (cf. \cite{MMT, ZLF}).

The corresponding result on $\T^d$ does not directly follow from Corollary \ref{Inf cascade cor} (nor does it imply it). This is somehow surprising because  one would naturally expect that adding a dispersive direction to $\T^d$ would drive the system closer to nonlinear asymptotic stability, and further from out-of-equilibrium dynamics (this is indeed the case if we study the equation on $\R^n\times \T^d$ for $n\geq 2$ as was shown by the scattering result in \cite{TzVi}). Our construction draws heavily on \cite{CKSTTTor, GuKa, Hani} where unbounded Sobolev orbits are constructed for the resonant system and applied to get finite time amplifications of the Sobolev norms  on $\mathbb{T}^2$.  However, in the case of the torus, nonresonant interactions do not disappear and feed back into the dynamics after a long but finite time. This is precisely where the \emph{more dispersive} setting of $\mathbb{R}\times\mathbb{T}^d$ makes a difference: in this case, nonresonant terms are transformed into quintic terms \emph{which scatter}, and hence, at least heuristically do not modify the long-term dynamics.

Previous results in the spirit of Corollary \ref{Inf cascade cor}  may be found in \cite{BoGro1,BoGro2} for linear Schr\"{o}dinger equations with potential,  \cite{CKSTTTor,GuKa,Ku} for finite time amplifications of the initial $H^s$ norm,  \cite{BoAspects, BoJAnalM, Hani} for NLS with suitably chosen non-local nonlinearities, and \cite{GerGrel,GerGre2,GerGre3,Poc,HaiXu} for the zero-dispersion Szeg\"{o} and half-wave equations. Concerning the opposite question of obtaining upper bounds on the rate of possible growth of the Sobolev norms of solutions of NLS equations we refer to \cite{BoOntheGrowth,B10,CKO,Soh,Staffilani}.

\medskip

One can also use Theorem \ref{ExMWO} to construct other interesting non-scattering dynamics for equation \eqref{CNLS} as is illustrated in the following result.

\begin{corollary}[Forward compact solutions]\label{CorForwComp}
Let $d\geq 2$. For functions $U(t)$ on $\mathbb{R}\times \mathbb{T}^d$ defined for all $t\ge0$, we define the ``limit profile set" as
\begin{equation*}
\omega(U)=\limsup_{t\to+\infty}\{e^{-it\Delta_{\mathbb{R}\times\mathbb{T}^d}}U(t)\}=\cap_{\tau\in(0,\infty)}\overline{\{e^{-it\Delta_{\mathbb{R}\times\mathbb{T}^d}}U(t):\,\,t\ge\tau\}}.
\end{equation*}
Then
\begin{enumerate}
\item (no nontrivial scattering) Assume that $U$ solves \eqref{CNLS} and that $\omega(U)$ is a point. If $U(0)$ is sufficiently small, then $U\equiv 0$.
\item (scattering up to phase correction)
There exists a nontrivial solution $U$ of \eqref{CNLS} and a real function $b:\mathbb{R}\to\mathbb{R}$ such that $\omega(U(t)e^{ib(t)})$ is a point.
\item ((quasi-)periodic frequency dynamics)
There exists a global solution $U(t)$ such that $\omega(U)$ is compact but $\dim\hbox{Span}(\omega(U))\ge 2$. 
\end{enumerate}
\end{corollary}

The proof of part (3) in the above corollary is interesting in its own right. In fact, we construct global solutions to \eqref{CNLS} that asymptotically bounce their energy (and mass) between two disjoint sites in frequency space periodically in time. These correspond to periodic-in-time solutions of \eqref{RSS} that exhibit the following ``beating effect'' (in the nomenclature of \cite{GrPaTh}): there exists two disjoint subsets $\R\times \Lambda_1$ and $\R\times \Lambda_2$ in $\R\times \Z^2$, so that for any $\varepsilon\in (0,1)$ there exists a solution $G(t)$ of \eqref{RSS} that is supported in frequency space on $\R\times \Lambda_1 \cup \R\times \Lambda_2$ in such a way that the fraction of the mass carried by each of the two sets alternates between $\varepsilon$ and $1-\varepsilon$ periodically in time. We refer to Subsection \ref{SpeDyn} for more constructions including asymptotically quasi-periodic dynamics.

\subsubsection{Comments}

It would be interesting to understand what is the optimal topology to obtain our results. It is probably a lot larger than the one we use. Progress in this direction would impact the following:

\begin{itemize}
\item  The results are restricted to small data. In the absence of a ``correct'' topology, the exact meaning of ``large data'' is not well established.
\item We cannot let any $H^s$ norm, $s>1$ grow in Corollary \ref{Inf cascade cor}, partly because we want to cover all the cases $0\le d\le 4$ in a uniform manner, using simple exponents. More careful analysis might address this point (for instance, either lowering the regularity requirement in Theorem \ref{ExMWO} or a more quantified version of the construction in \cite{CKSTTTor} would resolve this). We decided not to pursue this point here because Corollary 1.3 already captures the energy cascade phenomenon.
\item It is possible that a more adapted topology allows to define the scattering operator in a good Banach space.
\end{itemize}

As already mentioned the results of Theorems~\ref{ModScatThm} and \ref{ExMWO}  can be extended to the case of spheres (i.e. $\mathbb{R}\times S^d$, $d=2,3$).  A good understanding of the corresponding resonant system is presently missing. 

Finally, we also mention the situation in \cite{AnCaSi} where the partial periodicity is replaced by adding a (partially) confining potential.

%\newpage

\subsection{Overview of proof} 

\subsubsection{Modified scattering}

In order to describe the asymptotic behavior of a nonlinear dispersive equation like \eqref{CNLS}, it may be relevent to study the limiting behavior of the profile $F(t)=e^{-it\Delta}U(t)$ obtained by conjugating out the linear flow. If $F(t)$ converges to a fixed function $G_\infty$, then the solution scatters. If not, the next best thing is to find the simplest possible dynamical system that describes the asymptotic dynamics of $F(t)$. To find this system, one has to work on proving global a priori energy and decay estimates that allow to decompose the nonlinearity in the $F$ equation in the following way:
\begin{equation}\label{toy scheme}
i\partial_t F= \mathcal N(F)=\mathcal N_{eff}(F)+ \mathcal E(F)
\end{equation}
where $\mathcal E(F)$ is integrable in time. When this is possible, one can hope to prove that the asymptotic dynamics converge to that of the effective system 
\begin{equation}\label{limit sys int}
i\partial_t G=\mathcal N_{eff}(G).
\end{equation}
Proving the global a priori energy and decay estimates can be a daunting task depending on the problem at hand. On the other hand, the process of proving the convergence to the dynamics of \eqref{limit sys int} depends very much on how simple or complicated $\mathcal N_{eff}(G)$ is.

Previous modified scattering results that we are aware of, only concerned equations (or systems) posed on $\mathbb{R}^d$, quasilinear or semilinear \cite{AlDe,Ca,DeZh,Del,HaLiNa,HaNa,HaNaShTo,IoPu1,IoPu2,KaPu,Oz,ZS} and had an integrable asymptotic system for \eqref{limit sys int}. This often allowed for a simple phase conjugation (in physical or Fourier space) to give the modification. In contrast, our limiting system is given by \eqref{RSS} which is not only a non-integrable system, but also allows for the growth of norms of its solutions as we saw in Corollary \ref{Inf cascade cor}. This requires a robust approach to modified scattering, that tolerates the growth of the limiting system $\eqref{limit sys int}$ as long as the decay of $\mathcal E(F)$ in \eqref{toy scheme} is sufficiently fast to trump the divergence effects of the effective part $\mathcal N_{eff}(F)$. 

\subsubsection{Isolating the resonant system: heuristics}

To isolate the effective interactions ($\mathcal{N}_{eff}$ above), we can argue formally by looking at \eqref{CNLS} in Fourier space:
\begin{equation}\label{Duhamel intro}
i\partial _t \widehat F_p(\xi,t) =\sum_{q-r+s=p}e^{i\omega t}\int_{\mathbb{R}^2}e^{2i\eta\kappa t}\widehat{F}_{q}(\xi-\eta,t)\overline{\widehat{F}_{r}}(\xi-\eta-\kappa,t)\widehat{F}_{s}(\xi-\kappa,t)d\kappa d\eta,
\end{equation}
where $\omega=\vert p\vert^2-\vert q\vert^2+\vert r\vert^2-\vert s\vert^2$ and where $\widehat{F}_p(\xi)$ denotes the Fourier transform of $F$ at $(\xi,p)\in\mathbb{R}\times\mathbb{Z}^d$. Roughly speaking, a stationary phase argument in the $(\eta, \kappa)$ integral implies\footnote{see \cite{KaPu} for a previous use of this remark.} that for very large times the equation for $\widehat F_p(\xi,t)$ can be written as
\begin{equation*}
i\partial _t \widehat F_p(\xi,t) =\frac{\pi}{t} \sum_{q-r+s=p}e^{i\omega t}\widehat{F}_{q}(\xi,t)\overline{\widehat{F}_{r}}(\xi,t)\widehat{F}_{s}(\xi,t)+l.o.t.
\end{equation*}
This is essentially an ODE system for each $\xi \in \mathbb{R}$. As is well-known, \emph{resonant interactions} corresponding to  $(p,q,r,s)$ for which $\omega=0$ play a particularly important role in the dynamics of such an ODE, especially given the decay of $\partial_t \widehat F_p$. This suggests that the expression above can be simplified to
\begin{equation*}
i\partial _t \widehat F_p(\xi,t) =\frac{\pi}{t} \sum_{\substack{q-r+s=p\\|q|^2-|r|^2+|s|^2=|p|^2}}\widehat{F}_{q}(\xi,t)\overline{\widehat{F}_{r}}(\xi,t)\widehat{F}_{s}(\xi,t)+l.o.t.
\end{equation*}
As a result, one should expect the asymptotic dynamics of $F$ to be dictated by the ODE system given by the first term on the right-hand side above. The latter system can be seen to be autonomous when written in terms of the slow time scale $\tau=\pi\ln t$ in which it has the form \eqref{RSS}. Note that this system was previously studied and shown to have interesting dynamics \cite{CaFa,CKSTTTor,FGH,Hani}.

The upshot of the above formal calculation is that one should expect a solution $F(t)$ to \eqref{Duhamel intro} to asymptote to some $G(\pi\ln t)$ where $G(\tau)$ solves \eqref{RSS}. This is the content of Theorem \ref{ModScatThm}.

\subsubsection{Norms and the control of the solution} As mentioned above, establishing a priori energy and decay estimates is a precursor to isolating the leading order dynamics. In the scalar case $d=0$ \cite{HaNa, KaPu}, the needed energy estimates follow easily once we guarantee the $t^{-1/2}$ decay for the $L^\infty$ norm. Indeed, schematically speaking, if $E(t)$ is an appropriate energy of the system that controls its strong norms, then one has the relation
$$
\partial_t E(t)\lesssim \|u(t)\|^2_{L^\infty}E(t)\lesssim t^{-1}E(t)
$$
which barely allows to close any polynomial-growth bootstrap for $E(t)$. The $L^\infty$ decay can be bootstrapped by relying on the boundedness of the Fourier transform, which follows from the equation satisfied by $\widehat F(\xi)$. An almost identical energy method argument works in the case $d=1$, but reaches its limit there. Indeed, for $d\geq 2$, we do not have access to the sharp linear decay $t^{-1/2}$ which was crucial to closing the energy bootstrap above. To overcome this difficulty, we need additional estimates coming from the low-regularity theory. We use a hierarchy of three norms.
\begin{itemize}
\item The $Z$-norm is bounded and essentially corresponds to the strongest information that remains a priori bounded uniformly in time.
\item The $S$-norm controls the number of periodic derivatives we want to consider. It grows slowly with time, but the difference with the asymptotic dynamics {\it decays} in this norm.
\item The $S^+$-norm is slightly stronger than the $S$-norm. It is allowed to grow slowly, but still yields better control on objects in the $S$-norm. In particular, it controls the {\it same} number of derivatives in the periodic directions as the $S$-norm.
\end{itemize}
While the choice of the $Z$ norm is dictated by the resonant system, there is considerable flexibility in the choices of the two other norms. Another possible choice might be a variation\footnote{But for the moment, it seems difficult in the proof of the modified wave operator to work with an intermediate norm controlling no weight in $x$.} of  $Z, Z\cap H^N, S$. One of the main problems complicating the situation here is the need for a bounded linearized operator around a solution $G$ of the asymptotic system, which is not trivial in view of the missing $t^{-1/2}$ decay of $\Vert U(t)\Vert_{L^\infty}$.

 The significance of the $Z$ norm stems from the following two key facts: 1) it is conserved for the resonant system\footnote{Ultimately, this leads to the key non-perturbative ingredient, see \eqref{H1NormPres} and \eqref{H1Cancellation}.}, and 2) it is a controlling norm for the existence and growth of its solution\footnote{In particular, this forces the restriction $d\le 4$ in Theorems \ref{ModScatThm} and \ref{ExMWO}.} in view of Lemma \ref{gAdmi}. This, combined with Lemma \ref{DispLem}, provides the extent to which we can get decay for solutions of \eqref{CNLS}. Interestingly, all this global analysis of the resonant system \eqref{RSS} relies heavily on using local in time Strichartz estimates on the torus in order to get global-in-time bounds for the  the $Z$ norm of the nonlinearity (see Lemma \ref{DiscreteStricTorLem}). At this place our view point is quite different form a na\"{i}ve 1d vector valued analysis (as is the case in \cite{TzVi}).

We also note that although our approach is close in spirit to recent developments in global existence for quasilinear equations \cite{GeMaSh,GeMaSh2,IoPaKG,IoPu1,IoPu2}, some of the key estimates really pertain to the low-regularity theory (see Lemma \ref{DiscreteStricTorLem} and Lemma \ref{1dBE}\footnote{This is somewhat parallel to the energy method in the quasilinear results.}).
   
\subsection*{Organization of the paper}

Section \ref{SNot} introduces the notations used in this paper. Section \ref{SStrucN} provides a decomposition of the nonlinearity as in \eqref{toy scheme}. Section \ref{SRS} introduces the resonant system \eqref{RSS} and gives some properties of its solutions. Section \ref{SMWO} shows existence of the modified wave operators and proves Theorem \ref{ExMWO}. Section \ref{SMS} shows the modified scattering statement and proves Theorem \ref{ModScatThm}. Finally, Section \ref{SAE} collects various additional estimates needed in the proofs.

\section{Notations}\label{SNot}

\subsection{Standard notations}
In this paper $\T:=\R/(2\pi\Z)$.
We will often consider functions $f:\mathbb{R}\to\mathbb{C}$ and functions $F:\mathbb{R}\times\mathbb{T}^d\to\mathbb{C}$. To distinguish between them, we use the convention that lower case letters denote functions defined on $\mathbb{R}$, capitalized letters denote functions defined on $\mathbb{R}\times\mathbb{T}^d$, and calligraphic letters denote operators, except for the Littlewood-Paley operators and dyadic numbers which are capitalized most of the time.

We define the Fourier transform on $\mathbb{R}$ by
\begin{equation*}
\widehat{g}(\xi):=\frac{1}{2\pi}\int_{\mathbb{R}}e^{-ix\xi}g(x)dx.
\end{equation*}
Similarly, if $F(x,y)$ depends on $(x,y)\in\mathbb{R}\times\mathbb{T}^d$, $\widehat{F}(\xi,y)$ denotes the partial Fourier transform in $x$. We also consider the Fourier transform of $h:\mathbb{T}^d\to\mathbb{C}$,
\begin{equation*}
h_p:=\frac{1}{(2\pi)^{d}}\int_{\mathbb{T}^d}h(y)e^{-i\langle p,y\rangle}dy,\qquad p\in\mathbb{Z}^d,
\end{equation*}
and this extends to $F(x,y)$. Finally, we also have the full (spatial) Fourier transform
\begin{equation*}
\left(\mathcal{F}F\right)(\xi,p)=\frac{1}{(2\pi)^{d}}\int_{\mathbb{T}^d}\widehat{F}(\xi,y)e^{-i\langle p,y\rangle}dy=\widehat{F}_p(\xi).
\end{equation*}

We will often use Littlewood-Paley projections. For the full frequency space, these are defined as follows:
\begin{equation*}
\begin{split}
\left(\mathcal{F}P_{\le N}F\right)(\xi,p)=\varphi(\frac{\xi}{N})\varphi(\frac{p_1}{N})\dots\varphi(\frac{p_d}{N})\left(\mathcal{F}F\right)(\xi,p),
\end{split}
\end{equation*}
where $\varphi\in C^\infty_c(\mathbb{R})$, $\varphi(x)=1$ when $\vert x\vert\le 1$ and $\varphi(x)=0$ when $\vert x\vert\ge 2$. Next, we define
%\begin{equation}\label{DefphiLP}
%\phi(x)=\varphi(x)-\varphi(2x)
%\end{equation}
%and
\begin{equation*}
P_N=P_{\le N}-P_{\le N/2},\quad P_{\ge N}=1-P_{\le N/2}.
\end{equation*}
Many times we will concentrate on the frequency in $x$ only, and we therefore define
\begin{equation*}
\begin{split}
\left(\mathcal{F}Q_{\le N}F\right)(\xi,p)=\varphi(\frac{\xi}{N})\left(\mathcal{F}F\right)(\xi,p),
\end{split}
\end{equation*}
and define $Q_N$ similarly to $P_N$. By a slight abuse of notation, we will consider $Q_N$ indifferently as an operator on functions defined on $\mathbb{R}\times\mathbb{T}^d$ and on $\mathbb{R}$.
We shall use the following commutator estimate
\begin{equation}\label{comm}
\|[Q_N,x]\|_{L^2_x\rightarrow L^2_x}\lesssim N^{-1}\,.
\end{equation}

\medskip

Below, we will need a few parameters. We fix $\delta<10^{-3}$ and\footnote{The exact value of $N$ can be significantly lowered e.g. by allowing more weights in the $S$ norm in \eqref{DefSNorm}.} $N\ge30$.
For $T\gtrsim 1$ a positive number, we let $q_T:\mathbb{R}\to\mathbb{R}$ be an arbitrary function satisfying
\begin{equation*}
0\le q_T(s)\le 1,\qquad q_T(s)=0\quad \hbox{if}\quad\vert s\vert\le T/4\quad\hbox{or}\quad \vert s\vert \ge T,\qquad\hbox{and}\qquad\int_{\mathbb{R}}\vert q_T^\prime(s)\vert ds\le 10.
\end{equation*}
Particular examples are the characteristic functions $q_T(s)=\mathfrak{1}_{[T/2,T]}(s)$, with the natural interpretation of the integral on $\R$ of $|q_T^\prime|$.

We will use the following sets corresponding to momentum and resonance level sets:
\begin{equation}\label{DiscPar}
\begin{split}
\mathcal{M}&:=\{(p,q,r,s)\in\mathbb{Z}^{4d}:\,\,p-q+r-s=0\},\\
\Gamma_\omega&:=\{(p,q,r,s)\in\mathcal{M}:\,\,\vert p\vert^2-\vert q\vert^2+\vert r\vert^2-\vert s\vert^2=\omega\}.\\
\end{split}
\end{equation}
In particular, note that $(p,q,r,s)\in\Gamma_0$ if and only if $\{p,q,r,s\}$ are the vertices of a rectangle.

%%%%%%%%%%%%%%%%%%%%%%%%%%%
\subsection{Duhamel formula}
We will prove all our statements for $t\geq 0$. By time-reversal symmetry, one obtains the analogous claims for $t\leq 0$.
%We shall consider only the case $\lambda=1$ in \eqref{CNLS}. Since we only deal with small data the case $\lambda=-1$ can be treated similarly. 
In studying solutions to \eqref{CNLS}, it will be convenient to factor out the linear flow and write a solution $U$ of \eqref{CNLS} as
\begin{equation*}
U(x,y,t)=\sum_{p\in\mathbb{Z}^d}e^{i\langle p,y\rangle}e^{-it\vert p\vert^2}(e^{it\partial_{xx}}F_p(t))(x)=
e^{it\Delta_{\mathbb{R}\times\mathbb{T}^d}}(F(t)).
\end{equation*}
We then see that $U$ solves \eqref{CNLS} if and only if $F$ solves
\begin{equation}\label{edno}
i\partial_t F(t) = e^{-it\Delta_{\mathbb{R}\times\mathbb{T}^d}}\Big( e^{it\Delta_{\mathbb{R}\times\mathbb{T}^d}}F(t)\cdot e^{-it\Delta_{\mathbb{R}\times\mathbb{T}^d}}\overline{F(t)}\cdot e^{it\Delta_{\mathbb{R}\times\mathbb{T}^d}}F(t)\Big).
\end{equation}
%which can be rewritten in Fourier space as
%\begin{multline}
%\label{Duhamel}
%i\partial_t\widehat{F}_p(\xi,t)
%=\sum_{(p,q,r,s)\in\mathcal{M}}e^{it\left[\vert q\vert^2-\vert r\vert^2+\vert s\vert^2-\vert p\vert^2\right]}
%\\
%\int_{\mathbb{R}^2}e^{it2\eta\kappa}\widehat{F}_{q}(\xi-\eta,t)\overline{\widehat{F}_{r}}(\xi-\eta-\kappa,t)\widehat{F}_{s}(\xi-\kappa,t)d\kappa d\eta\,.
%\end{multline}
We will denote the nonlinearity in \eqref{edno} by $\mathcal{N}^t[F(t),F(t),F(t)]$, where the trilinear form $\mathcal{N}^t$ is defined by 
\begin{equation*}
\mathcal{N}^t[F,G,H]:=e^{-it\Delta_{\mathbb{R}\times\mathbb{T}^d}}\Big( e^{it\Delta_{\mathbb{R}\times\mathbb{T}^d}}F\cdot e^{-it\Delta_{\mathbb{R}\times\mathbb{T}^d}}\overline{G}\cdot e^{it\Delta_{\mathbb{R}\times\mathbb{T}^d}}H\Big).
\end{equation*}
Now, we can compute the Fourier transform of the last expression which leads to the identity
\begin{equation}\label{sss}
\mathcal{F}\mathcal{N}^t[F,G,H](\xi,p)=\sum_{(p,q,r,s)\in\mathcal{M}}e^{it\left[\vert p\vert^2-\vert q\vert^2+\vert r\vert^2-\vert s\vert^2\right]}
\widehat{\mathcal{I}^t[F_q,G_r,H_s]}(\xi),
\end{equation}
where
\begin{equation}\label{DefOfI}
\mathcal{I}^t[f,g,h]:=\mathcal{U}(-t)\Big(\mathcal{U}(t)f\, \overline{\mathcal{U}(t)g}\,\mathcal{U}(t)h \Big),\quad \mathcal{U}(t)=\exp(it\partial_x^2).
\end{equation}
One verifies that
$$
\widehat{\mathcal{I}^t[f,g,h]}(\xi)=\int_{\mathbb{R}^2}e^{it2\eta\kappa}\widehat{f}(\xi-\eta)\overline{\widehat{g}}(\xi-\eta-\kappa)\widehat{h}(\xi-\kappa)d\kappa d\eta.
$$
Thus one may also write
\begin{multline*}
\mathcal{F}\mathcal{N}^t[F,G,H](\xi,p)
=
\sum_{(p,q,r,s)\in\mathcal{M}}e^{it\left[\vert p\vert^2-\vert q\vert^2+\vert r\vert^2-\vert s\vert^2\right]}
\\
\int_{\mathbb{R}^2}e^{it2\eta\kappa}\widehat{F}_{q}(\xi-\eta)\overline{\widehat{G}_{r}}(\xi-\eta-\kappa)\widehat{H}_{s}(\xi-\kappa)d\kappa d\eta\,.
\end{multline*}
According to our previous discussion, we now define  
the resonant part of the nonlinearity\footnote{$(\pi/t)\mathcal{R}$ corresponds to $\mathcal{N}_{eff}$ in \eqref{toy scheme}.} as
\begin{equation}\label{DefOfR}
\begin{split}
\mathcal{F}\mathcal{R}[F,G,H](\xi,p)&:=\sum_{(p,q,r,s)\in\Gamma_0}\widehat{F}_q(\xi)\overline{\widehat{G}_r}(\xi)\widehat{H}_s(\xi).
\end{split}
\end{equation}
We have a remarkable Leibniz rule for $\mathcal{I}^t[f,g,h]$, namely
\begin{equation}\label{leib}
Z\mathcal{I}^t[f,g,h]=\mathcal{I}^t[Zf,g,h]+\mathcal{I}^t[f,Zg,h]+\mathcal{I}^t[f,g,Zh],\quad Z\in\{ix,\partial_x\}.
\end{equation}
A similar property holds for the whole nonlinearity $\mathcal{N}^t[F,G,H]$, where $Z$ can also be a derivative in the transverse direction, $Z=\partial_{y_j}$. Property \eqref{leib} will be of importance in order to ensure the hypothesis of  the transfer principle displayed by Lemma~\ref{ControlSS+}.
%%%%%%%%%%%%%%%%%%%%%%%%%%%%%%
\subsection{Norms}
We will often consider sequences and we define the following norm on these:
\begin{equation*}
\Vert \{a_p\}\Vert_{h^s_p}^2:=\sum_{p\in\mathbb{Z}^d}\left[1+\vert p\vert^2\right]^s\vert a_p\vert^2.
\end{equation*}
For functions, we will often omit the domain of integration from the description of the norms. However, we will indicate it by a subscript $x$ (for $\mathbb{R}$), $x,y$ (for $\mathbb{R}\times\mathbb{T}^d$) or $p$ (for $\mathbb{Z}^d$). We will use mainly three different norms: a weak norm
\begin{equation*}
\Vert F\Vert_{Z}^2:=\sup_{\xi\in\mathbb{R}}\left[1+\vert \xi\vert^2\right]^2\sum_{p\in\mathbb{Z}^d}\left[1+\vert p\vert\right]^2\vert\widehat{F}_p(\xi)\vert^2=\sup_{\xi\in\mathbb{R}}\left[1+\vert \xi\vert^2\right]^2\Vert \widehat{F}_p(\xi)\Vert_{h^1_p}^2
\end{equation*}
and two strong norms
\begin{equation}\label{DefSNorm}
\begin{split}
\Vert F\Vert_{S}:=&\Vert F\Vert_{H^N_{x,y}}+\Vert xF\Vert_{L^2_{x,y}},\quad
\Vert F\Vert_{S^+}:=\Vert F\Vert_S+\Vert (1-\partial_{xx})^4F\Vert_{S}+\Vert xF\Vert_{S}.
\end{split}
\end{equation}
We have the following hierarchy
\begin{equation}\label{StrongerNorm}
\Vert F\Vert_{H^1_{x,y}}\lesssim \Vert F\Vert_Z\lesssim\Vert F\Vert_S\lesssim \Vert F\Vert_{S^+}.
\end{equation}
To verify the middle inequality, using \eqref{comm} and the elementary inequality
\begin{equation}\label{L11}
\|f\|_{L^1_x(\mathbb{R})}\lesssim \|f\|_{L^2_x(\R)}^{\frac{1}{2}}\| x  f\|_{L^2_x(\R)}^{\frac{1}{2}},
\end{equation}
one might observe that
\begin{equation*}
\begin{split}
\left[1+\vert \xi\vert^2\right]\vert \widehat{F}(\xi,p)\vert&
\lesssim\sum_N N^2\vert \widehat{Q_NF}(\xi,p)\vert\lesssim \sum_NN^2\Vert Q_NF_p\Vert_{L^2_x}^\frac{1}{2}\Vert  x  Q_NF_p\Vert_{L^2_x}^\frac{1}{2}\\
&\lesssim \sum_NN^{-\frac{1}{2}}\Vert (1-\partial_{xx})^\frac{5}{2}F_p\Vert_{L^2_x}^\frac{1}{2}\Vert \langle x\rangle F_p\Vert_{L^2_x}^\frac{1}{2}\lesssim \Vert F_p\Vert_{H^5_x}^\frac{1}{2}\Vert \langle x\rangle F_p\Vert_{L^2_x}^\frac{1}{2}
\end{split}
\end{equation*}
squaring and multiplying by $\langle p\rangle^2$, we find that (using interpolation too)
\begin{equation}\label{ZSNorm}
 \Vert F\Vert_Z\lesssim \Vert F\Vert_{L^2_{x,y}}^\frac{1}{4}\Vert F\Vert_S^\frac{3}{4}.
\end{equation}
We also remark that the operators $Q_{\le N}$, $P_{\le N}$ and the multiplication by $\varphi(\cdot/N)$ are bounded in $Z$, $S$, $S^+$, uniformly in $N$.

The space-time norms we will use are
\begin{equation}\label{XNorm}
\begin{split}
\Vert F\Vert_{X_T}:=&\sup_{0\le t\le T}\big\{\Vert F(t)\Vert_Z+(1+\vert t\vert)^{-\delta}\Vert F(t)\Vert_S +(1+\vert t\vert)^{1-3\delta}\Vert\partial_t F(t)\Vert_S\big\},\\
\Vert F\Vert_{X_T^+}:=&\Vert F\Vert_{X_T}+\sup_{0\le t\le T}\big\{(1+\vert t\vert)^{-5\delta}\Vert F(t)\Vert_{S^+}+(1+\vert t\vert)^{1-7\delta}\Vert\partial_t F(t)\Vert_{S^+}\big\}.
\end{split}
\end{equation}

In most of the cases, in order to sum-up the $1d$ estimates we make use of the following elementary bound
\begin{equation}\label{sum_in_p}
\Big\|
\sum_{(q,r,s)\,:\,(p,q,r,s)\in{\mathcal M}}
c^1_qc^2_rc^3_s\Big\|_{l^2_{p}}
\lesssim
\min_{\sigma\in\mathfrak{S}_3}\|c^{\sigma(1)}\|_{l^2_p}\|c^{\sigma(2)}\|_{l^1_p}\|c^{\sigma(3)}\|_{l^1_p}.\,
\end{equation}
As a warm up, we can prove the following simple estimates which are sufficient for the local theory.
\begin{lemma}\label{warm-up}
The following estimates hold:
\begin{equation}\label{CrudeEst}
\begin{split}
\Vert \mathcal{N}^t[F,G,H]\Vert_S&\lesssim (1+\vert t\vert)^{-1}\Vert F\Vert_S\Vert G\Vert_S\Vert H\Vert_S,\\
\Vert \mathcal{N}^t[F^a,F^b,F^c]\Vert_{S^+}&\lesssim (1+\vert t\vert)^{-1}\max_{\sigma\in\mathfrak{S}_3}\Vert F^{\sigma(a)}\Vert_{S^+}\Vert F^{\sigma(b)}\Vert_{S}\Vert F^{\sigma(c)}\Vert_{S}
\end{split}
\end{equation}
\end{lemma}
However, these estimates fall short of giving a satisfactory global theory.
\begin{proof}
Coming back to \eqref{DefOfI}, we readily obtain
\begin{equation}\label{Pla}
\Vert \mathcal{I}^t[f^a,f^b,f^c]\Vert_{L^2_x}\lesssim \min_{\sigma\in\mathfrak{S}_3}\Vert f^{\sigma(a)}\Vert_{L^2_x}\Vert e^{it\partial_{xx}}f^{\sigma(b)}\Vert_{L^\infty_x}\Vert  e^{it\partial_{xx}}f^{\sigma(c)}\Vert_{L^\infty_x}.
\end{equation}
Assume $\vert t\vert\ge 1$. We use the basic dispersive bound for the $1d$ Schr\"odinger equation and \eqref{L11} to get 
 \begin{equation}\label{basic-bound-corr}
\Vert e^{it\partial_{xx}}f\Vert_{L^\infty_x}
\lesssim\vert t\vert^{-\frac{1}{2}}\|f\|_{L^1_x}
\lesssim\vert t\vert^{-\frac{1}{2}}
\|f\|^{\frac{1}{2}}_{L^2_x}\|x f\|^{\frac{1}{2}}_{L^2_x}\,.
\end{equation}
Estimate \eqref{basic-bound-corr} allows us to write for any $\alpha>d$
\begin{equation*}
\begin{split}
\sum_{p\in\mathbb{Z}^d}\Vert e^{it\partial_{xx}}F_p\Vert_{L^\infty_x}
&\lesssim \vert t\vert^{-\frac{1}{2}}\sum_{p\in\mathbb{Z}^d}\langle p \rangle^{-\alpha}\|\langle p \rangle^{2\alpha}F_p\|^{\frac{1}{2}}_{L^2_x}\| x F_p\|^{\frac{1}{2}}_{L^2_x}\lesssim  \vert t\vert^{-\frac{1}{2}}\Vert F\Vert_S\,.
\end{split}
\end{equation*}
If $\vert t\vert\le 1$, we use Sobolev estimates instead of \eqref{basic-bound-corr} and get
\begin{equation*}
\sum_{p\in\mathbb{Z}^d}\Vert e^{it\partial_{xx}}F_p\Vert_{L^\infty_x}\lesssim \sum_{p\in\mathbb{Z}^d}\Vert F_p\Vert_{H^1_x}\lesssim \Vert F\Vert_S.
\end{equation*}
We now can come back to \eqref{Pla}: recalling \eqref{sss} and using \eqref{sum_in_p} we get the bound
\begin{equation}\label{SuffLem}
\Vert \mathcal{N}^t[F^a,F^b,F^c]\Vert_{L^2_{x,y}}\lesssim (1+\vert t\vert)^{-1}\min_{\sigma\in\mathfrak{S}_3}\Vert F^{\sigma(a)}\Vert_{L^2_{x,y}}\Vert F^{\sigma(b)}\Vert_S\Vert F^{\sigma(c)}\Vert_{S}.
\end{equation}
Now we can use Lemma~\ref{ControlSS+}.
This completes the proof of Lemma~\ref{warm-up}.
\end{proof}
%%%%%%%%%%%%%%%%
%%%%%%%%%%%%%%%%%%%%%%%%%%%%%%%%%%%%%%%%%%%%%%%%%%%%%%%%%%%%%%%%%%%%%%%%%%%%%%%%%%%%%%%%%%%%%%%%%%%%%%%%%%%%%%%%%%%%%%%%%%%%%%%%%%%%%%%%%%%%%%%%%%%%%%%%%%%%%%%%%%%%%%%%%%%%%%%%%%%%%%%%%%%%%%%%%%%%%%%%%%%%%%%%%%%%%%%%%%%%%%%%%%%%%%%%%%%%%%%%%%%%%%%%%%%%%%%%%%%%%%%%%%%%%%%%%%%%%%%%%%%%%%%%%%%%%%%%%%%%%%%%%%%%%%%%%%%%%%%%%%%%%%%%%%%%%%%%%%%
\section{Structure of the nonlinearity}\label{SStrucN}
The purpose of this section is to extract the key effective interactions from the full nonlinearity in \eqref{CNLS}. 
We first decompose the nonlinearity as 
\begin{equation}\label{DecNon0}
\mathcal{N}^t[F,G,H]=\frac{\pi}{t}\mathcal{R}[F, G,H]+\mathcal{E}^t[F,G,H]\\
\end{equation}
where $\mathcal{R}$ is given in \eqref{DefOfR}. 
Our main result is the following
\begin{proposition}\label{StrucNon}
Assume that for $T\geq 1$,  $F$, $G$, $H$: $\mathbb{R}\to S$ satisfy
\begin{equation}\label{leq1}
\Vert F\Vert_{X_{T}}+\Vert G\Vert_{X_{T}}+\Vert H\Vert_{X_{T}}\le 1.
\end{equation}
Then for $t\in [T/4,T]$, we can write
$$
\mathcal{E}^t[F(t),G(t),H(t)]=\mathcal{E}_1^t+\mathcal{E}_2^t\,,
$$
where the following bounds hold uniformly in $T\geq 1$,
\begin{equation}\label{DecNon}
\begin{split}
T^{-\delta}\Vert \int_{\mathbb{R}}q_T(t)\mathcal{E}_i(t)dt\Vert_S\lesssim 1,\quad i=1,2,\\
T^{1+\delta}\sup_{T/4\leq t\le T}\Vert \mathcal{E}_1(t)\Vert_Z\lesssim 1,\\
T^{\frac{1}{10}}\sup_{T/4\leq t\le T}\Vert \mathcal{E}_3(t)\Vert_S\lesssim 1,
\end{split}
\end{equation}
where $\mathcal{E}_2(t)=\partial_t\mathcal{E}_3(t)$.
Assuming in addition
\begin{equation}\label{BA+}
\Vert F\Vert_{X^+_{T}}+\Vert G\Vert_{X^+_{T}}+\Vert H\Vert_{X^+_{T}}\le 1,
\end{equation}
we also have that
\begin{equation}\label{AddDecNon}
T^{-5\delta}\Vert \int_{\mathbb{R}}q_T(t)\mathcal{E}_i(t)dt\Vert_{S^+}\lesssim 1,\qquad  T^{2\delta}\Vert \int_{\mathbb{R}}q_T(t)\mathcal{E}_i(t)dt\Vert_S\lesssim 1,\quad i=1,2.\\
\end{equation}
\end{proposition}
%%%%%%%%%%%%%%%%%%%%%
We will give a proof of Proposition \ref{StrucNon} at the end of this section.
It depends on various lemmas that we prove first. Among these lemmas, Lemma \ref{BilEf},
Lemma \ref{FastOsLem} and the first part of Lemma~\ref{Res} are essentially based on $L^2$ arguments, while Lemma \ref{EstimO} and the second part of Lemma \ref{Res} are based on regularity in Fourier space.
%%%%%%%%%%%%%%%%%%%%%%
\subsection{The high frequency estimates}
We start with an estimate bounding high frequencies in $x$. It uses essentially 
the bilinear Strichartz estimates on $\mathbb{R}$ (see Lemma~\ref{1dBE} and \cite{CKSTTSIMA}).
\begin{lemma}\label{BilEf}
Assume that $ T\ge 1$. The following estimates hold uniformly in $T$:
\begin{equation*}
\begin{split}
\Vert \sum_{\substack{A,B,C\\\max(A,B,C)\ge T^{\frac{1}{6}}}}\mathcal{N}^t[Q_AF,Q_BG,Q_CH]\Vert_{Z}&\lesssim T^{-\frac{7}{6}}
\Vert F\Vert_{S}\Vert G\Vert_{S}\Vert H\Vert_{S},\quad\forall t\ge T/4,\\
\Vert \sum_{\substack{A,B,C\\\max(A,B,C)\ge T^{\frac{1}{6}}}}\int_{\mathbb{R}}q_T(t)\mathcal{N}^t[Q_AF(t),Q_BG(t),Q_CH(t)]dt\Vert_S&\lesssim T^{-\frac{1}{50}}\Vert F\Vert_{X_T}\Vert G\Vert_{X_T}\Vert H\Vert_{X_T},\\
\Vert \sum_{\substack{A,B,C\\\max(A,B,C)\ge T^{\frac{1}{6}}}}\int_{\mathbb{R}}q_T(t)\mathcal{N}^t[Q_AF(t),Q_BG(t),Q_CH(t)]dt\Vert_{S^+}&\lesssim T^{-\frac{1}{50}}\Vert F\Vert_{X_T^+}\Vert G\Vert_{X_T^+}\Vert H\Vert_{X_T^+}.
\end{split}
\end{equation*}
\end{lemma}
%%%%%%%%%%%%
\begin{proof}
We start by proving the first inequality of Lemma~\ref{BilEf}. 
Fixing $t\ge T/4$ and invoking the bound \eqref{ZSNorm} and Lemma~\ref{warm-up}, we obtain that it suffices to prove the bound
\begin{equation}\label{L2}
\Vert \sum_{\substack{A,B,C\\\max(A,B,C)\ge T^{\frac{1}{6}}}}\mathcal{N}^t[Q_AF,Q_BG,Q_CH]\Vert_{L^2_{x,y}}
\lesssim T^{-\frac{5}{3}}\Vert F\Vert_{S}\Vert G\Vert_{S}\Vert H\Vert_{S}\,.
\end{equation}
Coming back to \eqref{sss} and using that $l^1_p\subset l^2_p$, we see that \eqref{L2} follows from
\begin{equation*}
\sum_{(p,q,r,s)\in\mathcal{M}}
\sum_{\substack{A,B,C\\\max(A,B,C)\ge T^{\frac{1}{6}}}}
\Vert
\mathcal{I}^t[Q_{A}F_q,Q_{B}G_r,Q_{C}H_s]\Vert_{L^2_x}\lesssim T^{-\frac{5}{3}}\Vert F\Vert_{S}\Vert G\Vert_{S}\Vert H\Vert_{S}.
\end{equation*}
Using \eqref{DefOfI} and the Sobolev embedding, we see that
\begin{equation*}
\begin{split}
\Vert \mathcal{I}^t[Q_{A}F_q,Q_{B}G_r,Q_{C}H_s]\Vert_{L^2_x}
&\lesssim (ABC)^{-11}\|Q_{A}F_q\|_{H^{12}_x}\|Q_{B}G_r\|_{H^{12}_x}\|Q_{C}H_s\|_{H^{12}_x}\\
&\lesssim (ABC)^{-11}(\langle q\rangle\langle r\rangle\langle s\rangle)^{-d-1}
\|F\|_{H^{13+d}_{x,y}}\|G\|_{H^{13+d}_{x,y}}\|H\|_{H^{13+d}_{x,y}}.\\
\end{split}
\end{equation*}
Summing, we complete the proof of  the first inequality of Lemma~\ref{BilEf}. 

Let us now turn to the proof of the two remaining estimates. We first remark that, for every $t$ and every $F,G,H\in S$ (resp. $S^+$)
\begin{equation}\label{var-lemma-warm}
\begin{split}
\Vert \sum_{\substack{A,B,C\\ \operatorname{med}(A,B,C)\ge T^{\frac{1}{6}}/16}} \mathcal{N}^t[Q_AF,Q_BG,Q_CH]\Vert_S&
\lesssim T^{-\frac{7}{6}}\Vert F\Vert_{S}\Vert G\Vert_{S}\Vert H\Vert_{S},\\
\Vert \sum_{\substack{A,B,C\\ \operatorname{med}(A,B,C)\ge T^{\frac{1}{6}}/16}}  \mathcal{N}^t[Q_AF,Q_BG,Q_CH]\Vert_{S^+}&\lesssim T^{-\frac{7}{6}}\Vert 
F\Vert_{S^+}\Vert G\Vert_{S^+}\Vert H\Vert_{S^+},
\end{split}
\end{equation}
where $\operatorname{med}(A,B,C)$ means the second largest number between $A,B,C$.
The proof of \eqref{var-lemma-warm} is slightly more delicate than the first inequality of Lemma~\ref{BilEf} because 
in aiming to apply Lemma~\ref{ControlSS+}, we are not allowed to lose derivatives on at least one of the $F$, $G$, $H$. Let $K\in L^2_{x,p}$, then we need to bound
\begin{equation}\label{expression-tris}
\begin{split}
I_K&=\langle K, \sum_{\substack{A,B,C\\ \operatorname{med}(A,B,C)\ge T^{\frac{1}{6}}/16}} \mathcal{N}^t[Q_AF,Q_BG,Q_CH]\rangle_{L^2_{x,p}\times L^2_{x,p}}\\
&\le \sum_{(p,q,r,s)\in\mathcal{M}}
\sum_{\substack{A,B,C\\
\operatorname{med}(A,B,C)\ge T^{\frac{1}{6}}/16}
}
\Big|
\int_{\R}\mathcal{U}(t)(Q_{A}F_q)\cdot\overline{\mathcal{U}(t)(Q_{B}G_r)}\cdot\mathcal{U}(t)(Q_C H_{s})\cdot\overline{\mathcal{U}(t)K_p}\Big|.
\end{split}
\end{equation}
We will show that
\begin{equation}\label{BenNov1}
I_K\lesssim T^{-\frac{5}{3}}\Vert F\Vert_{L^2_{x,y}}\Vert K\Vert_{L^2_{x,p}}\Vert G\Vert_S\Vert H\Vert_S.
\end{equation}
Similar estimates hold with $F$ replaced by $G$ or $H$. By duality and Lemma~\ref{ControlSS+}, this is sufficient to prove \eqref{var-lemma-warm}.

By performing a Littlewood-Paley decomposition of $K_p$, and using Sobolev inequality, we see from \eqref{expression-tris} that
\begin{equation}\label{expression-4}
\begin{split}
I_K&\lesssim \sum_{(p,q,r,s)\in\mathcal{M}}\sum_\ast(BC)^{-11}\|Q_{A}F_q\|_{L^2_x}\|Q_{B}G_r\|_{H^{12}_x}\|Q_C H_{s}\|_{H^{12}_x}\|Q_D K_p\|_{L^2_x},\\
\end{split}
\end{equation}
where $\sum_\ast$ denotes the sum over all dyadic integers $A,B,C,D$ such that the two highest are comparable and in addition, $\hbox{med}(A,B,C)\ge T^\frac{1}{6}$. 
Now remark that
\begin{equation*}
\begin{split}
&\sum_\ast(BC)^{-11}\|Q_{A}F_q\|_{L^2_x}\|Q_{B}G_r\|_{H^{12}_x}\|Q_C H_{s}\|_{H^{12}_x}\|Q_D K_p\|_{L^2_x}\\
&\lesssim \sum_\ast(\hbox{med}(A,B,C))^{-11}\|Q_{A}F_q\|_{L^2_x}\|Q_{B}G_r\|_{H^{12}_x}\|Q_C H_{s}\|_{H^{12}_x}\|Q_D K_p\|_{L^2_x}\\
&\lesssim T^{-\frac{5}{3}}\Vert F_q\Vert_{L^2_x}\Vert K_p\Vert_{L^2_x}\Vert G_r\Vert_{H^{12}_x}\Vert H_s\Vert_{H^{12}_x},
\end{split}
\end{equation*}
where in the last inequality, we have crudely summed over the two smallest dyadic numbers and applied the Cauchy-Schwarz inequality on the two highest. Using Cauchy-Schwarz inequality again in $p,q$, we see from \eqref{expression-4} that
\begin{equation*}
I_K\lesssim T^{-\frac{5}{3}}\Vert F\Vert_{L^2_{x,y}}\Vert K\Vert_{L^2_{x,p}}\Big(\sum_r\Vert G_r\Vert_{H^{12}_x}\Big)
\Big(\sum_s\Vert H_s\Vert_{H^{12}_x}\Big)
\end{equation*}
which yields \eqref{BenNov1} and thus \eqref{var-lemma-warm}.

\medskip

It therefore remains to prove that
\begin{equation}\label{edno-pak}
\Vert \sum_{(A,B,C)\in \Lambda}
\int_{\mathbb{R}}q_T(t)\mathcal{N}^t[Q_AF(t),Q_BG(t),Q_CH(t)]dt\Vert_S
\lesssim T^{-\frac{1}{50}}\Vert F\Vert_{X_T}\Vert G\Vert_{X_T}\Vert H\Vert_{X_T},
\end{equation}
and
\begin{equation}\label{dve}
\Vert \sum_{(A,B,C)\in \Lambda}
\int_{\mathbb{R}}q_T(t)\mathcal{N}^t[Q_AF(t),Q_BG(t),Q_CH(t)]dt\Vert_{S^+}
\lesssim T^{-\frac{1}{50}}\Vert F\Vert_{X^{+}_T}\Vert G\Vert_{X^{+}_T}\Vert H\Vert_{X^{+}_T},
\end{equation}
where $(A,B,C)\in \Lambda$ means that the $A,B,C$ summation ranges over $\hbox{med}(A,B,C)\le T^\frac{1}{6}/16$ and $\max(A,B,C)\ge T^{\frac{1}{6}}$. 
We shall only give the proof of \eqref{edno-pak}, the proof of \eqref{dve} being similar.

We consider a decomposition 
\begin{equation}\label{AddedVersion6}
[T/4,2T]=\bigcup_{j\in J} I_j,\quad I_j=[jT^\frac{9}{10},(j+1)T^\frac{9}{10}]=[t_j,t_{j+1}],\quad \#J\lesssim T^{\frac{1}{10}}
\end{equation}
and consider $\chi\in C^\infty_c(\mathbb{R})$, $\chi\geq 0$ such that $\chi(x)=0$ if $\vert x\vert\ge 2$ and
\begin{equation*}
\sum_{k\in\mathbb{Z}}\chi(x-k)\equiv 1.
\end{equation*}
The left hand-side of \eqref{edno-pak} can be estimated by $C(E_1+E_2)$, where
\begin{multline*}
E_1=\Big\| \sum_{j\in J} \sum_{(A,B,C)\in \Lambda}\int_{\mathbb{R}}q_T(t)\chi\big(\frac{t}{T^{\frac{9}{10}}}-j\big)\\
\Big(\mathcal{N}^t[Q_AF(t),Q_BG(t),Q_CH(t)]-\mathcal{N}^t[Q_AF(t_j),Q_BG(t_j),Q_CH(t_j)]\Big)dt\Big\|_S
\end{multline*}
and
\begin{equation*}
E_2= \Big\|\sum_{j\in J} \sum_{(A,B,C)\in \Lambda}
\int_{\mathbb{R}}q_T(t)\chi\big(\frac{t}{T^{\frac{9}{10}}}-j\big)\mathcal{N}^t[Q_AF(t_j),Q_BG(t_j),Q_CH(t_j)]dt\Big\|_S\,.
\end{equation*}
%%%%%%%%%%%%%%%%%

\medskip

Let us now turn to the estimate for $E_1$.
We can write
\begin{equation}\label{E1Dec}
\begin{split}
E_1\leq & \sum_{j\in J} \int_{\mathbb{R}}q_T(t)\chi\big(\frac{t}{T^{\frac{9}{10}}}-j\big)E_{1,j}(t)dt,\\
\end{split}
\end{equation}
where
\begin{equation*}
\begin{split}
E_{1,j}(t):=&\Big\|\sum_{(A,B,C)\in \Lambda}\Big(\mathcal{N}^t[Q_AF(t),Q_BG(t),Q_CH(t)]-\mathcal{N}^t[Q_AF(t_j),Q_BG(t_j),Q_CH(t_j)]\Big)\Big\|_S.
\end{split}
\end{equation*}
At this point, we remark that
\begin{equation*}
\begin{split}
&\sum_{(A,B,C)\in\Lambda}\mathcal{N}^t[Q_AF,Q_BG,Q_CH]\\
&=\mathcal{N}^t[Q_{+}F,Q_{-}G,Q_{-}H]+\mathcal{N}^t[Q_{-}F,Q_{+}G,Q_{-}H]+\mathcal{N}^t[Q_{-}F,Q_{-}G,Q_{+}H]\\
&Q_+:=Q_{\ge T^\frac{1}{6}},\qquad Q_-:=Q_{\le T^\frac{1}{6}/16}.
\end{split}
\end{equation*}
Therefore, using Lemma \ref{warm-up}, and the boundedness of $Q_\pm$ on $S$, we see that
\begin{equation}\label{EstimE1j}
\begin{split}
E_{1,j}(t)&
\le (1+\vert t\vert)^{-1}\Big[\Vert F(t)-F(t_j)\Vert_S\Vert G(t)\Vert_S\Vert H(t)\Vert_S+\Vert F(t_j)\Vert_S\Vert G(t)-G(t_j)\Vert_S\Vert H(t)\Vert_S\\
&\qquad+\Vert F(t_j)\Vert_{S}\Vert G(t_j)\Vert_{S}\Vert H(t)-H(t_j)\Vert_{S}\Big].\\
\end{split}
\end{equation}
Since $\vert t-t_j\vert\le T^{\frac{9}{10}}$, we see by definition \eqref{XNorm} that
\begin{equation*}
\Vert F(t)-F(t_j)\Vert_S\le \int_{t_j}^t\Vert \partial_t F(\sigma)\Vert_{S}d\sigma\lesssim T^{-\frac{1}{10}+3\delta}\|F\|_{X_T}.
\end{equation*}
Similar bounds hold for $G$ and $H$. Therefore, we can bound \eqref{EstimE1j} by
\begin{equation*}
E_{1,j}(t)\lesssim T^{-\frac{11}{10}+5\delta}\Vert F\Vert_{X_T}\Vert G\Vert_{X_T}\Vert H\Vert_{X_T}
\end{equation*}
which in view of \eqref{AddedVersion6} and \eqref{E1Dec} is more than enough to bound the contribution of $E_1$.

\medskip

It therefore only remains to estimate $E_2$. In this case the $A,B,C$ summation will not cause any difficulty since the bilinear Strichartz estimates will provide a decay
in terms of $(\max(A,B,C))^{-1}$. 
We have that
$$
E_2\leq  \sum_{j\in J} \sum_{(A,B,C)\in \Lambda}E_{2,j}^{A,B,C},
$$
where
\begin{equation*}
E_{2,j}^{A,B,C}= \Big\| \int_{\mathbb{R}}q_T(t)\chi\big(\frac{t}{T^{\frac{9}{10}}}-j\big)\mathcal{N}^t[Q_AF(t_j),Q_BG(t_j),Q_CH(t_j)]dt\ \Big\|_S\,.
\end{equation*}
Note that the profiles $F(t_j)$, $G(t_j)$, $H(t_j)$ are fixed. Using Lemma~\ref{ControlSS+}, it suffices to show that
\begin{equation}\label{SuffE2}
\begin{split}
&\Big\Vert \sum_{(p,q,r,s)\in\mathcal{M}}\int_{\mathbb{R}}q_T(t)\chi\big(\frac{t}{T^{\frac{9}{10}}}-j\big)
\mathcal{I}^t[Q_AF_q^a,Q_BF^b_r,Q_CF^c_s]dt\Big\Vert_{L^2_{x,p}}\\
&\lesssim (\max(A,B,C))^{-1}\min_{\sigma\in\mathfrak{S}_3}\Vert F^{\sigma(a)}\Vert_{L^2_{x,y}}\Vert F^{\sigma(b)}\Vert_S\Vert F^{\sigma(c)}\Vert_S.
\end{split}
\end{equation}
Indeed
\begin{equation*}
 \sum_{j\in J} \sum_{(A,B,C)\in \Lambda}(\max(A,B,C))^{-1}\lesssim T^{-\frac{1}{20}},\quad \Vert F(t_j)\Vert_S\Vert G(t_j)\Vert_S\Vert H(t_j)\Vert_S\le 
 T^{3\delta} \Vert F\Vert_{X_T}\Vert G\Vert_{X_T}\Vert H\Vert_{X_T}.
\end{equation*}
We proceed by duality. Let $K\in L^2_{x,p}$, we consider
\begin{equation*}
\begin{split}
I_K&=\langle K,\sum_{(p,q,r,s)\in\mathcal{M}}\int_{\mathbb{R}}q_T(t)\chi\big(\frac{t}{T^{\frac{9}{10}}}-j\big)\mathcal{I}^t[Q_AF_q^a,Q_BF^b_r,Q_CF^c_s]dt\rangle_{L^2_{x,p}\times L^2_{x,p}}\\
&= \sum_{(p,q,r,s)\in\mathcal{M}}\int_{\R^2}q_T(t)\chi\big(\frac{t}{T^{\frac{9}{10}}}-j\big)\mathcal{U}(t)(Q_A F^a_q)\cdot\overline{\mathcal{U}(t)(Q_B F^b_r)}\cdot\mathcal{U}(t)(Q_C F^c_s)\cdot\overline{\mathcal{U}(t)K_p}dxdt
\end{split}
\end{equation*}
where we may assume that $K=Q_DK$, $D\simeq \max(A,B,C)$. Using Lemma~\ref{1dBE}, we can estimate
\begin{equation*}\label{mir-bis}
I_K\le \sum_{(p,q,r,s)\in\mathcal{M}}D^{-1}\|F^a_q\|_{L^2_x}\|F^b_r\|_{L^2_x}\|F^c_s\|_{L^2_x}\Vert K_p\Vert_{L^2_x}\\
\end{equation*}
We can now use \eqref{sum_in_p} to evaluate the sum. By duality, this yields \eqref{SuffE2} and therefore \eqref{edno-pak}. As already mentioned, the proof of \eqref{dve} is similar.
This completes the proof of Lemma~\ref{BilEf}.
\end{proof}

\medskip

At this point, we introduce a first decomposition
\begin{equation}\label{Dec2}
\begin{split}
\mathcal{N}^t[F,G,H]=&\,\,\Pi^t[F,G,H]+\widetilde{\mathcal{N}}^t[F,G,H],\\
\mathcal{F}\Pi^t[F,G,H](\xi,p):=&\sum_{(p,q,r,s)\in\Gamma_0}\widehat{\mathcal{I}^t[F_q,G_r,H_s]}(\xi).
\end{split}
\end{equation}
The contribution of $\widetilde{\mathcal N}$ is treated in Subsection \ref{FastOsSec}, and that of $\Pi^t$ in Subsection \ref{Res Subsec}.
\subsection{The fast oscillations}\label{FastOsSec}
The main purpose of this subsection is to prove the following:
\begin{lemma}\label{FastOsLem}
For $T\geq 1$,  assume that $F$, $G$, $H$: $\mathbb{R} \to S$ satisfy \eqref{leq1} and
\begin{equation*}
F=Q_{\le T^{1/6}}F, \quad G=Q_{\le T^{1/6}}G, \quad H=Q_{\le T^{1/6}}H\,.
\end{equation*}
Then for $t\in [T/4,T]$, we can write
$$
\widetilde{\mathcal{N}}^t[F(t),G(t),H(t)]=\widetilde{\mathcal{E}}_1^t+\mathcal{E}_2^t,
$$
%and if we set $\widetilde{\mathcal{E}_1}(t):=\widetilde{\mathcal{E}}^t_1[F(t),G(t),H(t)]$
%and $\mathcal{E}_2(t):=\mathcal{E}_2^t[F(t),G(t),H(t)]$ then 
where it holds that, uniformly in $T\geq 1$,
\begin{equation*}
T^{1+2\delta}\sup_{T/4\leq t\le T}\Vert \widetilde{\mathcal{E}_1}(t)\Vert_S\lesssim 1,\quad
T^{1/10}\sup_{T/4\leq t\le T}\Vert \mathcal{E}_3(t)\Vert_S\lesssim 1,
\end{equation*}
where $\mathcal{E}_2(t)=\partial_t\mathcal{E}_3(t)$.
Assuming in addition that \eqref{BA+} holds we have 
\begin{equation*}
T^{1+2\delta}\sup_{T/4\leq t\le T}\Vert \widetilde{\mathcal{E}_1}(t)\Vert_{S^+}\lesssim 1,
\qquad T^{1/10}\sup_{T/4\le t\le T}\Vert \mathcal{E}_3(t)\Vert_{S^+}\lesssim 1.
\end{equation*}
\end{lemma}

\bigskip

To prove this lemma, we start by decomposing $\widetilde{\mathcal N}^t$ along the non-resonant level sets as follows:
\begin{align}
\mathcal{F}\widetilde{\mathcal{N}}^t[F,G,H](\xi,p)&=
\sum_{\omega\ne 0}\sum_{(p,q,r,s)\in\Gamma_\omega}e^{it\omega}\left(\mathcal{O}^t_1[F_q,G_r,H_s](\xi)+\mathcal{O}^t_2[F_q,G_r,H_s](\xi)\right),\label{dec tildeN0}\\
\mathcal{O}^t_1[f,g,h](\xi)&:=\int_{\mathbb{R}^2}e^{2it\eta\kappa}(1-\varphi(t^{\frac{1}{4}}\eta\kappa))\widehat{f}(\xi-\eta)\overline{\widehat{g}}(\xi-\eta-\kappa)\widehat{h}(\xi-\kappa)d\eta d\kappa,
\nonumber\\
\mathcal{O}^t_2[f,g,h](\xi)&:=\int_{\mathbb{R}^2}e^{2it\eta\kappa}\varphi(t^{\frac{1}{4}}\eta\kappa)\widehat{f}(\xi-\eta)\overline{\widehat{g}}(\xi-\eta-\kappa)\widehat{h}(\xi-\kappa)d\eta d\kappa.\nonumber
\end{align}
Essentially, on $\mathcal{O}_1$, we use the fact that the interactions are noncoherent\footnote{In the terminology of Germain-Masmoudi-Shatah \cite{GeMaSh}, $\mathcal{O}_1$ corresponds to space nonresonant interactions and $\mathcal{O}_2$ to time nonresonant interactions.}(cf Lemma \ref{EstimO}), while on $\mathcal{O}_2$, we exploit the fact that they are non resonant and we can use a normal forms transformation. 

Before we go into the proof of Lemma \ref{FastOsLem}, we insert the following remarks.

\begin{remark}\label{Y to S}
Some of our estimates below will concern functions of one real variable. To pass them on to functions on $\mathbb{R}\times\mathbb{T}^d$, we define
\begin{equation*}
\Vert f\Vert_Y:=\Vert\langle x\rangle^{\frac{9}{10}}f\Vert_{L^2_x}+\Vert f\Vert_{H^{\frac{3N}{4}}_x}
\end{equation*}
and use that
\begin{equation}\label{YSNorm}
\sum_{p\in\mathbb{Z}^d}\Vert F_p\Vert_Y\lesssim \Vert F\Vert_S.
\end{equation}
\end{remark}
%%%
\begin{remark}\label{The multiplier}
Assume that $T/4\le t\le T$. Under the assumptions of  
Lemma~\ref{FastOsLem}, the multiplier appearing in the definition of $\mathcal{O}^t_2$ in \eqref{dec tildeN0}
can be taken to be
\begin{equation*}
\widetilde{m}(\eta, \kappa):=\varphi(t^{1/4}\eta \kappa) \varphi((10T)^{-1/6}\eta)\varphi((10T)^{-1/6}\kappa).
\end{equation*}
We remark that
$$
\Vert \mathcal{F}_{\eta\kappa}\widetilde{m}\Vert_{L^1(\mathbb{R}^2)}=\Vert I(x_1,x_2)\Vert_{L^1_{x_1,x_2}},
$$
where
$$
I(x_1,x_2)=\int_{\mathbb{R}^2}e^{ix_1\eta}e^{ix_2\kappa}\varphi(S\eta\kappa)\varphi(\eta)\varphi(\kappa)d\eta d\kappa,\quad S\approx T^\frac{7}{12}\,.
$$
Then one may show that
\begin{equation*}
\vert I(x_1,x_2)\vert+\vert x_1I(x_1,x_2)\vert+\vert x_2I(x_1,x_2)\vert \lesssim 1,\quad
\vert x_1x_2I(x_1,x_2)\vert\lesssim \log(1+T)\,.
\end{equation*}
One also has rough polynomial in $T$ bounds for $(x_1^2+x_2^2+x_1^2x_2^2)|I(x_1,x_2)|$. Therefore by interpolation one obtains that
for every $\varepsilon>0$ there exists $\kappa>1$ such that
$$
|I(x_1,x_2)|\lesssim (1+T)^{\varepsilon}(1+x_1^2+x_2^2)^{-\kappa}\,. 
$$
We hence deduce that $\|\mathcal F_{\eta\kappa} \widetilde{m}\|_{L^1(\mathbb{R}^2)}\lesssim t^{\frac{\delta}{100}}$. 
Applying Lemma~\ref{CM}, we arrive at the following conclusion: if
\begin{equation*}
 f^a=Q_{\le T^\frac{1}{6}}f^a,\qquad  f^b=Q_{\le T^\frac{1}{6}}f^b, \qquad  f^c=Q_{\le T^\frac{1}{6}}f^c,
\end{equation*}
$t\ge T/4$, then
\begin{equation}\label{CEC}
\begin{split}
\Vert \mathcal{O}^t_2 [f^a,f^b,f^c]\Vert_{L^2_{\xi}}&=\Vert \mathcal{F}\mathcal{O}^t_2 [f^a,f^b,f^c]\Vert_{L^2_x}\\
&\lesssim(1+|t|)^\frac{\delta}{100}
\min_{\sigma\in\mathfrak{S}_3}\Vert f^{\sigma(a)}\Vert_{L^2_x}\Vert e^{it\partial_{xx}}f^{\sigma(b)}\Vert_{L^{\infty}_x}\Vert e^{it\partial_{xx}}f^{\sigma (c)}\Vert_{L^{\infty}_x}\\
&\lesssim (1+\vert t\vert)^{-1+\frac{\delta}{100}}\min_{\sigma\in\mathfrak{S}_3}\Vert f^{\sigma(a)}\Vert_{L^2_x}\Vert f^{\sigma(b)}\Vert_{Y}\Vert f^{\sigma(c)}\Vert_{Y}.
\end{split}
\end{equation}
A similar bound holds for $\mathcal{O}_1^t$ because $\mathcal{O}_1^t+\mathcal{O}_2^t$ enjoys a bound better than \eqref{CEC}.
\end{remark}
\begin{proof}[Proof of Lemma \ref{FastOsLem}]
In the decomposition of $\widetilde{\mathcal N}$ in \eqref{dec tildeN0}, the first sum involving $\mathcal{O}^t_1$ contributes
to $\widetilde{\mathcal{E}}_1(t)$ and its estimate follows by combining \eqref{YSNorm},  Lemma~\ref{EstimO} below with Lemma~\ref{ControlSS+}. 
Indeed, from \eqref{OSmall}, \eqref{sum_in_p} and Remark \ref{Y to S}, we get that for $t\geq T/4$
\begin{equation*}
\Vert \sum_{\omega\ne 0}\sum_{(p,q,r,s)\in\Gamma_\omega}e^{it\omega}\mathcal{O}^t_1[F^a_q,F^b_r,F^c_s]\Vert_{L^2_{\xi,p}}
\le T^{-\frac{201}{200}}\min_{\sigma\in\mathfrak{S}_3}\Vert F^{\sigma(a)}\Vert_{L^2_{x,y}}\Vert F^{\sigma(b)}\Vert_{S}\Vert F^{\sigma(c)}\Vert_{S}.
\end{equation*}
Lemma \ref{ControlSS+} (with $\theta_1=1$ and $\theta_2=0$, and $K=T^{-1-\frac{1}{200}}\lesssim T^{-1-5\delta}$) then gives the result.

\medskip

We now consider the contribution of the second sum in \eqref{dec tildeN0}. We start with a simple observation. Defining
\begin{equation*}
\widetilde{\mathcal{O}}^t_{2,\omega}[F,G,H](\xi,p)=\sum_{(p,q,r,s)\in\Gamma_\omega}\mathcal{O}^t_2[F_q,G_r,H_s](\xi),
\end{equation*}
it follows from \eqref{CEC} that, for $K\in L^2_{\xi,p}(\mathbb{R}\times \mathbb{Z}^d)$,
\begin{equation}\label{L2IP}
 \begin{split}
  &\langle K,\widetilde{\mathcal{O}}^t_{2,\omega}[F,G,H] \rangle_{L^2_{\xi,p}\times L^2_{\xi,p}}\le\sum_{(p,q,r,s)\in\Gamma_\omega}
  \left\vert\langle K_p,\mathcal{O}^t_2[F_q,G_r,H_s]\rangle_{L^2_\xi\times L^2_\xi}\right\vert\\
  &\lesssim (1+\vert t\vert)^{-1+\delta}\sum_{(p,q,r,s)\in\Gamma_\omega}\Vert K_p\Vert_{L^2_\xi}\times\\
&\quad\min\left\{\Vert F_q\Vert_{L^2_x}\Vert G_r\Vert_{Y}\Vert H_s\Vert_{Y},\Vert F_q\Vert_{Y}\Vert G_r\Vert_{L^2_x}\Vert H_s\Vert_{Y},
  \Vert F_q\Vert_{Y}\Vert G_r\Vert_{Y}\Vert H_s\Vert_{L^2_x}\right\}\\
\end{split}
\end{equation}
and summing over $\omega$, using \eqref{YSNorm} and \eqref{sum_in_p}, we get
\begin{equation}\label{BdOO}
\begin{split} 
&\Vert \sum_\omega e^{it\omega} \widetilde{\mathcal{O}}^t_{2,\omega}[F^a,F^b,F^c]\Vert_{L^2_{\xi,p}}\lesssim (1+\vert t\vert)^{-1+\delta}\min_{\sigma\in\mathfrak{S}_3}\Vert F^{\sigma(a)}\Vert_{L^2_{x,y}}\Vert F^{\sigma(b)}\Vert_S\Vert F^{\sigma(c)}\Vert_S.
\end{split}
\end{equation}

Now observe that
\begin{multline}\label{TIPP}
 e^{it\omega}\mathcal{O}^t_2[f,g,h]=\partial_t\Big( \frac{e^{it\omega}}{i\omega}\mathcal{O}^t_2[f,g,h]\Big)
-e^{it\omega}\left(\partial_t\mathcal{O}^t_2\right)[f,g,h]
\\
-e^{it\omega}\mathcal{O}^t_2[\partial_tf,g,h]
-e^{it\omega}\mathcal{O}^t_2[f,\partial_tg,h]-e^{it\omega}\mathcal{O}^t_2[f,g,\partial_th],
\end{multline}
where
$$
\left(\partial_t\mathcal{O}^t_2\right)[f,g,h](\xi):=\int_{\mathbb{R}}\partial_t\Big(e^{2it\eta\kappa}\varphi(t^{\frac{1}{4}}\eta\kappa)\Big)
\widehat{f}(\xi-\eta)\overline{\widehat{g}}(\xi-\eta-\kappa)\widehat{h}(\xi-\kappa)d\eta d\kappa.
$$
Using \eqref{BdOO}, the definition of the $X_{T}$ norm and Lemma \ref{ControlSS+}, we see that the contribution of the second line in \eqref{TIPP} is acceptable.
Similarly, since $(1+\vert t\vert)^{1/4}(\partial_t\mathcal{O}_2^t)$ satisfies similar estimates as $\mathcal{O}_2^t$, the second term in the right hand-side of 
\eqref{TIPP} is acceptable. It remains to analyze the first one. We define $\mathcal{E}_3$ by
\begin{equation*}
 \mathcal{F}\mathcal{E}_3(\xi,p):=\sum_{\omega\neq 0}\sum_{(p,q,r,s)\in\Gamma_\omega}\frac{e^{it\omega}}{i\omega}\mathcal{O}^t_2[F_q,G_r,H_s](\xi).
\end{equation*}
Using \eqref{L2IP} and \eqref{sum_in_p} we see that, for $K\in L^2_{x,y}(\mathbb{R}\times\mathbb{T}^d)$,
\begin{equation*}
 \begin{split}
\langle K,\mathcal{E}_3\rangle_{L^2_{x,y}\times L^2_{x,y}}
  &\le\sum_{\omega\ne 0}\left\vert\langle \mathcal{F}K,\widetilde{\mathcal{O}}^t_{2,\omega}[F,G,H]\rangle_{L^2\times L^2}\right\vert,\\
&\lesssim (1+\vert t\vert)^{-1+\delta}\sum_{\omega\ne 0}\sum_{(p,q,r,s)\in\Gamma_\omega}\Vert \widehat{K}_p\Vert_{L^2_\xi}\times\\
&\quad\min\left\{\Vert F_q\Vert_{L^2_x}\Vert G_r\Vert_{Y}\Vert H_s\Vert_{Y},\Vert F_q\Vert_{Y}\Vert G_r\Vert_{L^2_x}\Vert H_s\Vert_{Y},
  \Vert F_q\Vert_{Y}\Vert G_r\Vert_{Y}\Vert H_s\Vert_{L^2_x}\right\}\\
  &\lesssim (1+\vert t\vert)^{-1+\delta}\Vert K\Vert_{L^2_{x,y}}\times\\
  &\quad \min\left\{\Vert F\Vert_{L^2_{x,y}}\Vert G\Vert_{S}\Vert H\Vert_{S},\Vert F\Vert_{S}\Vert G\Vert_{L^2_{x,y}}\Vert H\Vert_{S},
  \Vert F\Vert_{S}\Vert G\Vert_{S}\Vert H\Vert_{L^2_{x,y}}\right\}.
 \end{split}
\end{equation*}
Another application of Lemma \ref{ControlSS+} shows that this term 
in the right hand-side of 
\eqref{TIPP} 
gives an acceptable contribution.
\end{proof}

We now give the argument to bound the operator $\mathcal{O}^t_1$.

\begin{lemma}\label{EstimO}
Assume that $t$, $f^a$, $f^b$, $f^c$ satisfy
\begin{equation*}
 t\ge T/4,\qquad f^a=Q_{A}f^a,\quad f^b=Q_{B}f^b,\quad f^c=Q_{C}f^c,\quad \max(A,B,C)\le T^\frac{1}{6}
\end{equation*}
then
\begin{equation}\label{OSmall}
\begin{split}
\Vert \mathcal{O}^t_1[f^a,f^b,f^c]\Vert_{L^2_\xi}&\lesssim  T^{-\frac{201}{200}}\min_{\sigma\in\mathfrak{S}_3}\Vert f^{\sigma(a)}\Vert_{L^2_x}\Vert f^{\sigma(b)}\Vert_{Y}\Vert f^{\sigma(c)}\Vert_{Y},\\
\end{split}
\end{equation}
\end{lemma}

\begin{proof}

We will show that
\begin{equation}\label{OSmall1}
\begin{split}
\Vert \mathcal{O}^t_1[f,g,h]\Vert_{L^2_\xi}&\lesssim  T^{-\frac{201}{200}}\Vert f\Vert_{L^2_x}\Vert g\Vert_{Y}\Vert h\Vert_{Y}.\\
\end{split}
\end{equation}
The other inequalities in \eqref{OSmall} follow by symmetry and conjugation.

We first decompose
\begin{equation*}
\begin{split}
g=g_c+g_f,\quad h=h_c+h_f,\quad g_c(x)=\varphi(\frac{x}{D})g(x),\quad h_c(x)=\varphi(\frac{x}{D})h(x),\quad D:=T^{\frac{7}{12}-\frac{\delta}{10}}
 \end{split}
\end{equation*}
Using the remark after \eqref{CEC}, we see that
\begin{equation*}
\Vert \mathcal{O}^t_1[f,g,h]\Vert_{L^2_\xi}\lesssim t^{\frac{\delta}{100}}\Vert f\Vert_{L^2_x}\Vert e^{it\partial_{xx}}g\Vert_{L^\infty_x}\Vert e^{it\partial_{xx}}h\Vert_{L^\infty_x}.
\end{equation*}
In addition,  for $\gamma>1/2$
\begin{equation*}
\Vert e^{it\partial_{xx}}f\Vert_{L^\infty_x}
\lesssim
\langle t\rangle^{-\frac{1}{2}}\|f\|_{L^1_x}
=
\langle t\rangle^{-\frac{1}{2}}\|\langle x\rangle ^\gamma \langle x\rangle ^{-\gamma}   f\|_{L^1_x}
\lesssim
\langle t\rangle^{-\frac{1}{2}}\|\langle x\rangle ^\gamma   f\|_{L^2_x}\,.
\end{equation*}
Hence, if  $f(x)$ is supported in $|x|>R$, 
$$
\Vert e^{it\partial_{xx}}f\Vert_{L^\infty_x}
\lesssim
\langle t\rangle^{-\frac{1}{2}}
R^{-\alpha}\|f\|_{Y},
$$
with $\alpha>\frac{2}{5}$. Therefore we obtain that \eqref{OSmall1} is a consequence of the estimate
\begin{equation*}
\Vert \mathcal{O}^t_1[f,g_c,h_c]\Vert_{L^2_\xi}\lesssim T^{-20}\Vert f\Vert_{L^2_x}\Vert g\Vert_{L^2_x}\Vert h\Vert_{L^2_x}.
\end{equation*}
But this follows by repeated integration by parts in $\kappa$ since on the support of integration, we necessarily have $\vert \eta\vert\gtrsim T^{-\frac{5}{12}}$.
This completes the proof of Lemma~\ref{EstimO}.
\end{proof}
%%%%%%%%%%%%%%%%%%%%%%%%%%%%%%%%%%%%%%%%
\subsection{The resonant level set}\label{Res Subsec}
We now turn to the contribution of the resonant set in \eqref{Dec2},
\begin{equation*}
\mathcal{F}\Pi^t[F,G,H](\xi,p)=\sum_{(p,q,r,s)\in\Gamma_0}\mathcal{F}_x\mathcal{I}^t[F_q(t),G_r(t),H_s(t)](\xi).
\end{equation*}
This term yields the main contribution in Proposition \ref{StrucNon} and in particular is responsible for the slowest $1/t$ decay. We show that it gives rise to a contribution which grows slowly in $S$, $S^+$ and that it can be well approximated by the resonant system in the $Z$ norm.

In this subsection, we will bound quantities in terms of
\begin{equation*}
\Vert F\Vert_{\tilde{Z}_t}:=\Vert F\Vert_Z+(1+\vert t\vert)^{-\delta}\Vert F\Vert_S,                                                    
\end{equation*}
so that $F(t)$ remains uniformly bounded in $\tilde{Z}_t$ under the assumption of Proposition \ref{StrucNon}.
%%%%%%
\begin{lemma}\label{Res}
Let $t\ge 1$. There holds that
\begin{equation}\label{GrowthSRes}
\Vert \Pi^t[F^a,F^b,F^c]\Vert_{S}\lesssim (1+\vert t\vert)^{-1}\sum_{\sigma\in\mathfrak{S}_3}\Vert F^{\sigma(a)}\Vert_{\tilde{Z}_t}\cdot \Vert F^{\sigma(b)}\Vert_{\tilde{Z}_t}\cdot \Vert F^{\sigma(c)}\Vert_{S}
\end{equation}
and
\begin{equation}\label{GrowthSRes2}
\begin{split}
\Vert \Pi^t[F^a,F^b,F^c]\Vert_{S^+}\lesssim& (1+\vert t\vert)^{-1}\sum_{\sigma\in\mathfrak{S}_3}\Vert F^{\sigma(a)}\Vert_{\tilde{Z}_t}\cdot \Vert F^{\sigma(b)}\Vert_{\tilde{Z}_t}\cdot \Vert F^{\sigma(c)}\Vert_{S^+}\\
&\quad+(1+\vert t\vert)^{-1+2\delta}\sum_{\sigma\in\mathfrak{S}_3}\Vert F^{\sigma(a)}\Vert_{\tilde{Z}_t}\cdot \Vert F^{\sigma(b)}\Vert_{S}\cdot \Vert F^{\sigma(c)}\Vert_{S}.
\end{split}
\end{equation}
In addition,
\begin{equation}\label{ApproxRes}
\Vert \Pi^t[F,G,H]-\frac{\pi}{t}\mathcal{R}[F,G,H]\Vert_Z\lesssim (1+\vert t\vert)^{-1-20\delta}\Vert F\Vert_S\Vert G\Vert_S\Vert H\Vert_S.
\end{equation}
and
\begin{equation}\label{ApproxRes2}
\Vert \Pi^t[F,G,H]-\frac{\pi}{t}\mathcal{R}[F,G,H]\Vert_{S}\lesssim (1+\vert t\vert)^{-1-20\delta}\Vert F\Vert_{S^+}\Vert G\Vert_{S^+}\Vert H\Vert_{S^+}.
\end{equation}
\end{lemma}
%%%%%%%%%%
\begin{remark}\label{ControlR}
 Using Lemma \ref{DiscreteStricTorLem} and Lemma \ref{ControlSS+}, we directly see that \eqref{GrowthSRes} and \eqref{GrowthSRes2}
also holds if $\Pi^t[F^a,F^b,F^c]$ is replaced by $(1+t)^{-1}\mathcal{R}[F^a,F^b,F^c]$.
\end{remark}
\begin{remark}
Note that in Lemma~\ref{Res}, the summation in $p$ is a highly non trivial part of the estimate, as opposed to the previous lemmas which were essentially concerned with functions of a real variable, and the summation in $p$ was treated in a crude way via \eqref{sum_in_p}.
\end{remark}
\begin{proof}[Proof of Lemma~\ref{Res}]
Combining \eqref{DefOfI} with Lemma \ref{DiscreteStricTorLem}, we see that
\begin{equation*}
\begin{split}
\Vert \Pi^t[F^a,F^b,F^c]\Vert_{L^2_{x,y}}&\le\Vert \sum_{(p,q,r,s)\in\Gamma_0}\vert e^{it\partial_{xx}}F^a_q(x)\vert\cdot\vert e^{-it\partial_{xx}}F^b_r(x)\vert\cdot\vert e^{it\partial_{xx}}F^c_s(x)\vert\,\,\Vert_{l^2_p\, L^2_x}\\
&\lesssim \min_{j\in\{a,b,c\}}
\Big\|
\Vert e^{it\partial_{xx}}F^{j}_p(x)\Vert_{l^2_{p}}\prod_{k\ne j}\Big[\sum_{p\in\mathbb{Z}^d}\Big[1+\vert p\vert^2\Big]\vert e^{it\partial_{xx}}F^{k}_p(x)\vert^2 \Big]^\frac{1}{2}
\Big\|_{L^2_x}
\end{split}
\end{equation*}
and therefore
\begin{equation*}
\Vert \Pi^t[F^a,F^b,F^c]\Vert_{L^2_{x,y}}\lesssim \min_{j\in\{a,b,c\}}\Vert F^{j}\Vert_{L^2_{x,y}}\prod_{k\ne j}\Big[\sup_{x\in \R}\sum_{p\in\mathbb{Z}^d}\left[1+\vert p\vert^2\right]\vert e^{it\partial_{xx}}F^{k}_p(x)\vert^2 \Big]^\frac{1}{2}.
\end{equation*}
Using Lemma \ref{DispLem}, we can conclude that
\begin{equation}\label{GrowthSRes111}
\Vert \Pi^t[F^a,F^b,F^c]\Vert_{L^2_{x,y}}\lesssim (1+\vert t\vert)^{-1} \min_{j\in\{a,b,c\}}\Vert F^{j}\Vert_{L^2_{x,y}}\prod_{k\ne j}\Vert F^{k}\Vert_{\tilde{Z}_t}.
\end{equation}
Using Lemma \ref{ControlSS+}, we obtain  \eqref{GrowthSRes}. In order to show \eqref{GrowthSRes2}, we will apply the second part of Lemma \ref{ControlSS+}. For this, it suffices to prove that
\begin{equation}\label{SuffGrowthRes2}
\Vert xF\Vert_Z\lesssim  T^{-\delta}\Vert F\Vert_{S^+}+T^{2\delta}\Vert F\Vert_S.
\end{equation}
Indeed, one first observes that it suffices to prove \eqref{SuffGrowthRes2} for functions independent of $y$. 
Then, we notice that
$$
\sup_\xi [(1+|\xi|^2) |{\mathcal F} (x f)|]
\sim \sup_M (1+M^2)\|{\mathcal F} Q_M (x f)\|_{L^\infty_\xi}\,.
$$
Next, for every $M, R$ we get
$$
\|{\mathcal F} Q_M [x (1-\varphi(x/R)) f]\|_{L^\infty_\xi}\lesssim
\| [x (1-\varphi(x/R)) f]\|_{L^1_x}
\lesssim R^{-\frac 12} \|x^2f \|_{L^2}\leq C R^{-\frac 12} \|f \|_{S^+}
$$
On the other hand, by invoking  \eqref{L11}, we get
\begin{equation}\label{pisna-mi}
\|{\mathcal F} Q_M [(x \varphi(x/R)) f]\|_{L^\infty_\xi}
\lesssim 
\|Q_{M}( x (\varphi(x/R)) f)\|_{L^2_x}^{\frac 12}   \|x Q_M [x(\varphi(x/R)) f]\|_{L^2_x}^{\frac 12} 
\end{equation}
We now estimate each factor at the right hand-side of the last inequality. 
By setting $\tilde{\varphi}(x)=x\varphi(x)$, we may write for $M$ a dyadic integer
$$
\|Q_{M}( x (\varphi(x/R)) f)\|_{L^2_x}=R \|Q_{M}(  (\tilde{\varphi}(x/R)) f)\|_{L^2_x}\lesssim RM^{-N}\|f\|_{H^N_x}.
$$
We next estimate the second factor at the right hand-side of \eqref{pisna-mi} as follows
$$
 \|x Q_M [x(\varphi(x/R)) f]\|_{L^2_x}\lesssim
 \|\langle x\rangle^2 \varphi(x/R)) f\|_{L^2_x}\lesssim R\|f\|_{S}.
$$
We conclude the proof of \eqref{SuffGrowthRes2} by choosing $R= T^{2\delta} (1+M^2)^2$.  

We now turn to the proof of \eqref{ApproxRes} and \eqref{ApproxRes2}. First decompose
$$
F=F_c+F_f,\quad G=G_c+G_f,\quad H=H_c+H_f,
$$
where
\begin{equation*}
F_c(x,y)=\varphi(t^{-\frac{1}{4}}x)F(x,y),\quad G_c(x,y)=\varphi(t^{-\frac{1}{4}}x)G(x,y),\quad H_c(x,y)=\varphi(t^{-\frac{1}{4}}x)H(x,y).
\end{equation*}
We claim that
\begin{equation}\label{GrowthSRes1112}
\begin{split}
&\Vert \Pi^t[F,G,H]-\Pi^t[F_c,G_c,H_c]\Vert_{Z}+\frac{1}{t}\Vert \mathcal{R}[F,G,H]-\mathcal{R}[F_c,G_c,H_c]\Vert_{Z}\\
\lesssim&(1+\vert t\vert)^{-\frac{21}{20}}\Vert F\Vert_S\Vert G\Vert_S\Vert H\Vert_S,\\
\end{split}
\end{equation}
and
\begin{equation}\label{GrowthSRes1113}
\begin{split}
&\Vert \Pi^t[F,G,H]-\Pi^t[F_c,G_c,H_c]\Vert_{S}+\frac{1}{t}\Vert \mathcal{R}[F,G,H]-\mathcal{R}[F_c,G_c,H_c]\Vert_{S}\\
\lesssim& (1+\vert t\vert)^{-\frac{21}{20}}\Vert F\Vert_{S^+}\Vert G\Vert_{S^+}\Vert H\Vert_{S^+}.
\end{split}
\end{equation}

Indeed, with $\tilde G$ denoting either $G_c$ or $G_f$ (and similarly for $\tilde{H}$) and using \eqref{GrowthSRes} and \eqref{GrowthSRes111}, we obtain that
\begin{equation*}
 \begin{split}
\Vert \frac{\pi}{t}\mathcal{R}[F_f,\tilde{G},\tilde{H}]\Vert_S+\Vert \Pi^t[F_f,\tilde{G},\tilde{H}]\Vert_{S}&\lesssim (1+\vert t\vert)^{-1}\Vert F\Vert_{S}\Vert G\Vert_{S}\Vert H\Vert_{S}\\
\Vert \frac{\pi}{t}\mathcal{R}[F_f,\tilde{G},\tilde{H}]\Vert_{L^2_{x,y}}+\Vert \Pi^t[F_f,\tilde{G},\tilde{H}]\Vert_{L^2_{x,y}}&\lesssim (1+\vert t\vert)^{-1}\Vert F_f\Vert_{L^2}\Vert \tilde{G}\Vert_S\Vert \tilde{H}\Vert_S\\
&\lesssim (1+\vert t\vert)^{-5/4}\Vert F\Vert_{S}\Vert G\Vert_{S}\Vert H\Vert_{S}
 \end{split}
\end{equation*}
Using \eqref{ZSNorm} allows to bound the contribution of this term to \eqref{GrowthSRes1112}. The terms involving $G_f$ and $H_f$ can be treated similarly.

Similarly, using \eqref{GrowthSRes}, we see that
\begin{equation*}
 \begin{split}
  \Vert \frac{\pi}{t}\mathcal{R}[F_f,\tilde{G},\tilde{H}]\Vert_S+\Vert \Pi^t[F_f,\tilde{G},\tilde{H}]\Vert_S
&\lesssim (1+\vert t\vert)^{-1}\Vert F_f\Vert_{S}\Vert G\Vert_{S}\Vert H\Vert_{S}\\
&\lesssim (1+\vert t\vert)^{-5/4}\Vert F\Vert_{S^+}\Vert G\Vert_{S^+}\Vert H\Vert_{S^+}.
 \end{split}
\end{equation*}
This bounds the contribution of terms involving $F_f$ to the right hand side of \eqref{GrowthSRes1112}. The contribution of terms involving $H_f$ or $G_f$ follows similarly.

Therefore, to show \eqref{ApproxRes} and \eqref{ApproxRes2}, it suffices to show that
\begin{equation}\label{CS1}
\Vert \Pi^t[F_c,G_c,H_c]-\frac{\pi}{t}\mathcal{R}[F_c,G_c,H_c]\Vert_{Z}\lesssim (1+\vert t\vert)^{-\frac{15}{14}}\Vert F\Vert_S\Vert G\Vert_S\Vert H\Vert_S
\end{equation}
and
\begin{equation}\label{CS2}
\Vert \Pi^t[F_c,G_c,H_c]-\frac{\pi}{t}\mathcal{R}[F_c,G_c,H_c]\Vert_{S}\lesssim (1+\vert t\vert)^{-\frac{15}{14}}\Vert F\Vert_{S^+}\Vert G\Vert_{S^+}\Vert H\Vert_{S^+}.
\end{equation}
%But now this follows from \eqref{sum_in_p}, Lemma \ref{ControlSS+} and the following:
The proof of \eqref{CS1} and \eqref{CS2} will follow from the following key statement. 
\begin{lemma}\label{SPL}
Assume that
\begin{equation}\label{fghCompSup}
f(x)=\varphi(s^{-\frac{1}{4}}x)f(x),\qquad g(x)=\varphi(s^{-\frac{1}{4}}x)g(x),\qquad h(x)=\varphi(s^{-\frac{1}{4}}x)h(x)
\end{equation}
and that $s\ge 1$. There holds that
\begin{equation}\label{Claim1}
\begin{split}
&\left\vert \int_{\mathbb{R}^2}e^{i2s\eta\kappa}\widehat{f}(\xi-\eta)\overline{\widehat{g}}(\xi-\eta-\kappa)\widehat{h}(\xi-\kappa) d\eta d\kappa-\frac{\pi}{s}\widehat{f}(\xi)\overline{\widehat{g}}(\xi)\widehat{h}(\xi)\right\vert\\
&\lesssim s^{-\frac{11}{10}}\Vert  f\Vert_{L^2_x}\Vert  g\Vert_{L^2_x}\Vert  h\Vert_{L^2_x}.
\end{split}
\end{equation}
In fact, for $\theta$ an integer,
\begin{equation}\label{Claim2}
\begin{split}
&\vert \xi\vert^{\theta}\left\vert \int_{\mathbb{R}^2}e^{i2s\eta\kappa}\widehat{f^a}(\xi-\eta)\overline{\widehat{f^b}}(\xi-\eta-\kappa)\widehat{f^c}(\xi-\kappa) d\eta d\kappa-\frac{\pi}{s}\widehat{f^a}(\xi)\overline{\widehat{f^b}}(\xi)\widehat{f^c}(\xi)\right\vert\\
&\lesssim s^{-\frac{11}{10}}\min_{\sigma\in\mathfrak{S}_3}\Vert  f^{\sigma(a)}\Vert_{H^\theta_x}\Vert  f^{\sigma(b)}\Vert_{L^2_x}\Vert  f^{\sigma(c)}\Vert_{L^2_x}.
\end{split}
\end{equation}
\end{lemma}

\begin{proof}[Proof of Lemma \ref{SPL}] We may rewrite (here we follow the computations in \cite{KaPu})
\begin{equation*}
\begin{split}
&\int_{\mathbb{R}^2}e^{i2s\eta\kappa}\widehat{f}(\xi-\eta)\overline{\widehat{g}}(\xi-\eta-\kappa)\widehat{h}(\xi-\kappa) d\eta d\kappa\\
&=\int_{\mathbb{R}^3}f(y_1)\overline{g}(y_2)h(y_3)\int_{\mathbb{R}^2}e^{i\left[2s\eta\kappa-y_1(\xi-\eta)-y_2(\eta+\kappa-\xi)-y_3(\xi-\kappa)\right]} d\eta d\kappa dy_1dy_2dy_3\\
&=\frac{1}{2s}\int_{\mathbb{R}^3}f(y_1)\overline{g}(y_2)h(y_3)e^{-i\xi(y_1-y_2+y_3)}e^{-i\frac{y_1-y_2}{\sqrt{2s}}\frac{y_3-y_2}{\sqrt{2s}}}\left\{\int_{\mathbb{R}^2}e^{i\left[\eta+\frac{y_3-y_2}{\sqrt{2s}}\right]\cdot \left[\kappa+\frac{y_1-y_2}{\sqrt{2s}}\right]}d\eta d\kappa\right\} dy_1dy_2dy_3\\
&=\frac{\pi}{s}\int_{\mathbb{R}^3}f(y_1)\overline{g}(y_2)h(y_3)e^{-i\xi(y_1-y_2+y_3)}e^{-i\frac{y_1-y_2}{\sqrt{2s}}\frac{y_3-y_2}{\sqrt{2s}}}dy_1dy_2dy_3.
\end{split}
\end{equation*}
Therefore, for $\xi\in\mathbb{R}$,
\begin{equation*}
\begin{split}
&\left\vert \int_{\mathbb{R}^2}e^{i2s\eta\kappa}\widehat{f}(\xi-\eta)\overline{\widehat{g}}(\xi-\eta-\kappa)\widehat{h}(\xi-\kappa) d\eta d\kappa-\frac{\pi}{s}\widehat{f}(\xi)\overline{\widehat{g}}(\xi)\widehat{h}(\xi)\right\vert\\
&\le \left\vert \frac{\pi}{s}\int_{\mathbb{R}^3}f(y_1)\overline{g}(y_2)h(y_3)e^{-i\xi(y_1-y_2+y_3)}\left\{e^{-i\frac{y_1-y_2}{\sqrt{2s}}\frac{y_3-y_2}{\sqrt{2s}}}-1\right\}dy_1dy_2dy_3\right\vert\\
&\lesssim s^{-\frac{11}{10}}\Vert  f\Vert_{L^2_x}\Vert  g\Vert_{L^2_x}\Vert h\Vert_{L^2_x}.
\end{split}
\end{equation*}
This concludes the proof of \eqref{Claim1}.

Now, \eqref{Claim2} follows from \eqref{Claim1} and the fact that
\begin{equation*}
\begin{split}
\left\vert \int_{\mathbb{R}^2}e^{i2s\eta\kappa}\widehat{f^a}(\xi-\eta)\overline{\widehat{f^b}}(\xi-\eta-\kappa)\widehat{f^c}(\xi-\kappa)(\kappa^\alpha\eta^\beta) d\eta d\kappa\right\vert\\
\lesssim s^{-\frac{3}{4}(\alpha+\beta)}\min_{\sigma\in\mathfrak{S}_3}\Vert \widehat{f}^{\sigma(a)}\Vert_{L^1_x}\Vert \widehat{f}^{\sigma(b)}\Vert_{L^2_x}\Vert \widehat{f}^{\sigma(c)}\Vert_{L^2_x}
\end{split}
\end{equation*}
which is readily verified upon integrating by parts in $\eta$ and $\kappa$.
This completes the proof of Lemma~\ref{SPL}.
\end{proof}
The proof of \eqref{CS1}  follows from Lemma~\ref{SPL} and \eqref{sum_in_p}.
Using once again Lemma~\ref{SPL} and \eqref{sum_in_p} one directly estimates the $L^2_{x,y}$ contribution to the $S$ norm in the left hand-side of \eqref{CS2}.
Using in addition a Leibniz rule one estimates the $\|xF\|_{L^2_{x,y}}$ and the $\|\partial_y^N F\|_{L^2_{x,y}}$ 
contributions to the $S$ norm in the left hand-side of \eqref{CS2} by a use of Lemma~\ref{SPL} and \eqref{sum_in_p} . 
Finally, the $\|\partial_x^N F\|_{L^2_{x,y}}$ contribution to the $S$ norm in the left hand-side of \eqref{CS2} can be evaluated as follows. 
Let us first explain how we evaluate the first derivative. To simplify notations, let us set 
$$
T[F_c,G_c,H_c]=\Pi^t[F_c,G_c,H_c]-\frac{\pi}{t}\mathcal{R}[F_c,G_c,H_c].
$$
Then
\begin{multline*}
\partial_x T[F_c,G_c,H_c]=T[(\partial_x F)_c,G_c,H_c]+T[F_c,(\partial_x G)_c,H_c]+T[F_c,G_c,(\partial_x H)_c]
\\
+
t^{-\frac{1}{4}}
\Big(
T[\tilde{F}_c,G_c,H_c]+T[F_c,\tilde{G}_c,H_c]+T[F_c,G_c,\tilde{H}_c]
\Big),
\end{multline*}
where $\tilde{F}_c=\varphi'(t^{-\frac{1}{4}}x)F$ and similarly for $\tilde{G}_c$ and $\tilde{H}_c$. 
We are now in position to apply Lemma~\ref{SPL} and \eqref{sum_in_p} to estimate the first $x$ derivative contribution to the $S$ norm in the left hand-side of
\eqref{CS2}. The estimates for higher order derivatives can be performed inductively. This completes the proof of Lemma~\ref{Res}.
\end{proof}
%%%%%%%%%%%%%%%%
Finally we can give the proof of Proposition \ref{StrucNon}.
\begin{proof}[Proof of Proposition \ref{StrucNon}]
For $t\in[T/4,T]$, we may decompose
\begin{equation*}
\begin{split}
\mathcal{N}^t=&\sum_{\substack{A,B,C\\\max(A,B,C)\ge T^{\frac{1}{6}}}}\mathcal{N}^t[Q_AF(t),Q_BG(t),Q_CH(t)]\\
&+\widetilde{\mathcal N}^t[Q_{\leq T^{\frac{1}{6}}}F(t),Q_{\leq T^{\frac{1}{6}}}G(t),Q_{\leq T^{\frac{1}{6}}}H(t)]+
\Pi^t[Q_{\leq T^{\frac{1}{6}}}F(t),Q_{\leq T^{\frac{1}{6}}}G(t),Q_{\leq T^{\frac{1}{6}}}H(t)]\,.
\end{split}
\end{equation*}
The first term above contributes to $\mathcal E_1$ by Lemma \ref{BilEf}. The second term contains $\mathcal E_2$ as it can be written by Lemma \ref{FastOsLem} as $\widetilde{\mathcal E}_1+\mathcal E_2$ with $\widetilde{\mathcal E}_1$ giving an acceptable contribution to $\mathcal E_1$. The third term can be written as
\begin{align*}
&\Pi^t[Q_{\leq T^{\frac{1}{6}}}F(t),Q_{\leq T^{\frac{1}{6}}}G(t),Q_{\leq T^{\frac{1}{6}}}H(t)]=\frac{\pi}{t}\mathcal R[F(t),G(t),H(t)]\\
&\qquad+\Big(\Pi^t[Q_{\leq T^{\frac{1}{6}}}F(t),Q_{\leq T^{\frac{1}{6}}}G(t),Q_{\leq T^{\frac{1}{6}}}H(t)]-\frac{\pi}{t}\mathcal{R}[Q_{\leq T^{\frac{1}{6}}}F(t),Q_{\leq T^{\frac{1}{6}}}G(t),Q_{\leq T^{\frac{1}{6}}}H(t)]\Big)\\
&\qquad-\frac{\pi}{t}\sum_{\substack{A,B,C\\\max(A,B,C)\geq T^{\frac{1}{6}}}}\mathcal{R}[Q_AF(t),Q_BG(t),Q_CH(t)]
\end{align*} 
The second term on the right hand-side contributes to $\mathcal{E}_1$ as per Lemma \ref{Res}.
The third term on the right hand-side also contributes to $\mathcal{E}_1$.
Indeed, one needs to invoke Remark~\ref{ControlR} and to observe that, similarly to above, 
the summations over $A$, $B$, $C$ factorizes properly by using the projectors 
$Q_{\leq T^{\frac{1}{6}}}$ (at least one of the factors is localized at $x$ frequencies $\gtrsim T^{\frac{1}{6}}$ and thus the passage from $S$ to $S^{+}$ gains a decay in $T$). 
This finishes the proof of Proposition \ref{StrucNon}.
\end{proof}
%%%%%%%%%%%%%%%%%%%%%%%%%%%%%%%%%%%%%%%%%%%%%%%%%%%%%%%%%%%%%%%%%%%%%%%%%%%%%%%%%%%%%%%%%%%%%%%%%%%%%%%%%%%%%%%%%%%%%%%%%%%%%%%%%%%%%%%%%%%%%%%%%%%%%%%%%%%%%%%%%%%%%%%%%%%%%%%%%%%%%%%%%%%%%%%%%%%%%%%%%%%%%%%%%%%%%%%%%%%%%%%%%%%%%%%%%%%%%%%%%%%%%%%%%%%%%%%%%%%%%%%%%%%%%%%%%%%%%%%%%%%%%%%%%%%%%%%%%%%%%%%%%%%%%%%%%%%%%%%%%%%%%%%%%%%%%%%%%%%
\section{The resonant system}\label{SRS}
Here we review some useful facts about the resonant system which will be at the heart of the asymptotic analysis of \eqref{CNLS}. The {\it resonant system} is defined for a vector $a=\{a_p\}_{p\in\mathbb{Z}^d}$
as\footnote{Of course, $R$ is very much related to $\mathcal{R}$ defined in \eqref{RSS} and properties of $R$ will directly imply similar properties for $\mathcal{R}$.}
\begin{equation}\label{RS}
i\partial_ta_p(t)=\sum_{(p,q,r,s)\in\Gamma_0}a_{q}(t)\overline{a_{r}(t)}a_{s}(t)=:R[a(t),a(t),a(t)]_p.
\end{equation}
This is a Hamiltonian system for the symplectic form
\begin{equation*}
\Omega(\{a_p\},\{b_q\})={\rm Im}\Big[\sum_{p\in\mathbb{Z}^d}\overline{a_p}b_p\Big]={\rm Re}\langle -i\{a_p\},\{b_p\}\rangle_{l^2_p\times l^2_p}
\end{equation*}
and Hamiltonian
\begin{equation}\label{DefHam}
\begin{split}
\mathcal{H}(a):=\langle R(a,a,a),a\rangle_{l^2_p\times l^2_p}&=\sum_{(p,q,r,s)\in\Gamma_0}a_p\overline{a_q}a_r\overline{a_s}=
\sum_{\lambda\in\mathbb{Z}}\sum_{\mu\in\mathbb{Z}^d}\Big\vert\sum_{\substack{p-q=\mu,\\\vert p\vert^2-\vert q\vert^2=\lambda}}a_p\overline{a}_q\Big\vert^2\\
&=\Vert e^{is\Delta_{\mathbb{T}^d}}\mathcal{F}^{-1}_y a\Vert_{L^4_{y,s}(\mathbb{T}^d\times [0,2\pi])}^4.\\
\end{split}
\end{equation}
In addition, for any function $g$, we write
\begin{equation*}
\begin{split}
\frac{d}{dt}\sum_{p\in\mathbb{Z}^d}g(p)a_p\overline{a_p}&=2\sum_{p\in\mathbb{Z}^d}g(p)\hbox{Re}\left\{\overline{a_p}\partial_ta_p\right\}\\
&=-i\sum_{(p,q,r,s)\in\Gamma_0}\left\{g(p)\overline{a_{p}}a_{q}\overline{a_{r}}a_{s}-g(p)a_{p}\overline{a_{q}}a_{r}\overline{a_{s}}\right\}
\end{split}
\end{equation*}
and using symmetry, this becomes
\begin{equation*}
\frac{d}{dt}\sum_{p\in\mathbb{Z}^d}g(p)a_p\overline{a_p}=-\frac{i}{2}\sum_{\substack{p+r=q+s\\\vert p\vert^2+\vert r\vert^2=\vert q\vert^2+\vert s\vert^2}}\left[g(p)+g(r)-g(q)-g(s)\right]\overline{a_{p}}a_{q}\overline{a_{r}}a_{s}.
\end{equation*}
Hence, upon taking $g(p)\equiv1$, $g(p)=p$, $g(p)=\vert p\vert^2$, we see that we have conservation of the mass, momentum and energy
\begin{equation}\label{ConsRS}
\begin{split}
\hbox{mass}( a)&=\sum_{p\in\mathbb{Z}^d}\vert a_p\vert^2,\qquad
\hbox{mom}( a)=\sum_{p\in\mathbb{Z}^d}p\vert a_p\vert^2,\qquad\hbox{energy}(a)=\sum_{p\in\mathbb{Z}^d}\vert p\vert^2\vert a_p\vert^2.
\end{split}
\end{equation}
Another way to recover the first and last of these formulas is to see that $R[ a, a,\cdot]$ is a self-adjoint operator on $l^2_p$ and that
\begin{equation}\label{H1NormPres}
\langle iR[ a, a, a], a\rangle_{h^1_p\times h^1_p}=0
\end{equation}
for all $a\in h^1_p$.

\medskip

A first simple remark is that the resonant system is well defined for initial data in $h^1_p$:
\begin{lemma}\label{GlobSol}
Let $1\le d\le 4$. For any $ a(0)\in h^1_p$, there exists a unique global solution $u\in C^1(\mathbb{R}:h^1_p)$ of \eqref{RS}.
In addition, higher regularity is preserved in the sense that if $a(0)\in h^s_p$, then the solution belongs to $C^1(\mathbb{R}:h^s_p)$.
\end{lemma}

Note that this is the reason for our restriction to $d\le 4$ in Theorem \ref{ModScatThm} and Theorem \ref{ExMWO}. When $d\ge 5$, the flow map of \eqref{RS} cannot even be $C^3$ in $h^1_p$ in any neighborhood of $0$.

\begin{proof}
From Lemma \ref{DiscreteStricTorLem}, we see that the mapping $ a\mapsto R[ a, a, a]$ is locally Lipschitz in $h^1_p$, uniformly on bounded subset.
A contraction mapping argument gives local well-posedness in $h^s_p$ for any $s\geq 1$ which is extended to a global statement in $h^1_p$ by \eqref{ConsRS}. The preservation of higher regularity is classical.
\end{proof}

\begin{remark}\label{RemScal}
Small data do not make a difference: using to the symmetry $(a_n(t))\to (\lambda a_n(\lambda^2 t))$ enjoyed by \eqref{RS} we can normalize the initial data to any pre-assigned size $\delta$ in $h^s_p$. In addition, by a complex conjugation one can pass from the ``focusing'' to the ``defocusing'' resonant system.
\end{remark}

\subsection{Estimation of solutions to the resonant system}

\begin{lemma}\label{gAdmi}

$i)$ Assume that\footnote{Here $S^{(+)}$ denotes either $S$ or $S^+$.} $G_0\in S^{(+)}$ and that $G$ evolves according to \eqref{RSS}. Then, there holds that, for $t\ge 1$,
\begin{equation}\label{SolRSBdd}
\begin{split}
\Vert G(\ln t)\Vert_Z&=\Vert G_0\Vert_Z\\
\Vert G(\ln t)\Vert_{S^{(+)}}&\lesssim (1+\vert t\vert)^{\delta^\prime}\Vert G_0\Vert_{S^{(+)}}.
\end{split}
\end{equation}
Besides, we may choose $\delta^\prime\lesssim \Vert G_0\Vert_Z^2$.

$ii)$ In addition, we have the following uniform continuity result:
if $A$ and $B$ solve \eqref{RSS} and satisfy
\begin{equation*}
\sup_{0\le t\le T}\left\{\Vert A(t)\Vert_{Z}+\Vert B(t)\Vert_{Z}\right\}\le \theta
\end{equation*}
and
\begin{equation*}
\Vert A(0)-B(0)\Vert_{S^{(+)}}\le \delta
\end{equation*}
then, there holds that, for $0\le t\le T$,
\begin{equation}\label{StabRSS}
\Vert A(t)-B(t)\Vert_{S^{(+)}}\le \delta e^{C\theta^2 t}. 
\end{equation}
\end{lemma}
%%%%%%
\begin{proof}[Proof of Lemma \ref{gAdmi}]
The first equality in \eqref{SolRSBdd} follows from \eqref{ConsRS}.
%In fact, we have that, for any $\xi,\tau\in\mathbb{R}$, there holds that
%\begin{equation*}
%\left[1+\vert\xi\vert^2\right]\Vert {\widehat G}_p(\xi,\tau)\Vert_{h^1_p}=\left[1+\vert\xi\vert^2\right]\Vert {\widehat G}_{0, p}(\xi)\Vert_{h^1_p}\le\sup_{\xi\in\mathbb{R}}\left[1+\vert\xi\vert^2\right]\Vert {\widehat G}_{0,p}(\xi)\Vert_{h^1_p}=\Vert G_0\Vert_Z
%\end{equation*}
 For the second, we simply use \eqref{StriEst} and \eqref{StriEst2} to show that, for $\sigma\ge 0$ and fixed $\xi$,
\begin{equation}\label{Trivial111}
 \begin{split}
 \Vert \mathcal{F}\mathcal{R}[G,G,G](\xi)\Vert_{h^\sigma_p}\lesssim \Vert G\Vert_Z^2\Vert \widehat{G}(\xi)\Vert_{h^\sigma_p}\\
 \Vert \partial_\xi\mathcal{F}\mathcal{R}[G,G,G](\xi)\Vert_{l^2_p}\lesssim \Vert G\Vert_Z^2\Vert \partial_\xi \widehat{G}(\xi)\Vert_{l^2_p}.\\
 \end{split}
\end{equation}
An application of Gronwall inequality yields the statement about the $S$ norm in \eqref{SolRSBdd}. For the $S^+$ norm, we use again \eqref{StriEst} and \eqref{StriEst2} to get
\begin{equation*}
 \begin{split}
  \Vert \partial_\xi\mathcal{F}\mathcal{R}[G,G,G](\xi)\Vert_{h^\sigma_p}
&\lesssim \Vert G\Vert_Z^2\Vert \partial_\xi \widehat{G}(\xi)\Vert_{h^\sigma_p}+\Vert G\Vert_Z\Vert\partial_\xi\widehat{G}\Vert_{h^1_p}\Vert \widehat{G}\Vert_{h^\sigma_p},\\
 \Vert \partial_\xi^2\mathcal{F}\mathcal{R}[G,G,G](\xi)\Vert_{l^2_p}
&\lesssim \Vert G\Vert_Z^2\Vert \partial_\xi^2 \widehat{G}(\xi)\Vert_{l^2_p}+\Vert G\Vert_Z\Vert\partial_\xi\widehat{G}\Vert_{h^1_p}\Vert \partial_\xi\widehat{G}\Vert_{h^1_p},\\
\end{split}                                                                                                                                          
\end{equation*}
Bounding first the case $\sigma=1$ and applying inhomogeneous Gronwall estimates, we obtain the bound on the $S^+$ norm in \eqref{SolRSBdd}.

\medskip

The proof of  \eqref{StabRSS} is similar, based on the fact that
\begin{equation*}
\begin{split}
\partial_\tau\left\{\widehat A_p(\xi)-\widehat B_p(\xi)\right\}=&i\left\{{\mathcal R}[\widehat A(\xi),\widehat A(\xi), \widehat A(\xi)]_p
-{\mathcal R}[\widehat B(\xi),\widehat B(\xi),\widehat B(\xi)]_p\right\}\\
=&i{\mathcal R}[\widehat A(\xi)-\widehat B(\xi),\widehat A(\xi),\widehat A(\xi)]_p+i{\mathcal R}[\widehat{B}(\xi),\widehat{A}(\xi)-\widehat{B}(\xi),\widehat{A}(\xi)]_p\\
&+i{\mathcal R}[\widehat{B}(\xi),\widehat{B}(\xi),\widehat{A}(\xi)-\widehat{B}(\xi)]_p.
\end{split}
\end{equation*}
\end{proof}

\subsection{Special dynamics of the resonant system}\label{SpeDyn}

In view of Theorems \ref{ModScatThm} and \ref{ExMWO}, it seems interesting to elaborate on some asymptotic dynamics for \eqref{RS}. From \eqref{DefHam} and \eqref{ConsRS} we have $d+3$ conserved scalar quantities and it is not hard to check that they are in involution. Below we illustrate some simple dynamics related to Remark \ref{NoCascade} and Corollary \ref{CorForwComp}, and finally recall the theorem from \cite{Hani} leading to the infinite cascade in Corollary \ref{Inf cascade cor} .

\begin{remark}\label{rreemm}
To transfer information from a global solution $a(t)$ of \eqref{RS} to a solution of \eqref{RSS}, all one needs to do is take an initial data of the form
\begin{equation*}
G_0(x,y)=\varepsilon_0\check{\varphi}(x)g(y)
\end{equation*}
where $g_p=a_p(0)$. The solution $G(t)$ to \eqref{RSS} with initial data $G_0$ as above is given in Fourier space by 
$$\widehat G_p(t, \xi)=\varphi(\xi)a_p(\varphi(\xi)^2 t).$$
In particular, if $\varphi=1$ on an open interval $I$, then $\widehat G_p(t, \xi)=a_p(t)$ for all $t\in \R$ and $\xi \in I$.
\end{remark}
%%%%%%%%%%%%%%%%%%%%%%%%%%%%%%%%%%%%%%%%%%%%%%%%%%%%%%%%%

We start with a simple observation that prevents linear scattering.

\begin{lemma}
Assume that $a$ solves \eqref{RS} and that
\begin{equation*}
\Vert \partial_t a\Vert_{l^2_p}\to0\hbox{ as }t\to+\infty,
\end{equation*}
then $a\equiv0$.
\end{lemma}

\begin{proof}
This follows from the conservation and coercivity of the mass and Hamiltonian:
\begin{equation*}
\mathcal{H}(a)=\langle i\partial_t a,a\rangle_{l^2_p\times l^2_p},\qquad \Vert a(t)\Vert_{l^2}=\hbox{mass}(a),
\end{equation*}
hence we see that $\mathcal{H}(a)=0$ and \eqref{DefHam} now implies that $a\equiv 0$.
\end{proof}

\subsubsection{The case $d=1$} This case can be integrated explicitely:

\begin{equation*}
i\partial_ta_p=2\sum_{q\in\mathbb{Z}}\vert a_q\vert^2a_p-\vert a_p\vert^2a_p.
\end{equation*}
Thus, we see that
\begin{equation}\label{Explicit1d}
a_p(t)=e^{ib_pt}a_p(0),\qquad b_p=2\hbox{mass}(a)-\vert a_p(0)\vert^2.
\end{equation}
In particular, $\vert a_p(t)\vert^2\equiv\vert a_p(0)\vert^2$ remains constant in time and there can be no cascade.

\subsubsection{Solutions supported on a rectangle}

The simplest genuinely multi-dimensional solution is supported on a rectangle $(p_0,p_1,p_2,p_3)$. We refer to \cite{CKSTTTor,GrPaTh} for related (and more elaborate) computations. Letting
\begin{equation*}
a_j=a_{p_j},\qquad j\in\{0,1,2,3\}=\mathbb{Z}/4\mathbb{Z}
\end{equation*}
we see that \eqref{RS} becomes
\begin{equation*}
i\partial_ta_j=2a_{j+1}\overline{a_{j+2}}a_{j-1}+2(\vert a_{j+1}\vert^2+\vert a_{j+2}\vert^2+\vert a_{j-1}\vert^2)a_{j}+\vert a_j\vert^2a_j.
\end{equation*}
An application of Gronwall's inequality shows that a solution initially supported in a rectangle will remain supported on this rectangle. Besides, we can see that mass, hamiltonian and momentum in two different directions in the span of the rectangle are generically independent and thus the Liouville-Arnold-Jost theorem provides many $4$-torii of solutions.

\medskip

There is a simple subsystem corresponding to the case when
\begin{equation*}
b_0(t):=a_0(t)=a_2(t),\qquad b_1(t):=a_1(t)=a_3(t)
\end{equation*}
which, by an application of Gronwall's inequality can be seen to be invariant by the flow. Besides, \eqref{RS} becomes
\begin{equation*}
\begin{split}
i\partial_t b_j=-|b_j|^2 b_j +4b_j(|b_j|^2+|b_{j+1}|^2)+2b_{j+1}^2 \overline{b_j},\qquad j\in\{0,1\}=\mathbb{Z}/2\mathbb{Z}
\end{split}
\end{equation*}
Without any loss of generality, we can normalize the initial data so that $|b_0|^2+|b_1|^2=1$ (see Remark \ref{RemScal}). We now move to polar coordinates and define
$$
I_j=|b_j|^2\qquad \mbox{and}\qquad \theta_j= \arg b_j-4mt,\qquad m=\hbox{mass}(b)=\vert b_0\vert^2+\vert b_1\vert^2.
$$
A direct calculation, shows that the system satisfied by $(I_j, \theta_j)$ is given by
\begin{equation}\label{rtheta}
\begin{split}
\dot \theta_j=I_j-2I_{j+1}\cos (2(\theta_{j+1}-\theta_{j})),\qquad\dot I_j=& 4I_jI_{j+1} \sin (2(\theta_{j+1}-\theta_{j}))\\
\end{split}
\end{equation}
The conservation of mass and Hamiltonian translate in the above variables into
\begin{equation}\label{massHamZ}
I_0+I_1=1;\qquad \tilde h(I_0, I_1, \theta_0, \theta_1)=\frac{1}{2}(I_0^2+I_1^2)-2I_0I_1\cos (2(\theta_0-\theta_1))=cst
\end{equation}
It is easy to see either by direct verification or by noticing that all the above variable changes are symplectic that the above system \eqref{rtheta} is Hamiltonian. Let $r=I_0$ and define $\varphi=\theta_1-\theta_0$. The system satisfied by $(r, \varphi)$ is the following:
\begin{equation}\label{rphi}
\dot \varphi= (1-2r)(1+2\cos 2\varphi), \qquad \dot r= 4r(1-r) \sin 2\varphi,
\end{equation}
which is also Hamiltonian with energy
\begin{equation*}
h(\varphi,r)=r(1-r)[1+2\cos(2\varphi)].
\end{equation*}
Due to our mass normalization, we have that $r\in [0,1]$ for all time. Notice\footnote{Also notice that the energy curve $h=0$ supports only two types of orbit namely that given by $\cos 2\varphi=-\frac 12$ and $\dot r= \pm 2\sqrt 3 r(1-r)$ which leads to the heteroclinic orbit at the basis of the construction in \cite{CKSTTTor}.} that $(I_0, I_1, \theta_0, \theta_1)$ can all be derived from the knowledge of $(\varphi,r)$ and \eqref{rtheta}.

Looking at the phase diagram inside the rectangle defined by the invariant lines $\{r=0\}, \{r=1\}, \{\varphi=-\frac{\pi}{3}\}$ and $\{\varphi=\frac{\pi}{3}\}$, we notice that $(\varphi=0,r=1/2)$ is the only stationary point and therefore the level sets $\{h(\varphi,r)=a\}$ foliate this rectangle as $a$ ranges between the two extreme values: 0 attained at the boundary and $3/4$ attained at the center (see Figure \ref{figure1}). 

\begin{figure}[h!]
\centering
\includegraphics[width=0.8\linewidth]{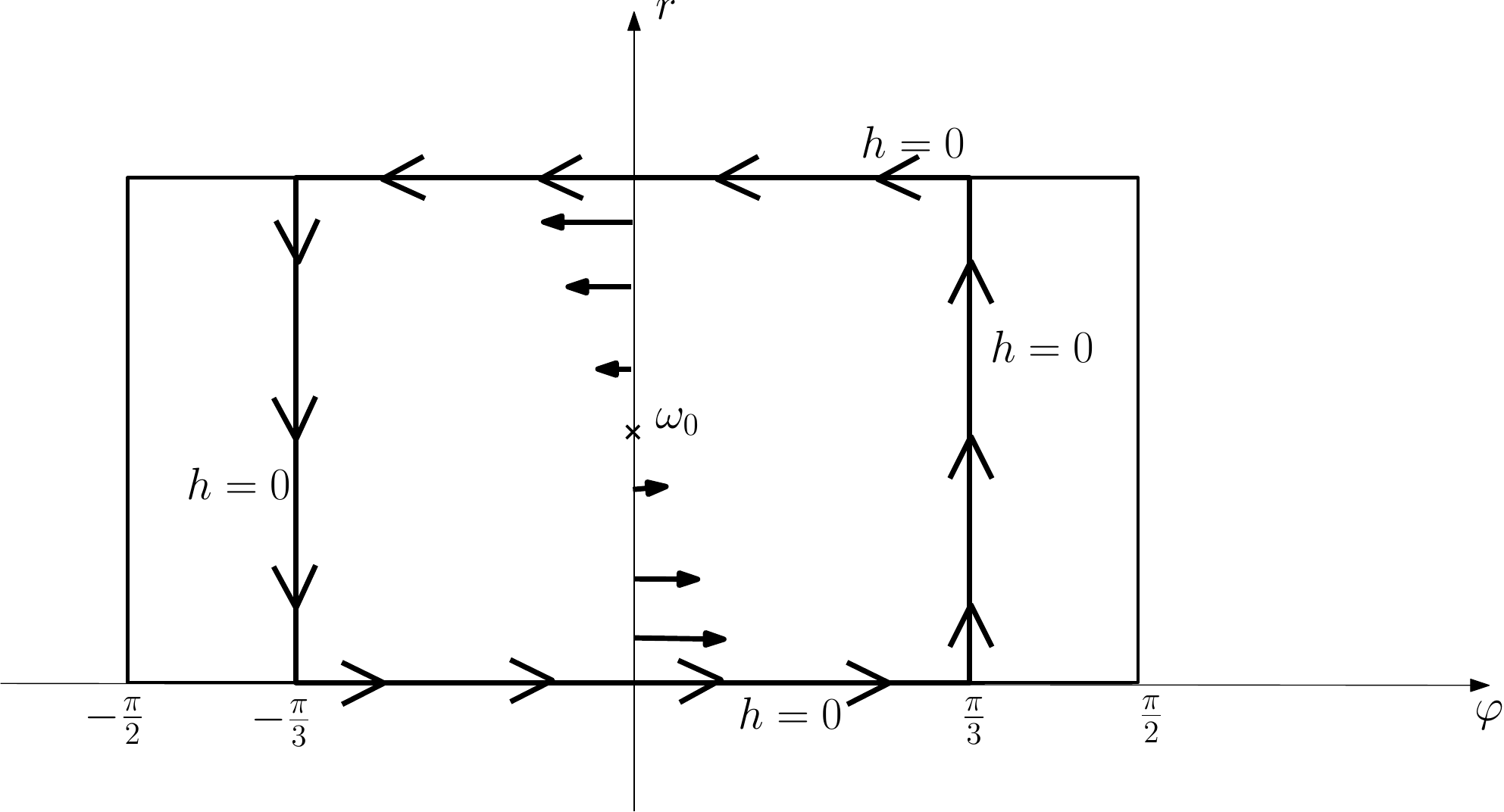}\\
\caption{Phase diagram for $h(\varphi,r)$}\label{figure1}
\end{figure}

An application of the Liouville-Arnold-Jost theorem shows that if we start with initial data $(\varphi=0,r=1-\delta)$, then the solution to \eqref{rphi} will be periodic in time with energy level given by $h(0,1-\delta)=3\delta(1-\delta)$. If $2T$ is the period, then the value of $r(t)$ will oscillate between the two extreme values of $\delta$ and $1-\delta$ attained at respectively even and odd multiples of $T$.

All in all, we have
\begin{proposition}\label{periodic RS}
Let $\Lambda$ be a rectangle with vertices $p_1, p_2, p_3, p_4$ in $\Z^d$. Let $\Lambda_1=\{p_1, p_3\}$ and $\Lambda_2=\{p_2, p_4\}$ denote the diagonally opposite pairs. 
\begin{enumerate}
\item
%There exists an invariant open subset $O$ of $\C^4$ on which \eqref{RS} is completely integrable in the sense of Liouville-Arnold-Jost. 
There exists solutions to \eqref{RS} supported on $R$ that are quasi-periodic with up to 4 periods (4 angle variables).
\item For any $\delta>0$, there exists a \emph{periodic-in-time} solution $(a_n(t))$of \eqref{RS} supported on $R$, with period $2T$, and satisfying the following 
\begin{equation*}
\begin{split}
\hbox{mass}[(a_n(0)), \Lambda_1]=\delta \qquad \mbox{ and }\qquad \hbox{mass}[a_n(0), \Lambda_2]=1-\delta\\
\hbox{mass}[(a_n(T)), \Lambda_1]=1-\delta \qquad \mbox{ and }\qquad \hbox{mass}[a_n(T), \Lambda_2]=\delta
\end{split}
\end{equation*}
where we denoted by $\hbox{mass}[(a_n), \Lambda_j]=\sum_{n\in \Lambda_j} |a_n|^2$ and $T$ is half the period of motion.
\end{enumerate}
\end{proposition}

\medskip

\begin{remark}\label{zremark}
While the above solutions were supported on one rectangle in $\Z^d$, one can actually construct the same solutions on any (possibly infinite) family of rectangles $\{\Lambda_l\}$ as long as the system \eqref{RS} decouples to each rectangle. This can be achieved by making sure that the rectangles $\Lambda_l$ do not form resonant interactions between them. We refer to \cite{Hani} for the precise definitions. In particular, the set $\Lambda_1, \Lambda_2,$ and $\Lambda=\Lambda_1\cup \Lambda_2$ in Proposition \ref{periodic RS} can be made infinite.
\end{remark}

The solutions constructed in the above proposition directly yield time periodic and quasi-periodic solutions of \eqref{RSS} by setting $\widehat G_p(t, \xi)=\frac12\mathbf{1}_{[-1,1]}(\xi) a_p(t)$. However, such solutions are in $H^N(\R\times \T^d)$ but not in $S$ or $S^+$. To fix this caveat, one can use, instead of $\mathbf1_{[-1,1]}(\xi)$, a smooth even function $\psi_\epsilon(\xi)$ satisfying
\begin{equation}\begin{cases}\label{psiepsilon}
\psi_\epsilon(\xi)=\frac{1}{2} \quad& |\xi|\leq 1-\epsilon\\
\psi_\epsilon(\xi)=0\quad&|\xi|\geq 1
\end{cases}
\end{equation}
and a smooth non-negative non-increasing interpolant on the interval $[1, 1+\epsilon]$. One can also arrange so that
%$\|\psi_\epsilon\|_{S}\leq 2\epsilon^{-1}$ and
$\|\psi_\epsilon\|_{S^+}\leq 3 \epsilon^{-2}$. If the initial data for $G$ is taken to be $\widehat G_p(0, \xi)=\epsilon^3\psi_\epsilon(\xi) a_p(0)$, then the obtained solution $G(t)$ is given by $\widehat G_p(t,\xi)=\epsilon^3\psi_\epsilon(\xi) a_p(\epsilon^3\psi_\epsilon(\xi)^2 t)$ (see Remark \ref{RemScal}). Notice that the $S$ and $S^+$ norms of $G(0)$ are then $O(\epsilon)$.

%\medskip
%
%Finally, we recall the Liouville-Arnold-Jost Theorem:
%\begin{theorem}[Liouville-Arnold-Jost \cite{Arnold,KKdV}]\label{LAJ}
%Let $(M, \nu)$ be a symplectic manifold of dimension 2n, and suppose that $F=(F_1, \ldots, F_n)$ be n independent functions on $M$ in involution\footnote{That is, $\{F_i, F_j\}=0$ for all $i, j\in\{1,\ldots, n\}$, where $\{\cdot, \cdot\}$ is the Poisson bracket.}. Suppose that one of the leaves of $F$, say $M_0=F^{-1}(0)$, is compact and connected. Then
%\begin{enumerate}
%\item $M_0$ is an $n-$dimensional embedded torus.
%\item There exists a neighborhood $U$ of $M_0$ in $M$, a neighborhood $V$ of 0 in $\R^n$, and a symplectic diffeomorphism $\Phi$ such that 
%$$
%\Phi: V\times \T^n \to U,\qquad (I, \varphi)\in V\times \T^n \mapsto x\in M
%$$
%introducing action-angle coordinates such that $M_0=\Phi(\{0\}\times \T^n)$ and $F_j\circ \Phi$ is independent of the angle coordinates.
%\end{enumerate}
%\end{theorem}

\subsubsection{Infinite cascade}

An important result for us is the existence of infinitely growing solutions to \eqref{RS} as proved in  \cite[Theorem 1.6 with $R=0$]{Hani}. We give a self-contained constructive proof of this result that follows from simple adaptations of the more recent work \cite{GuKa} in order to obtain an explicit global solution with a lower bound on the growth rate of its Sobolev norms. Our main result here is the following.

\begin{proposition}\label{growth on Z^d theo}%[\cite{Hani} and Appendix \ref{Appendix}]\label{growth on Z^d theo}
Let $d\geq 2$ and $s>1$. There exists global solutions to \eqref{RS} in $C(\R: h^s_p)$ such that
\begin{equation*}
\sup_{t>0}\|a(t)\|_{h^s_p}=\infty.
\end{equation*}
More precisely, for any $\varepsilon>0$, there exists a solution $a(t) \in C(\R: h^s_p)$ such that for some sequence of times $t_k\to \infty$ we have that
\begin{equation}\label{PropNormGrowthSol}
\Vert a(0)\Vert_{h^s_p}\leq \varepsilon,\qquad \Vert a(t_k)\Vert_{h^s_p}\gtrsim \exp(c(\log t_k)^\frac{1}{2})
\end{equation}
for some $c>0$.
\end{proposition}

By Remark~\ref{rreemm}, this yields a global solution of \eqref{RSS} in $C(\R: H^s(\R\times \T^d))$ whose $H^s$ norm grows at the rate \eqref{PropNormGrowthSol}. In particular, this solution grows (along a subsequence) faster than any power of $\log t$. For the NLS equation \eqref{CNLS} (by Theorem \ref{ExMWO}), this yields a growth of $\exp(c(\log\log t)^{1/2})$, i.e. faster than any power of $\log\log t$. We have no reason to believe that the rate of growth in \eqref{PropNormGrowthSol}, or the implied rate for \eqref{CNLS} is optimal. In addition, it is tempting to believe that for any $s>1$, there exists a solution in $H^\infty(\mathbb{R}\times\mathbb{T}^2)$ whose $H^s$ norm blows up in infinite time.
\medskip

We now move to the proof of Proposition \ref{growth on Z^d theo}. We start by noticing that it is enough to prove the result on $\Z^2$ as this gives a solution of \eqref{RS} on $\Z^d$ satisfying the same properties. In addition, we note as in \cite{CKSTTTor} that by an easy change of unknown,
\begin{equation}\label{Gauge}
a_p(t)\to a_p(t)e^{iGt},\qquad G=2\Vert a_p\Vert_{l^2_p}^2
\end{equation}
we may reduce \eqref{RS} to the system 
\begin{equation}\label{RSAdj}
i\partial_ta_p=-\vert a_p\vert^2a_p+\sum_{(p,q,r,s)\in\Gamma^\prime_0}a_q\overline{a}_ra_s,
\end{equation}
where $\Gamma^\prime_0$ corresponds to the non-degenerate rectangles $(p,q,r,s)$, i.e. rectangles with positive area. Of course, the transformation \eqref{Gauge} does not change the $h^s_p$-norms and may be easily inverted.

Next, we recall the following result, which is essentially contained in \cite[Theorem 3-bis and Appendix C]{GuKa}:

\begin{theorem}[\cite{CKSTTTor,GuKa}]\label{TdThm}
Fix $\gamma\gg1$. There exists $C,\nu>0$ (independent of $\gamma$) such that for any $N$ sufficiently large, 
there exists a finite set $S_N\subset\mathbb{Z}^2$ and a solution $a^{(N)}(t)=(a^{(N)}_k(t))_{k \in \Z^2}$ of \eqref{RSAdj} such that:
\begin{itemize}
\item ($0_\Lambda$) If $(p_0,q_0,r_0)$ form a right-angled triangle (at $q_0$) in $S_N$, then $r_0+p_0-q_0\in S_N$, i.e. a rectangle has either $4$ or (strictly) less than $3$ of its vertices inside $S_N$,
\item ($I_\Lambda$) $S_N=\Lambda_1\cup\Lambda_2\cup\dots\cup\Lambda_N\subset B(0,10^{6N^2})$,
\item ($II_\Lambda$) $\Lambda_j$ contains $2^{N-1}$ points, $1\le j\le N$,
\item ($III_\Lambda$) If $\Lambda_j\subset B(0,r)$, then $\Lambda_{j+1}\subset B(0,\sqrt{2}r)$,
\item ($IV_\Lambda$) There exists $R>0$ such that $\Lambda_1$ is contained in a disc of radius $R\le 10^{6N^2}$ and $\Lambda_{N-1}$ 
contains at least two points at distance $R2^{(N-10)/2}$ from the origin,
%\item if $\Lambda_j\subset B(0,R)$, then $\Lambda_{j+1}\subset B(0,\sqrt{2}R)$.
%\item $S_N\subset B(0, (60)^{N^2})$,
%\item There holds that
%\begin{equation*}
%\frac{\sum_{n\in\Lambda_{N-2}}\vert n\vert^{2s}}{\sum_{n\in\Lambda_{3}}\vert n\vert^2}\ge 2^{(s-1)(N-10)}
%\end{equation*}
\end{itemize}
and the solution $a^{(N)}(t)=(a^{(N)}_k(t))_{k\in \Z^2}$ satisfies:
\begin{itemize}
\item ($I_a$) for all times, $a^{(N)}(t)$ is supported on $S_N$ and for any $j=1,...,N$, $a^{(N)}(t)$ is constant on $\Lambda_j$, i.e. 
$a^{(N)}_k(t)=b^{(N)}_j(t)$ for $k \in \Lambda_j$,
\item ($II_a$) $a^{(N)}(t)$ cascades energy in the sense that there exists $T_N$ such that:
\begin{equation*}
\begin{split}
1\geq \vert b^{(N)}_3(0)\vert>1-\delta^\nu,&\qquad\vert b^{(N)}_{N-1}(T_N)\vert>1-\delta^\nu,\\
\vert b^{(N)}_j(0)\vert<\delta^\nu \hbox{ for }  j\ne 3&\qquad\vert b^{(N)}_j(T_N)\vert<\delta^\nu \hbox{ for } j\ne N-1
\end{split}
\end{equation*}
where  $\delta=e^{-\gamma N}$,
\item ($III_a$) there holds that
$0<T_N<C\gamma N^2$.
\end{itemize}
\end{theorem}
%\begin{remark}
%The condition $\vert b^{(N)}_3(0)\vert\leq 1$ in $II_a$ is not explicitely written in Theorem 3-bis in \cite{GuKa}.
%However by looking at the argument in \cite{GuKa} it is clear that it can be assumed. 
%\end{remark}

\begin{proof}[Proof of Proposition \ref{growth on Z^d theo}]
The needed solution is constructed using the observation (see \cite{Hani}) that compactly supported solutions of \eqref{RSAdj} of disjoint support can be easily superposed by appropriately positioning them in the lattice $\Z^2$.

Fix $s>1$ and $\gamma>2s/\nu$. We start by running Theorem \ref{TdThm} for every $N=j\in\mathbb{N}$, $j\ge N_0(\gamma)$. This gives a family of sets $S_j=\Lambda_1^j\cup\Lambda^j_2\cup\dots\cup\Lambda^j_j$ satisfying $(I_\Lambda)-(IV_\Lambda)$ and solutions $a^{(j)}(t)$ of \eqref{RSAdj} satisfying $(I_a)-(III_a)$.

In addition, considering ($III_\Lambda-IV_\Lambda$), we see that we may assume that there exists $\frac 12 10^{6j^2}\le R_j\le 10^{6j^2}$ such that
\begin{equation}\label{Sizellambda}
\Lambda^j_p\subset B(0,\sqrt{2}^pR_j)\mbox{ for }1\leq p \leq j; \qquad \Lambda^j_{j-1}\cap B(0,2^{(j-20)/2}R_j)^c\ne\emptyset.
\end{equation}
\medskip

Next, we claim that we can construct by induction a sequence of vectors $\{v_j\}_{j\geq N_0}\subset \Z^2$ such that:
\begin{equation}\label{Sizevj}
v_{N_0}=0, \qquad \vert v_j\vert\le 2^{10j},
\end{equation} and 
for any nondegenerate rectangle $(p_0, q_0, r_0, s_0)$ with three vertices included in
\begin{equation*}
\Xi=\bigcup_{j\ge N_0} (v_j+S_j)
\end{equation*}
then $\{p_0,q_0,r_0,s_0\}\subset\Xi$ and 
we have the following property:
\begin{equation}\label{rectangle}\hbox{ if } \{p_0, q_0, r_0, s_0\}\cap (v_j +S_j)\neq \emptyset \hbox{ and } \{p_0, q_0, r_0, s_0\}\cap (v_k +S_k)\neq \emptyset \hbox{ then }
j=k.\end{equation}

The existence of this sequence of vectors is proved inductively using Lemma \ref{vjlemma} below (at the $n-$th step, take $\Xi_n=\cup_{1\leq j \leq n-1}(S_j+v_j)$ which has $O(n2^n)$ elements). We then easily see that any nondegenerate right-angled triangle in $\Xi$ must belong to exactly one $v_j+S_j$. The fact that the fourth corner of a rectangle necessarily belongs to $\Xi$ follows from the fact that each component $v_j+S_j$ satisfies this property thanks to $(0_\Lambda)
$ above. Choosing any such sequence $\{v_j\}_{j\geq N_0}$, we define the following sequence of initial data $A^{(p)}(0)$ for \eqref{RS} to be given by
\begin{equation*}
\begin{split}
A^{(p)}(0)&= \sum_{N_0\leq j \leq p} \lambda_j a^{(j)}_{k-v_j}(0),
\end{split}
\end{equation*}
where $\lambda_j=(\varepsilon/j^{10})2^{-j/2}R_j^{-s}$ is a normalization factor. Note that for any $v\in \Z^2$, $a^{(j)}_{k-v}(t)$ is also a solution of \eqref{RSAdj}. Using $(I_a-II_a)$, \eqref{Sizellambda}, \eqref{Sizevj}
we therefore see that
\begin{equation*}
\begin{split}
\Vert A^{(p+1)}(0)-A^{(p)}(0)\Vert_{h^s_p}^2\lesssim p^{-20}\varepsilon^2,
\end{split}
\end{equation*}
so that $A^{(p)}(0)$ is a Cauchy sequence of initial data in $h^s_p$, and therefore it converges to some $A(0)\in h^s_p(\Z^2)$. Moreover, $A(0)$ satisfies the first property in \eqref{PropNormGrowthSol}.

What remains to show is that the solution $A(t)$ of \eqref{RS} with initial data $A(0)$ satisfies the second property in \eqref{PropNormGrowthSol}. We start by noticing that by \eqref{rectangle} (recall that $(p_0, q_0, r_0, s_0)\in \Gamma_0^\prime$ if and only if $(p_0, q_0, r_0, s_0)$ are the vertices of a nondegenerate rectangle), the solution $A^p(t)$ with initial data $A^p(0)$ is given by 
 $$
 A^p(t)=\sum_{N_0\leq j \leq p} \lambda_j a^{(j)}_{k-v_j}(\lambda_j^2t).
 $$
As a result, we see that if $m\ge n$ and $k \in v_n+S_n$, then
 \begin{equation}\label{UCte}
 A^{(m)}_k(t)=A^{(n)}_k(t)=\lambda_n a^{(n)}_{k-v_n}(\lambda_n^2t)=\lambda_n\sum_{1\le \ell \le n}b^{(n)}_{\ell}(\lambda_n^2t)\mathbf{1}_{\Lambda^n_\ell}(k-v_n).
 \end{equation}
By continuity of the flow, this also holds for $A^{(m)}(t)$ replaced by $A(t)$. In particular, using $(IV_\Lambda, I_a-II_a)$ and \eqref{Sizevj} we see that
\begin{equation*}
\begin{split}
\Vert A(\lambda_n^{-2}T_n)\Vert_{h^s}^2&\ge \lambda_n^2\sum_{k\in\Lambda^n_{n-1}}\vert b^{(n)}_{n-1}(T_n)\vert^2\cdot\vert k+v_n\vert^{2s}
%\gtrsim \lambda_j^22 (R_j2^{\frac{j-10}{2}})^{2s}\gtrsim \lambda_j^2R_j^{2s}2^{js}
\gtrsim n^{-20}\varepsilon^22^{n(s-1)}.
\end{split}
\end{equation*}
This finishes the proof using $(III_a)$. 
\end{proof}

We now present the lemma justifying the existence of the sequence $\{v_j\}$ above.% (cf. Proposition 4.7 of \cite{Hani} for a similar argument):
 
\begin{lemma}\label{vjlemma}
Let $\Xi\subset \Z^2$ have cardinality $O(j2^{j})$, and let $S_j$ be the set obtained from Theorem~\ref{TdThm} with $N=j$. % such that there exists $\frac 12 10^{6j^2}\le R_j\le 10^{6j^2}$ such that
%\begin{equation}\label{Sizellambda}
%\Lambda^j_p\subset B(0,\sqrt{2}^pR_j)\mbox{ for }1\leq p \leq j; \qquad \Lambda^j_{j-1}\cap B(0,2^{(j-20)/2}R_j)^c\ne\emptyset.
%\end{equation}
Then there exists $v \in \Z^2$ with $|v|\leq 2^{10j}$ such that for any nondegenerate right-angled triangle $(p_0, q_0, r_0)$ we have the following property:
\begin{equation}
\begin{split}
\hbox{ if } \vert\{p_0, q_0, r_0\}\cap \Xi\vert \ge 2,&\quad\hbox{ then } \quad\{p_0, q_0, r_0\}\cap (v+S_j)= \emptyset,\\
\hbox{ if } \vert\{p_0, q_0, r_0\}\cap (v+S_j)\vert \ge 2,&\quad\hbox{ then } \quad\{p_0, q_0, r_0\}\cap \Xi= \emptyset.
\end{split}
\end{equation}
\end{lemma}

\begin{proof} 

Let $\mathcal{L}$ denote the set of directions of lines joining two points of $\Xi$ or two points of $S_j$, or directions which are orthogonal to such lines. $\mathcal{L}$ has cardinality at most $2^{3j}$ and there exists a vector $v^\prime$ of length at most $2^{4j}$ which is not contained in $\mathcal{L}$.

We now define
$${\mathcal A}=\{(p,q,r),p,q\in  \Xi, r\in S_j\}, \qquad{\mathcal B}=\{(p,q,r),p\in  \Xi, q,r\in S_j\}.$$

We claim that for any $(p,q,r)\in\mathcal{A}$,  the condition ``($C^1_{pqr})$\,:\,$(p,q,r+\lambda v^\prime)$ form a right-angled triangle'' has at most two solutions $\lambda\in\mathbb{R}$ and that similarly, for any $(p,q,r)\in\mathcal{B}$, the condition ``$(C^2_{p,q,r})$\,:\, $(p,q+\lambda v^\prime,r+\lambda v^\prime)$ form a right-angled triangle'' has at most two solutions.

By translation invariance, it suffices to prove the first claim. If the right-angle is at $p$ or $q$, then the proof is direct since $v^\prime$ is not orthogonal nor parallel to $p-q$. If the right-angle happens at $r+\lambda v^\prime$, then $r+\lambda v^\prime$ belongs to the circle of diameter $(p,q)$ and a line directed by $v^\prime$ will intersect this circle in at most two points.

We now observe that $\vert\mathcal{A}\vert+\vert\mathcal{B}\vert\lesssim 2^{4j}$, and therefore we may choose $\lambda\in \mathbb{Z}\cap [0,2^{5j}]$ such that $(C^1_{pqr})$ and $(C^2_{pqr})$ are never satisfied. We now set $v=\lambda v^\prime$.

\end{proof}

%Notice that by Remark \ref{RemScal}, we can normalize the initial data in Theorem \ref{growth on Z^d theo} to any pre-assigned size $\delta$ in $h^s_p$. 

\section{Modified wave operators}\label{SMWO}

We start the proof of our main results with the slightly easier task of constructing (modified) wave operators for \eqref{CNLS}. The following implies Theorem \ref{ExMWO}.

\begin{theorem}
There exists $\varepsilon>0$ such that if $U_0\in S^+$ satisfies
\begin{equation}\label{AssMW}
\Vert U_0\Vert_{S^+}\le\varepsilon,
\end{equation}
and if $\widetilde{G}$ is the solution of \eqref{RSS} with initial data $U_0$, then there exists $U$ a solution of \eqref{CNLS} such that
$e^{-it\Delta_{\mathbb{R}\times\mathbb{T}^d}}U(t)\in C((0,\infty):S)$ and
\begin{equation*}
\begin{split}
\Vert e^{-it\Delta_{\mathbb{R}\times\mathbb{T}^d}}U(t)-\widetilde{G}(\pi\ln t)\Vert_S\to0\,\hbox{ as }\,t\to+\infty.
\end{split}
\end{equation*}

\end{theorem}

\begin{proof}
This follows by a fixed point argument. We let $G(t)=\widetilde{G}(\pi\ln t)$ and define the mapping
\begin{equation*}
\Phi(F)(t)=-i\int_t^\infty\left\{\mathcal{N}^\sigma[F+G,F+G,F+G]-\frac{\pi}{\sigma}\mathcal{R}[G(\sigma),G(\sigma),G(\sigma)]\right\}d\sigma
\end{equation*}
and the space\footnote{Of course continuing a solution $U$ of \eqref{CNLS} on the interval $(0,1)$ is direct.} 
\begin{equation*}
\begin{split}
\mathfrak{A}:=&\{F\in C^1((1,\infty):S)\,:\,\,\Vert F\Vert_\mathfrak{A}\le\varepsilon_1\}\\
\Vert F\Vert_\mathfrak{A}:=&\sup_{t> 1}\left\{(1+\vert t\vert)^\delta\Vert F(t)\Vert_{S}+(1+\vert t\vert)^{2\delta}\Vert F(t)\Vert_{Z}+(1+\vert t\vert)^{1-\delta}\Vert \partial_tF(t)\Vert_S\right\}
\end{split}
\end{equation*}
and we claim that if $\varepsilon$ is sufficiently small, there exists $\varepsilon_1$ such that $\Phi$ defines a contraction on the complete metric space $\mathfrak{A}$ endowed with the metric $\Vert \cdot\Vert_\mathfrak{A}$.

\medskip

We now decompose
\begin{equation}\label{Dec}
\begin{split}
\mathcal{N}^t[F+G,F+G,F+G]-\frac{\pi}{t}\mathcal{R}[G,G,G]&=\mathcal{E}^t[G,G,G]+\mathcal{L}^t[F,G]+\mathcal{Q}^t[F,G]\\
\end{split}
\end{equation}
where $\mathcal{E}^t[G,G,G]$ is defined as in \eqref{DecNon0}
%as in Proposition \ref{StrucNon} 
and
\begin{equation*}
\begin{split}
\mathcal{L}^t[F,G]&:=2\mathcal{N}^t[G,G,F]+\mathcal{N}^t[G,F,G],\\
\mathcal{Q}^t[F,G]&:=2\mathcal{N}^t[F,F,G]+\mathcal{N}^t[F,G,F]+\mathcal{N}^t[F,F,F].
\end{split}
\end{equation*}
We will show that, whenever $F,F_1,F_2\in\mathfrak{A}$,
\begin{equation}\label{SuffFPMWO}
\begin{split}
&\Vert \int_t^\infty\mathcal{E}^\sigma[G,G,G]d\sigma\Vert_\mathfrak{A}\lesssim \varepsilon^3,\\
&\Vert \int_t^\infty\mathcal{L}^\sigma[F,G]d\sigma\Vert_\mathfrak{A}\lesssim \varepsilon^2\Vert F\Vert_\mathfrak{A},\\
&\Vert \int_t^\infty\mathcal{Q}^\sigma[F,G]d\sigma\Vert_\mathfrak{A}\lesssim \varepsilon\Vert F\Vert_\mathfrak{A}^2,\\
&\Vert \int_t^\infty\left\{\mathcal{Q}^\sigma[F_1,G]-\mathcal{Q}^\sigma[F_2,G]\right\}d\sigma\Vert_\mathfrak{A}\lesssim \varepsilon\varepsilon_1\Vert F_1-F_2\Vert_\mathfrak{A}.
\end{split}
\end{equation}
Once \eqref{SuffFPMWO} is shown, the proof is complete.

\medskip

Recall that, if $\varepsilon\lesssim \delta^\frac{1}{2}$ and $F\in \mathfrak{A}$, (see Lemma \ref{gAdmi} for the estimates on $G$)
\begin{equation}\label{APEMWO}
\begin{split}
(1+\vert t\vert)^{2\delta}\Vert F(t)\Vert_Z+(1+\vert t\vert)^\delta \Vert F(t)\Vert_S+(1+\vert t\vert)^{1-\delta}\Vert\partial_t F(t)\Vert_S&\lesssim\varepsilon_1,\\
\Vert  G(t)\Vert_{S^+}+(1+\vert t\vert)\Vert\partial_tG(t)\Vert_{S^+}&\lesssim \varepsilon(1+\vert t\vert)^{\delta/100}\\
\Vert G(t)\Vert_Z&\lesssim\varepsilon.
\end{split}
\end{equation}
Using \eqref{CrudeEst}, the two last inequalities of \eqref{SuffFPMWO} follow.

\medskip

We now turn to the first inequality in \eqref{SuffFPMWO}. Using \eqref{CrudeEst} again (see also \eqref{Trivial111}), we easily see that
\begin{equation*}
\Vert \mathcal{E}^t[G,G,G]\Vert_S\le \Vert\mathcal{N}^t[G,G,G]\Vert_S+\frac{1}{t}\Vert \mathcal{R}[G,G,G]\Vert_S\lesssim (1+\vert t\vert)^{-1+\delta}\varepsilon^3.
\end{equation*}
This controls the time derivative in the $\mathfrak{A}$-norm. Independently, using \eqref{APEMWO} with Proposition \ref{StrucNon} we obtain that
\begin{equation*}
\|\int_{t}^{\infty} \mathcal{E}^{\sigma}(G,G,G)d\sigma\|_{S}\lesssim\varepsilon^3(1+|t|)^{-\delta},\qquad \|\int_{t}^{\infty} \mathcal{E}^{\sigma}(G,G,G)d\sigma\|_{Z}\lesssim \varepsilon^3(1+|t|)^{-2\delta}
\end{equation*}
This gives the first inequality in \eqref{SuffFPMWO}.

\medskip

Now we turn to the second inequality in \eqref{SuffFPMWO}. First, using \eqref{CrudeEst} and \eqref{APEMWO}, we see that
\begin{equation*}
\Vert\mathcal{N}^t[G,G,F]\Vert_S+\Vert \mathcal{N}^t[G,F,G]\Vert_S\lesssim \varepsilon^2\varepsilon_1(1+\vert t\vert)^{-1+\delta}
\end{equation*}
which is sufficient for the time-derivative component of the $\mathfrak{A}$-norm. Using \eqref{APEMWO} with Lemma \ref{BilEf} and Lemma \ref{FastOsLem}, it only remains to show that
\begin{equation}\label{EstimUUG1}
\begin{split}
\Vert \mathcal{R}[G,G,F]\Vert_{Z}+\Vert \mathcal{R}[G,F,G]\Vert_{Z}\lesssim (1+\vert t\vert)^{-2\delta}\varepsilon^2\varepsilon_1,\\
\Vert \Pi^t[G,G,F]-\frac{\pi}{t}\mathcal{R}[G,G,F]\Vert_{Z}+\Vert \Pi^t[G,F,G]-\frac{\pi}{t}\mathcal{R}[G,F,G]\Vert_{Z}\lesssim (1+\vert t\vert)^{-1-2\delta}\varepsilon^2\varepsilon_1,\\
\Vert \Pi^t[G,G,F]\Vert_{S}+\Vert \Pi^t[G,F,G]\Vert_{S}\lesssim (1+\vert t\vert)^{-1-\delta}\varepsilon^2\varepsilon_1\\
\end{split}
\end{equation}
Using Lemma \ref{DiscreteStricTorLem}, we see that for any $A, B, C\in Z$,
\begin{equation*}
\Vert \mathcal{R}[A,B,C]\Vert_{Z}\lesssim \Vert A\Vert_Z\Vert B\Vert_{Z}\Vert C\Vert_Z
\end{equation*}
and the first estimate follows from \eqref{APEMWO}. The second estimate follows directly from \eqref{ApproxRes}. For the third estimate, we use \eqref{GrowthSRes} to get
\begin{equation*}
\begin{split}
(1+\vert t\vert)\left\{\Vert \Pi^t[G,G,F]\Vert_{S}+\Vert \Pi^t[G,F,G]\Vert_{S}\right\}
&\lesssim \Vert G\Vert_{\tilde{Z}_t}^2\Vert F\Vert_S+\Vert G\Vert_{\tilde{Z}_t}\Vert F\Vert_{\tilde{Z}_t}\Vert G\Vert_S\\
&\lesssim \varepsilon^2\varepsilon_1(1+\vert t\vert)^{-\delta}.
\end{split}
\end{equation*}
The proof is complete.
\end{proof}
\begin{remark}
Observe that a key point in the proof of the existence of a modified wave operator is the fact that 
 $$
 \int_t^\infty\mathcal{E}^\sigma[G,G,G]d\sigma
$$
behaves better in the $Z$ norm compared to $G$ itself. This allows to get decay in the $S$ norm by assuming the stronger (only in $x$) 
$S^+$ control on the solution of
 \eqref{RSS}.
We also observe that in the modified wave operator proof, the argument is completely perturbative. 
We shall see in the next section that in the modified scattering proof the argument is not completely perturbative and relies on the conservation laws of the resonant 
system.
\end{remark}
%%%%%%%%%%%%%%%%%%%%%%%%%%%%%%%%%%%%%%%%%%%%%%%%%%%%%%%%%%%%%%%%%%%
%%%%%%%%%%%%%%%%%%%%%%%%%%%%%%%%%%%%%%%%%%%%%%%%%%%%%%%%%%%%%%%%%%%
%%%%%%%%%%%%%%%%%%%%%%%%%%%%%%%%%%%%%%%%%%%%%%%%%%%%%%%%%%%%%%%%%%%
\section{Small data scattering}\label{SMS}

The goal of this section is to prove a more precise version of Theorem \ref{ModScatThm} which is the main result of this paper.

\begin{theorem}\label{ModScatThm2}
There exists $\varepsilon>0$ such that if $U_0\in S^+$ satisfies
\begin{equation}\label{USmallID}
\Vert U_0\Vert_{S^+}\le\varepsilon,
\end{equation}
and if $U$ is the solution of \eqref{CNLS} with initial data $U_0$, then $U$ exhibits modified scattering to the resonant dynamics given by \eqref{RSS} in the following sense:
there exists $G_0\in S$ such that, letting $\widetilde{G}$ be the solution of \eqref{RSS} with initial data $\widetilde{G}(0)=G_0$, it holds that
\begin{equation}\label{ModScat}
\Vert F(t)-\widetilde{G}(\pi\ln t)\Vert_{S}\to 0\qquad\hbox{ as }t\to+\infty,
\end{equation}
where $F(t)=e^{-it\Delta_{\mathbb{R}\times\mathbb{T}^d}}U(t)$.
\end{theorem}

\subsection{Global bounds}

Before we turn to the asymptotic behavior of solutions, we need to obtain good global bounds. This is the purpose of the following:

\begin{proposition}
There exists $\varepsilon>0$ such that any initial data $u_0\in S^+$ satisfying \eqref{USmallID} generates a global solution of \eqref{CNLS}. Moreover, for any $T>0$, there holds that
\begin{equation}\label{AABA}
\Vert F(t)\Vert_{X^+_T}\le 2\varepsilon,
\end{equation}
where $F$ is defined in Theorem \ref{ModScatThm2}.
\end{proposition}

In case $d\le 3$, global existence can be established in a much more general setting (namely $U_0\in H^1(\mathbb{R}\times\mathbb{T}^d)$ is sufficient, see \cite{IoPa}). However, for $d=4$, due to the super-critical nature of the nonlinearity, even global existence seems to require the decay analysis we perform here. Estimate \eqref{AABA} relies on the key nonperturbative \eqref{H1NormPres}.

\begin{proof}

Let $F(t)$ be as in the statement of the theorem. Local existence theory and the fact that $t\mapsto \Vert F(t)\Vert_{S^+}$ is $C^1$ are classical (see \eqref{CrudeEst}), therefore it suffices to show the a priori estimate
\begin{equation}\label{AP}
\Vert F\Vert_{X^+_T}\le \Vert U_0\Vert_{S^+}+C\Vert F\Vert_{X^+_T}^3
\end{equation}
for all $T>0$ and all $U$ solving \eqref{CNLS} such that $\Vert F\Vert_{X^+_T}\le\sqrt{\varepsilon}$.

We pick $0\le t\le T$. Clearly, when $0\le t\le 1$, by \eqref{CrudeEst}
\begin{equation*}
 \Vert F(t)-F(0)\Vert_{S^+}\lesssim \sup_{[0,t]}\Vert \partial_tF\Vert_{S^+}\lesssim \Vert F\Vert_{X^+_T}^3.
\end{equation*}
Thus in the following, we may replace $t=0$ by $t=1$.

We start by remarking that, thanks to \eqref{CrudeEst}, we have that
\begin{equation*}
\begin{split}
\Vert \partial_tF\Vert_S&=\Vert\mathcal{N}^t[F,F,F]\Vert_S\lesssim (1+\vert t\vert)^{-1}\Vert F(t)\Vert_{S}^3\\
\Vert \partial_tF\Vert_{S^+}&=\Vert\mathcal{N}^t[F,F,F]\Vert_{S^+}\lesssim (1+\vert t\vert)^{-1}\Vert F(t)\Vert_{S}^2\Vert F(t)\Vert_{S^+}\\
\end{split}
\end{equation*}
which gives the needed bound for $\partial_t F$.

\medskip

Recall the decomposition in Proposition~\ref{StrucNon}.
For each fixed $\xi$, multiplying by $\left[1+\vert p\vert^2\right]$ and taking the inner product with $\widehat{F}(\xi)$, we obtain, after using \eqref{H1NormPres} that\footnote{A key cancellation appears here in that the resonant term $\mathcal{R}$ disappears, leaving only terms that decay faster.}
\begin{equation}\label{H1Cancellation}
\frac{d}{ds}\frac{1}{2}\Vert \widehat{F}_p(\xi,s)\Vert_{h^1_p}^2=\langle\widehat{\mathcal{E}}_1(\xi,p,s),\widehat{F}_p(\xi,s)\rangle_{h^1_p\times h^1_p}+\langle\partial_s\widehat{\mathcal{E}}_3(\xi,p,s),\widehat{F}_p(\xi,s)\rangle_{h^1_p\times h^1_p}.
\end{equation}
Using \eqref{StriEst} and \eqref{DecNon} we have that, for any $\xi$,
\begin{equation*}
\begin{split}
[1+|\xi|^2]\cdot\vert \int_0^t\langle\widehat{\mathcal{E}}_1(\xi,p,s),\widehat{F}_p(\xi,s)\rangle_{h^1_p\times h^1_p}ds\vert&\lesssim \Vert F\Vert_{X_T^+}^3\int_0^t(1+\vert s\vert)^{-1-\delta}ds\cdot\sup_{[0,t]}\Vert F(s)\Vert_Z
\end{split}
\end{equation*}
and, using \eqref{CrudeEst} and \eqref{DecNon},
\begin{equation*}
\begin{split}
&[1+|\xi|^2]\left\vert \int_0^t\langle \partial_t\widehat{\mathcal{E}}_3(\xi,p,s),\widehat{F}_p(\xi,s)\rangle_{h^1_p\times h^1_p}ds\right\vert
\le [1+|\xi|^2] \left \vert \langle \widehat{\mathcal{E}}_3(\xi,p,t),\widehat{F}_p(\xi,t)\rangle_{h^1_p\times h^1_p}\right\vert\\
&\quad+[1+|\xi|^2]\left\vert\langle\widehat{ \mathcal{E}}_3(\xi,p,0), \widehat{F}_p(\xi,0)\rangle_{h^1_p\times h^1_p}\right\vert
+[1+|\xi|^2]\left\vert\int_0^t\langle\widehat{\mathcal{E}}_3(\xi,p,s),\partial_t\widehat{F}_p(\xi,s)\rangle_{h^1_p\times h^1_p}\right\vert\\
&\lesssim \Vert F\Vert_{X_T^+}^3\cdot\sup_{t\in [0,T]}\Vert F(t)\Vert_Z+\Vert F\Vert_{X_T^+}^6
\end{split}
\end{equation*}
Combining the above estimates and integrating in time, we arrive at 
\begin{equation*}
\Vert F(t)\Vert_Z\le \Vert F(0)\Vert_Z+C\Vert F\Vert_{X_T^+}^3.
\end{equation*}
Independently, using Remark \ref{ControlR} and Proposition \ref{SStrucN} we also see that, so long as $1\le t\le T$,
\begin{equation*}
\begin{split}
\Vert F(t)-F(1)\Vert_S&\lesssim \Vert \int_1^t \mathcal{R}[F(s),F(s),F(s)]\frac{ds}{s}\Vert_S+\Vert \int_1^t\left[\mathcal{E}_1(s)+\mathcal{E}_2(s)\right]ds\Vert_{S},\\
&\lesssim (1+\vert t\vert)^\delta\Vert F\Vert_{X_T^+}^3,
\end{split}
\end{equation*}
and we may proceed similarly to control the $S^+$ norm. This gives the a priori estimate and finishes the proof.

\end{proof}

\subsection{Asymptotic behavior}

We can now give the proof of  the main theorem.

\begin{proof}[Proof of Theorem \ref{ModScatThm2}]

Define $T_n=e^{n/\pi}$ and $G_n(t)=\widetilde{G}_n(\pi\ln t)$, where $\widetilde{G}_n$ solves \eqref{RSS} with Cauchy data such that $\widetilde{G}_n(n)=G_n(T_n)=F(T_n)$.
We claim that for all $t\ge T_n$,
\begin{equation}\label{BdsGN}
\begin{split}
\Vert G_n(t)\Vert_Z+(1+\vert t\vert)^{-\delta}\Vert G_n(t)\Vert_S+(1+\vert t\vert)^{-5\delta}\Vert G_n(t)\Vert_{S^+}+(1+\vert t\vert)^{1-\delta}\Vert \partial_tG_n(t)\Vert_S\lesssim\varepsilon
\end{split}
\end{equation}
uniformly in $n\ge 0$. Indeed, first, using \eqref{H1NormPres} and \eqref{AABA}, we get that
\begin{equation*}
\Vert G_n(t)\Vert_Z=\Vert\widetilde{G}_n(\pi \ln t)\Vert_Z=\Vert\widetilde{G}_n(n)\Vert_Z=\Vert F(T_n)\Vert_Z\lesssim\varepsilon
\end{equation*}
uniformly in $n$. In addition, using also Lemma \ref{DiscreteStricTorLem} and Lemma \ref{ControlSS+}, we see that, uniformly in $n$,
\begin{equation}\label{AddedNov1}
\Vert \partial_t G_n(s)\Vert_{S}\lesssim s^{-1}\Vert G_n\Vert_Z^2\Vert G_n(s)\Vert_S\lesssim \varepsilon^2s^{-1}\Vert G_n(s)\Vert_S
\end{equation}
and since by \eqref{AABA}, $\Vert G_n(T_n)\Vert_{S}\lesssim \varepsilon T_n^\delta$, an application of Gronwall's lemma gives, for $\varepsilon$ small enough,
\begin{equation*}
\Vert G_n(s)\Vert_{S}\lesssim \varepsilon s^\delta,\quad s\ge T_n
\end{equation*}
which, combined with \eqref{AddedNov1} provides control of the second and last term in \eqref{BdsGN}. We can estimate the $S^+$ norm similarly, using Remark \ref{ControlR} and the above control to get
\begin{equation*}
\Vert \partial_t G_n(s)\Vert_{S^+}\lesssim s^{-1}\varepsilon^2\Vert G_n(s)\Vert_{S^+}+\varepsilon^3s^{-1+4\delta},\qquad \Vert G_n(T_n)\Vert_{S^+}\lesssim\varepsilon T_n^{5\delta}.
\end{equation*}
This concludes the proof of \eqref{BdsGN}.

\medskip

Now we claim that, for $T_n\le t\le T_{n+4}$,
\begin{equation}\label{BdsDiffG}
\begin{split}
\Vert F(t)-G_n(t)\Vert_{S}&\lesssim \varepsilon^3T_n^{-\delta}.\\
\end{split}
\end{equation}
Indeed, using \eqref{DecNon0}, we see that
\begin{equation*}
\begin{split}
F(t)-G_n(t)&=i\int_{T_n}^t\mathcal{E}^\sigma[F,F,F]d\sigma\\
&\quad+i\int_{T_n}^t\left\{\mathcal{R}[F(\sigma),F(\sigma),F(\sigma)]-\mathcal{R}[G_n(\sigma),G_n(\sigma),G_n(\sigma)]\right\}\frac{d\sigma}{\sigma}.
\end{split}
\end{equation*}
On the one hand, using \eqref{AABA} and Proposition \ref{StrucNon}, we obtain that
\begin{equation*}
\Vert \int_{T_n}^t\mathcal{E}^\sigma[F,F,F]d\sigma\Vert_{S}\lesssim \varepsilon^3T_n^{-2\delta}.
\end{equation*}
On the other hand, letting $X(t)=\Vert F(t)-G_n(t)\Vert_{Z}$, we see using \eqref{StriEst2} and Lemma \ref{ControlSS+} that
\begin{equation*}
\begin{split}
&\Vert \int_{T_n}^t\left\{\mathcal{R}[F(\sigma),F(\sigma),F(\sigma)]-\mathcal{R}[G_n(\sigma),G_n(\sigma),G_n(\sigma)]\right\}\frac{d\sigma}{\sigma}\Vert_Z\\
&\lesssim\int_{T_n}^t\left\{\Vert F(\sigma)\Vert_Z^2+\Vert G_n(\sigma)\Vert_Z^2\right\}X(\sigma)\frac{d\sigma}{\sigma}\lesssim\varepsilon^2\int_{T_n}^tX(\sigma)\frac{d\sigma}{\sigma}
\end{split}
\end{equation*}
so that $X(t)$ is continuous and satisfies
\begin{equation*}
X(T_n)=0,\qquad X(t)\lesssim \varepsilon^3T_n^{-2\delta}+\varepsilon^2\int_{T_n}^tX(\sigma)\frac{d\sigma}{\sigma}.
\end{equation*}
An application of Gronwall's lemma gives that $X(t)\lesssim \varepsilon^3T_n^{-2\delta}$ for $T_n\le t\le T_{n+4}$. We now define $Y(t)=\Vert F(t)-G_n(t)\Vert_{S}$. Proceeding as above, we find that $Y(T_n)=0$ and
\begin{equation*}
\begin{split}
Y(t)&\lesssim \varepsilon^3T_n^{-2\delta}+\varepsilon^2\int_{T_n}^tY(\sigma)\frac{d\sigma}{\sigma}+\int_{T_n}^t\left(\Vert F(\sigma)\Vert_Z+\Vert  G_n(\sigma)\Vert_{Z}\right)\left(\Vert F(\sigma)\Vert_{S}+\Vert  G_n(\sigma)\Vert_{S}\right)X(\sigma)\frac{d\sigma}{\sigma}\\
&\lesssim \varepsilon^3T_n^{-\delta}+\varepsilon^2\int_{T_n}^tY(\sigma)\frac{d\sigma}{\sigma}.
\end{split}
\end{equation*}
An application of Gronwall's lemma yields \eqref{BdsDiffG}.

\medskip

We now deduce from this that
\begin{equation}\label{CauchySeq}
\Vert \widetilde{G}_n(0)-\widetilde{G}_{n+1}(0)\Vert_{S}\lesssim \varepsilon^3e^{-n\delta/2}.
\end{equation}
Indeed, from \eqref{BdsDiffG}, we have that
\begin{equation*}
\begin{split}
\Vert \widetilde{G}_n(n+1)-\widetilde{G}_{n+1}(n+1)\Vert_{S}\lesssim \varepsilon^3e^{-n\delta},\qquad\Vert \widetilde{G}_n\Vert_{Z}+\Vert \widetilde{G}_{n+1}\Vert_Z\lesssim \varepsilon.
\end{split}
\end{equation*}
Using Lemma \ref{gAdmi}, $ii)$ we deduce \eqref{CauchySeq} if $\varepsilon$ is small enough. As a consequence, we see that $\{\widetilde{G}_n(0)\}_n$ is a Cauchy sequence in $S$ and therefore converges to an element $G_{0,\infty}\in S$ which satisfies that
\begin{equation*}
\Vert G_{0,\infty}\Vert_Z\lesssim\varepsilon,\qquad \Vert \widetilde{G}_n(0)-G_{0,\infty}\Vert_S\lesssim \varepsilon^3e^{-n\delta/2}.
\end{equation*}
Another application of Lemma \ref{gAdmi} gives
\begin{equation*}
\begin{split}
\sup_{[0,T_{n+2}]}\Vert G_\infty(t)-G_n(t)\Vert_S\lesssim \varepsilon^3 e^{-n\delta/4}\\
\end{split}
\end{equation*}
where $G_\infty(t)=\widetilde{G}_\infty(\pi\ln t)$ with $\widetilde{G}_\infty$ the solution of \eqref{RSS} with initial data $\widetilde{G}_\infty(0)=G_{0,\infty}$.
We deduce from this and \eqref{BdsDiffG} that
\begin{equation*}
\begin{split}
\sup_{T_n\le t\le T_{n+1}}\Vert G_\infty(t)-F(t)\Vert_{S}&\le \sup_{T_n\le t\le T_{n+1}}\Vert G_\infty(t)-G_n(t)\Vert_{S}+\sup_{T_n\le t\le T_{n+1}}\Vert G_n(t)-F(t)\Vert_{S}\\
&\lesssim \varepsilon^3e^{-n\delta/4}.
\end{split}
\end{equation*}
This finishes the proof.
\end{proof}
%%%%%%%%%%%%%%%%%%%%%%%%%%%%%%%%%%%%%%%%%%%%%%%%%%%%%%%%%%%%%%%%%%%%%%%%%%%%%%%%%%%%%%%%%%%%%%%%%%%%%%%%%%%%%%%%%%%%%%%%%%%%%%%%%%%%%%%%%%%%%%%%%%%%%%%%%%%%%%%%%%%%%%%%%%%%%%%%%%%%%%%%%%%%%%%%%%%%%%%%%%%%%%%%%%%%%%%%%%%%%%%%%%%%%%%%%%%%%%%%%%%%%%%%%%%%%%%%%%%%%%%%%%%%%%%%%%%%%%%%%%%%%%%%%%%%%%%%%%%%%%%%%%%%%%%%%%%%%%%%%%%%%%%%%%%%%%%%%%%
\section{Additional estimates}\label{SAE}
%%%%%%
\begin{lemma}\label{DiscreteStricTorLem}
Let $R$ be defined as in \eqref{RS}. For every sequences $(a^1)_p$, $(a^2)_p$, $(a^3)_p$ indexed by $\mathbb{Z}^d$, $d\le 4$, 
\begin{equation}\label{StriEst2}
\Vert R[a^1,a^2,a^3]\Vert_{l^2_p}\leq C_d\,
\min_{\tau\in\mathfrak{S}_3}\Vert a^{\tau(1)}\Vert_{l^2_p}\Vert a^{\tau(2)}\Vert_{h^1_p}\Vert a^{\tau(3)}\Vert_{h^1_p}.
\end{equation}
and consequently, for any $\sigma\geq 0$,
\begin{equation}\label{StriEst}
\Vert R[a^1,a^2,a^3]\Vert_{h^\sigma_p}\leq C_{\sigma,d}\,
\sum_{\tau\in\mathfrak{S}_3}\Vert a^{\tau(1)}\Vert_{h^\sigma_p}\Vert a^{\tau(2)}\Vert_{h^1_p}\Vert a^{\tau(3)}\Vert_{h^1_p}.
\end{equation}
\end{lemma}

\begin{proof}[Proof of Lemma \ref{DiscreteStricTorLem}]

One can deduce \eqref{StriEst} from \eqref{StriEst2} in a way similar to Lemma \ref{ControlSS+}.

By duality, we need to prove that
\begin{equation}\label{aimt1}
\Big|\sum_{\substack{p_0+p_2=p_1+p_3\\\vert p_0\vert^2+\vert p_2\vert^2=\vert p_1\vert^2+\vert p_3\vert^2}}a^0_{p_0}a^1_{p_1}a^2_{p_2}a^3_{p_3}
\Big|
\lesssim \Vert a^0\Vert_{l^2_p}\min_{\tau\in\mathfrak{S}_3}\Vert a^{\tau(1)}\Vert_{l^2_p}\Vert a^{\tau(2)}\Vert_{h^1_p}\Vert a^{\tau(3)}\Vert_{h^1_p}.
\end{equation}
We will reduce \eqref{aimt1} to a bound on free solutions on the torus $\T^d$.
Indeed, if we set  
\begin{equation*}
\phi_j(y)=\sum_{p\in\Z^d} \widetilde{a}^j_p e^{ip\cdot y}: \T^d \to \C,\quad j=0,1,2,3,
\end{equation*}
with $\widetilde a^j=a^j$ if $j=1,3$ and $\widetilde a_j=\overline{a^j}$ for $j=0,2$, then we have the identity 
\begin{equation*}
\sum_{\substack{p_0+p_2=p_1+p_3\\\vert p_0\vert^2+\vert p_2\vert^2=\vert p_1\vert^2+\vert p_3\vert^2}}a^0_{p_0}a^1_{p_1}a^2_{p_2}a^3_{p_3}=
\int_{\T_y^d\times \T_t}u_1(y,t)\overline{u_{2}(y,t)}u_3(y,t)\overline{u_{0}(y,t)}\,dydt\,,
\end{equation*}
where $u_{j}(y,t)=e^{it\Delta_{\T^d}}(\phi_j(y))$, $j=0,1,2,3$.
Therefore \eqref{aimt1} follows from 
\begin{equation}\label{tuto}
\Big\vert\int_{\T_y^d\times \T_t}
\prod_{j=0}^3 \widetilde{u}_j(y,t)\, dy \,dt
\Big\vert\lesssim\Vert\phi_0\Vert_{L^2_y} \min_{\tau\in\mathfrak{S}_3}\
\|\phi_{\tau(1)}\|_{L^2_y}\|\phi_{\tau(2)}\|_{H^1_y}\|\phi_{\tau(3)}\|_{H^1_y},
\end{equation}
where $L^2_y$ and $H^1_y$ denote the corresponding Sobolev norms on $\T^d$ and $\widetilde{u}_j\in\{u_j,\overline{u_j}\}$.
Estimate \eqref{tuto} follows from the analysis in \cite{Bo2,Bourgain2013,HeTaTz2} as we explain below.
By a slight abuse of notation, inside this proof, we denote again by $P_N$ the Littlewood-Paley projector on dyadic scales for functions on the torus $\T^d$. 
By simple renormalization and symmetry arguments, the estimate \eqref{tuto} can be reduced to
\begin{equation}\label{aimt12}
\sum_{\substack{N_0 \lesssim N_1\\N_3\leq N_2 \leq N_1}}
(N_2N_3)^{-1}\Big|\int_{\T^{d+1}} P_{N_0}\widetilde{u_{0}} P_{N_1}\widetilde{u_{1}}P_{N_2}\widetilde{u_{2}}P_{N_3}\widetilde{u_{3}}\Big|
\lesssim
\prod_{j=0}^{3}\|\phi_j\|_{L^2_y}.
\end{equation}

At this stage, we invoke the classical $L^4$ Strichartz estimates by Bourgain \cite{Bo2},
\begin{equation}\label{classical}
\|P_N e^{it\Delta_{\T^d}}\phi\|_{L^4_{y,t}(\T^{d+1})}\lesssim N^{s(d)}\|\phi\|_{L^2_y},
\end{equation}
where $s(1)=0$, $s(d)=\frac{d-2}{4}+\varepsilon$ for every $\varepsilon>0$ when $d=2,3$ and $s(4)=\frac{d-2}{4}=\frac{1}{2}$, when $d=4$.
Using the Galilean invariance of the Schr\"odinger equation (see e.g. \cite[page 338]{HeTaTz}) one deduces from \eqref{classical} the bound
\begin{equation}\label{Gali}
\|P_C e^{it\Delta_{\T^d}}\phi\|_{L^4_{y,t}(\T^{d+1})}\lesssim N^{s(d)}\|\phi\|_{L^2_y},
\end{equation}
where $C$ is a cube of $\Z^d$ with side length $N\geq 1$ and $P_C$ is the corresponding Fourier projector operator. Using \eqref{Gali} one gets a bilinear refinement of \eqref{classical},
\begin{equation}\label{bili}
\|(P_{N_1} e^{it\Delta_{\T^d}}\phi_1)
(P_{N_2} e^{it\Delta_{\T^d}}\phi_2)
\|_{L^2_{y,t}(\T^{d+1})}\lesssim N_2^{2s(d)}\|\phi_1\|_{L^2_y}\|\phi_2\|_{L^2_y}\,,
\end{equation}
where $N_2\leq N_1$. Indeed to get \eqref{bili}, it suffices to decompose the dyadic ring of size $N_1$ into cubes of size $N_2$, to use an orthogonality
argument in the spatial variable and to invoke \eqref{Gali}. Now, we estimate the left hand-side of \eqref{aimt12}, by using the Cauchy-Schwarz inequality (pairing $P_{N_0}u_{0} P_{N_2}u_{2}$ and  $P_{N_1}u_{1} P_{N_3}u_{3}$) in two ways depending on whether $N_2\leq N_0$ or not
and by invoking \eqref{bili}, as follows
\begin{equation}\label{sum-up}
\sum_{\substack{N_0 \sim N_1\\N_3\leq N_2 \leq N_0}}(N_2N_3)^{-1}(N_2N_3)^{2s(d)}\prod_{j=0}^{3}\|P_{N_j}\phi_j\|_{L^2_y}.\\
\end{equation}
Since for $d=1,2,3$, we have $2s(d)<1$ the expression \eqref{sum-up} sums properly. This ends the proof for $d=1,2,3$.

For $d=4$ the above argument does not suffice to conclude because of a lack of summability in $N_2$ and $N_3$. This causes a significant difficulty which
may be resolved by using the more recent works  \cite{Bourgain2013} and \cite{HeTaTz2} as we now explain.
In \cite{Bourgain2013} the $4$-dimensional estimate \eqref{classical} is improved to
\begin{equation}\label{new-classical}
\|P_N e^{it\Delta_{\T^4}}\phi\|_{L^q_{y,t}(\T^{4+1})}\lesssim N^{2-\frac{6}{q}}\|\phi\|_{L^2_y},\quad q>\frac{7}{2}\,.
\end{equation}
Observe that for $d=4$, the bound \eqref{classical} follows from \eqref{new-classical} via an interpolation with the elementary $L^{\infty}$ bound
\begin{equation}\label{elem}
\|P_N e^{it\Delta_{\T^4}}\phi\|_{L^\infty_{y,t}(\T^{4+1})}\lesssim N^{2}\|\phi\|_{L^2_y}\,.
\end{equation}
With \eqref{new-classical} in hand, we can substitute \eqref{Gali} by the more refined bound  
\begin{equation}\label{Gali-refine}
\|P_C e^{it\Delta_{\T^d}}\phi\|_{L^4_{y,t}(\T^{4+1})}\lesssim N^{\frac{1}{2}}\left(\frac{M}{N}\right)^\delta  \|\phi\|_{L^2_y}\, ,
\end{equation}
for a suitable $\delta>0$, where now $C$ is a ``rectangle" of the form 
$$
C=\{
n\in\Z^4\,:\,
|n-n_0|\leq N,\, 
|a\cdot n-c_0|\leq M
\}
$$
for some $n_0, c_0\in\R^4$ and $a\in\R^4$, $|a|=1$.
The proof of \eqref{Gali-refine} follows by an interpolation between \eqref{new-classical} and an $L^\infty$ bound of type \eqref{elem}
(eventhough elementary, the $L^\infty$ bound is sensitive to the size of $C$ which is crucial for getting the improvement \eqref{Gali-refine}).
Using \eqref{Gali-refine}, we may invoke \cite[Proposition~2.8]{HeTaTz2}, to get the following improvement of \eqref{bili} for $d=4$:
\begin{equation}\label{bili-mieux}
\|(P_{N_1} e^{it\Delta_{\T^4}}\phi_1)(P_{N_2} e^{it\Delta_{\T^4}}\phi_2)\|_{L^2_{y,t}(\T^{d+1})}
\lesssim N_2\Big(\frac{N_2}{N_1}+\frac{1}{N_2}\Big)^\delta\|\phi_1\|_{L^2_y}\|\phi_2\|_{L^2_y}\,,
\end{equation}
for some $\delta>0$, where again $N_2\leq N_1$.
Compared to the proof of \eqref{bili}, the proof of \eqref{bili-mieux} uses an additional almost orthogonality argument in the time variable via an application of
\eqref{Gali-refine} with $M=\max(1,N_1^2/N_2)$ (and $N=N_1$). 
Using \eqref{bili-mieux}, we replace \eqref{sum-up} (for $d=4$) by
\begin{equation*}
\sum_{\substack{N_1 \sim N_0\\N_3\leq N_2 \leq N_0}}\Big(\frac{N_2}{N_0}+\frac{1}{N_2}\Big)^\delta\Big(\frac{N_3}{N_1}+\frac{1}{N_3}\Big)^\delta
\prod_{j=0}^{3}\|P_{N_j}\phi_j\|_{L^2_y}.
\end{equation*}
This expression now sums properly. This completes the proof of Lemma~\ref{DiscreteStricTorLem}.
\end{proof}
\medskip
Next, we recall the one dimensional bilinear Strichartz estimates.
\begin{lemma}\label{1dBE}
Assume that $\lambda\geq 10 \mu\geq 1$ and that $u(t)=e^{it\partial_{xx}}u_0$, $v(t)=e^{it\partial_{xx}}v_0$.
Then, we have the bound
\begin{equation}\label{BS}
\Vert
Q_{\lambda}u\overline{Q_{\mu} v}
\Vert_{L^2_{x,t}(\mathbb{R}\times\mathbb{R})}\lesssim \lambda^{-\frac{1}{2}}\Vert u_0\Vert_{L^2_x(\mathbb{R})}\Vert v_0\Vert_{L^2_x(\mathbb{R})}.
\end{equation}
\end{lemma}
We refer to \cite{CKSTTSIMA} for the proof of Lemma~\ref{1dBE} (see also \cite{BoIMRN} for the earlier higher dimensional analogue of \eqref{BS} and
\cite{SGuo} for recent closely related estimates).
%%%%
\begin{lemma}\label{DispLem}
Assume that $N\geq 7$. Then we have the bound
\begin{equation}\label{OptDecay}
\sup_{x\in\mathbb{R}}\sum_{p\in\mathbb{Z}^d}\left[1+\vert p\vert^2\right]\vert e^{it\partial_{xx}}F_p(x)\vert^2\lesssim 
\langle t \rangle^{-1}
\Big(
\|F\|_{Z}^2+\langle t\rangle^{-\frac{1}{4}}\big(\|x F\|_{L^2}^2+\|F\|_{H^N}^2\big)
\Big).
\end{equation}
\end{lemma}
\begin{proof}%[Proof of Lemma \ref{DispLem}]
It suffices to prove the statement for $t\geq 1$, for $|t|\leq 1$, it simply follows from the Sobolev embedding, and for $t\leq -1$, it follows by symmetry.
We first claim that there exists a constant $c$ such that
\begin{equation}\label{ClaimDec1}
\vert e^{it\partial_{xx}}f(x)-c\frac{e^{-i\frac{x^2}{4t}}}{\sqrt{t}}\widehat{f}(-\frac{x}{2t})\vert\lesssim t^{-\frac{3}{4}}\Vert xf\Vert_{L^2}\,.
\end{equation}
Indeed, one can write
\begin{equation*}
e^{it\partial_{xx}}f(x)
=e^{-i\frac{x^2}{4t}}\int_{\mathbb{R}}e^{it\eta^2}\widehat{f}(\eta-\frac{x}{2t})d\eta
=e^{-i\frac{x^2}{4t}}\Big(\sum_{l=- 1}^{-\infty}I_l(x,t)+ I(x,t)\Big),
\end{equation*}
where
$$
I_l(x,t):=\int_{\mathbb{R}}e^{it\eta^2}\phi(2^{-l}\eta)\widehat{f}(\eta-\frac{x}{2t})d\eta,\quad
I(x,t):=\int_{\mathbb{R}}e^{it\eta^2}\tilde{\phi}(\eta)\widehat{f}(\eta-\frac{x}{2t})d\eta,
$$
for suitable bump functions $\phi$ and $\tilde{\phi}$ such that the support of $\phi$ does not meet zero.
By a crude estimate, we first get that
\begin{equation}\label{SP1E1}
\vert I_{l}(x,t)-\widehat{f}(-\frac{x}{2t})\int_{\mathbb{R}}e^{it\eta^2}\phi(2^{-l}\eta)d\eta\vert\lesssim  2^\frac{3l}{2} \Vert\partial_\xi\widehat{f}\Vert_{L^2}.
\end{equation}
On the other hand, an integration by parts gives that
\begin{equation*}
I_l(x,t)-\widehat{f}(-\frac{x}{2t})
\int_{\mathbb{R}}e^{it\eta^2}
\phi(2^{-l}\eta)d\eta=
\frac{1}{2it}\int_{\mathbb{R}}e^{it\eta^2}\partial_\eta\left[\frac{1}{\eta}\phi(2^{-l}\eta)
\big(\widehat{f}(\eta-\frac{x}{2t})-\widehat{f}(-\frac{x}{2t})\big)\right]d\eta.
\end{equation*}
Therefore
\begin{equation}\label{SPIPP}
\Big|
I_l(x,t)-\widehat{f}(-\frac{x}{2t})
\int_{\mathbb{R}}e^{it\eta^2}
\phi(2^{-l}\eta)d\eta
\Big|
\lesssim t^{-1}2^{-\frac{l}{2}}
\Vert \partial_\xi\widehat{f}\Vert_{L^2}.
\end{equation}
One also gets a similar bound for $I(x,t)$ (with $l=0$).
Since 
$
\int_{\mathbb{R}}e^{it\eta^2}
d\eta=ct^{-1/2},
$
using \eqref{SP1E1} for $l\le -\frac{1}{2}\log_{2} t$ and \eqref{SPIPP} otherwise, summing over $l$, we recover \eqref{ClaimDec1}. 

Now, we deduce that
\begin{equation*}
t\sum_{\substack{p\in\mathbb{Z}^d\\|p|\le  t^{1/8}}}\left[1+\vert p\vert^2\right]\vert e^{it\partial_{xx}}F_p(x)\vert^2
\lesssim\sum_{p\in\mathbb{Z}^d}\left[1+\vert p\vert^2\right]\vert\widehat{F}_p(-\frac{x}{2t})\vert^2
+t^{-\frac{1}{2}}\sum_{\substack{p\in\mathbb{Z}^d\\|p|\le  t^{1/8}}}[1+|p|^2]\Vert x F_p\Vert_{L^2}^2 .
\end{equation*}
On the other hand, we also have that
\begin{equation*}
t\sum_{|p|\geq t^{1/8}}\left[1+\vert p\vert^2\right]\vert e^{it\partial_{xx}}F_p(x)\vert^2\lesssim t^{1-\frac{N-2}{4}}
\sum_{p\in \Z^d}(1+|p|^2)^{N-1}\|F_p\|_{H^1}^2\lesssim t^{-1/4}\|F\|_{H^N}^2
\end{equation*}
provided that $N \geq 7$. This finishes the proof of Lemma~\ref{DispLem}.
\end{proof}
%%%%%%%%%%%%%%%%%%%%%%%%%%%%%%%%%%%%%%%%%%%%%%%%%
%\begin{remark}\label{NoDecay}
%The above estimate cannot be improved much because the inequality
%\begin{equation*}
%\Vert e^{it\Delta_{\mathbb{R}\times\mathbb{T}^d}}F(t)\Vert_{L^\infty_{x,y}}\lesssim (1+\vert t\vert)^{-1/2}\left[\Vert F\Vert_Z+(1+\vert t\vert)^{-\delta}\Vert F\Vert_S\right]
%\end{equation*}
%is false when $d\ge 2$. Indeed, let us  define
%\begin{equation*}
%F^N_p(x,T)=N^{-\frac{d+2}{2}}\phi(p/N)\varphi(x).
%\end{equation*}
%When $d\geq 3$, it suffices to choose $N=\log\log T$, $F(T)=F^N$, $y=0$, $T\in\mathbb{N}$, $T\gg1$. When $d=2$, one may choose instead
%\begin{equation*}
%F=\sum_{k=1}^N\frac{1}{k}F^{2^kN}.
%\end{equation*}
%\end{remark}
%%%%%%%%%%%%%%%%%%%%%%%%%%%%%%%%%%%%%%%%%%%%%%
%%%%%%%%%%%%%%%%%%%%%%%%%%%%%%%%%%%%%%%%%%%%%%%

\bigskip
We now turn to our basic lemma allowing to transform suitable $L^2_{x,y}$ bounds to bounds in terms of the $L^2_{x,y}$-based spaces $S$ and $S^+$.
We define an \emph{LP-family} $\widetilde{Q}=\{\widetilde{Q}_A\}_A$ to be a family of operators (indexed by the dyadic integers) of the form
\begin{equation*}
\widehat{\widetilde{Q}_1f}(\xi)=\widetilde{\varphi}(\xi)\widehat{f}(\xi),\qquad \widehat{\widetilde{Q}_Af}(\xi)=\widetilde{\phi}(\frac{\xi}{A})\widehat{f}(\xi), \quad A\ge 2
\end{equation*}
for two smooth functions $\widetilde{\varphi},\widetilde{\phi}\in C^\infty_c(\mathbb{R})$ with $\widetilde{\phi}\equiv0$ in a neighborhood of $0$.

We define the set of \emph{admissible transformations} to be the family of operators $\{T_B\}$ where for any $B$,
\begin{equation*}
T_B=\lambda_B\widetilde{Q}_B,\qquad \vert\lambda_B\vert\le 1
\end{equation*}
for some LP-family $\widetilde{Q}$. Given an trilinear operator 
$\mathfrak{T}$ and a set $\Lambda$ of 4-tuples of dyadic integers, we define an \emph{admissible realization} of $\mathfrak{T}$ at $\Lambda$ to be an operator of the form
\begin{equation*}
\mathfrak{T}_\Lambda[F,G,H]=\sum_{(A,B,C,D)\in\Lambda}T_D\mathfrak{T}[T^\prime_AF,T^{\prime\prime}_BG,T^{\prime\prime\prime}_CH]
\end{equation*}
for admissible transformations $T$, $T^\prime$, $T^{\prime\prime}$, $T^{\prime\prime\prime}$.

A norm $\mathcal{B}$ is called admissible if for any admissible transformation $T=\{T_A\}_A$, there holds that
\begin{equation}\label{OpComp}
\Vert \sum_A T_AF\Vert_{\mathcal{B}}\lesssim \Vert F\Vert_{\mathcal{B}}.
\end{equation}
We note that all norms that we consider are admissible.
\begin{lemma}\label{ControlSS+}
Assume that a trilinear operator $\mathfrak{T}$ satisfies
\begin{equation}\label{LeibnitzRule}
\begin{split}
Z\mathfrak{T}[F,G,H]= \mathfrak{T}[ZF,G,H]+\mathfrak{T}[F,ZG,H]+\mathfrak{T}[F,G,ZH],
\end{split}
\end{equation}
for $Z\in\{x,\partial_x,\partial_{y_1},\dots,\partial_{y_d}\}$ and let $\Lambda$ be a set of $4$-tuples of dyadic integers. With the notation introduced above, assume also that for all admissible realizations of $\mathfrak{T}$ at $\Lambda$,
\begin{equation}\label{L2AssFL}
\Vert \mathfrak{T}_\Lambda[F^a,F^b,F^c]\Vert_{L^2}\le K\min_{\sigma\in\mathfrak{S}_3}\Vert F^{\sigma(a)}\Vert_{L^2}\Vert F^{\sigma(b)}\Vert_{\mathcal{B}}\Vert F^{\sigma(c)}\Vert_{\mathcal{B}}
\end{equation}
for some admissible norm $\mathcal{B}$ such that the Littlewood-Paley projectors $P_{\le M}$ (both in $x$ and in $y$) are uniformly bounded on $\mathcal{B}$. Then, for all admissible realizations of $\mathfrak{T}$ at $\Lambda$,
\begin{equation}\label{ControlSSS}
\begin{split}
\Vert \mathfrak{T}_\Lambda[F^a,F^b,F^c]\Vert_{S}&\lesssim K\max_{\sigma\in\mathfrak{S}_3}\Vert F^{\sigma(a)}\Vert_{S}\Vert F^{\sigma(b)}\Vert_{\mathcal{B}}\Vert F^{\sigma(c)}\Vert_{\mathcal{B}}\\
\end{split}
\end{equation}
Assume in addition that, for $Y\in\{x,(1-\partial_{xx})^4\}$,
\begin{equation}\label{L2AssFL2}
\Vert YF\Vert_{\mathcal{B}}\lesssim \theta_1\Vert F\Vert_{S^+}+\theta_2\Vert F\Vert_S,
\end{equation}
then for all admissible realizations of $\mathfrak{T}$ at $\Lambda$,
\begin{equation}\label{ControlSSS+}
\begin{split}
\Vert \mathfrak{T}_\Lambda[F^a,F^b,F^c]\Vert_{S^+}&\lesssim K\max_{\sigma\in\mathfrak{S}_3}\Vert F^{\sigma(a)}\Vert_{S^+}\big(\Vert F^{\sigma(b)}\Vert_{\mathcal{B}}+\theta_1\Vert F^{\sigma(b)}\Vert_{S}\big) \Vert F^{\sigma(c)}\Vert_{\mathcal{B}}\\
&\quad+\theta_2K\max_{\sigma\in\mathfrak{S}_3}\Vert F^{\sigma(a)}\Vert_{S}\Vert F^{\sigma(b)}\Vert_{S}\Vert F^{\sigma(c)}\Vert_{\mathcal{B}}\\
\end{split}
\end{equation}
\end{lemma}

%%%%%%%%%%%%%%%%%%%%%%%%
\begin{proof}
The main information we need comes from the computations of the simple commutators
\begin{equation}\label{Comm}
\begin{split}
[x,\widetilde{Q}_A]=A^{-1}\widetilde{Q}^\prime_A,
\end{split}
\end{equation}
where if $\widetilde{Q}$ corresponds to the family $(\widetilde{\varphi},\widetilde{\phi})$, $\widetilde{Q}^\prime$ corresponds to $(\widetilde{\varphi}^\prime,\widetilde{\phi}^\prime)$.
Clearly \eqref{Comm} defines admissible transformations.
We may assume that
\begin{equation*}
\Vert F^a\Vert_{\mathcal{B}}=\Vert F^b\Vert_{\mathcal{B}}=\Vert F^c\Vert_{\mathcal{B}}=1,\qquad K=1.
\end{equation*}

We let $\mathfrak{T}_\Lambda$ be an arbitrary admissible realization of $\mathfrak{T}$ at $\Lambda$ (this realization may change from line to line, or even in the same line).
For $Z\in \{\partial_x,\partial_{y_1},\dots,\partial_{y_d}\}$, let $\mathcal{P}_{\nu}$ be the projector associated to $\vert Z\vert$ (e.g. $\mathcal P_\nu=\phi(\frac{|Z|}{\nu})$). 
Then we can decompose
$$
\mathcal{P}_{\nu}\mathfrak{T}_\Lambda[F^a,F^b,F^c]=\mathcal{P}_{\nu}\Sigma_{\nu,low}+\mathcal{P}_{\nu}\Sigma_{\nu,high}\, ,
$$
where 
$$
\Sigma_{\nu,low}:=\mathfrak{T}_\Lambda[\mathcal{P}_{\le \nu}F^a,\mathcal{P}_{\le \nu}F^b,\mathcal{P}_{\le \nu}F^c]
$$
and
\begin{equation*}
\Sigma_{\nu,high}:= \mathfrak{T}_\Lambda[\mathcal{P}_{\ge 2\nu}F^a,F^b,F^c]+\mathfrak{T}_\Lambda[\mathcal{P}_{\le \nu}F^a,\mathcal{P}_{\ge 2\nu}F^b,F^c]
+\mathfrak{T}_\Lambda[\mathcal{P}_{\le \nu}F^a,\mathcal{P}_{\le \nu}F^b,\mathcal{P}_{\ge 2\nu}F^c].
\end{equation*}

\medskip

Using the boundedness of $\mathcal{P}_\nu$ on $L^2$, we remark that, using the Leibnitz rule \eqref{LeibnitzRule}, for $s$ a positive integer,
\begin{equation*}
\begin{split}
\Vert Z^{s}\mathcal{P}_\nu\Sigma_{\nu,low}\Vert_{L^2}&\lesssim \nu^{-s}\Vert Z^{2s}\mathcal{P}_\nu\Sigma_{\nu,low}\Vert_{L^2}\\
&\lesssim \nu^{-s}\sum_{\alpha, \beta, \gamma \le \nu}\,\,\sum_{t+u+v\le2s}\Vert \mathfrak{T}_\Lambda[Z^t\mathcal{P}_\alpha F^a,Z^{u}\mathcal{P}_\beta F^b,Z^v\mathcal{P}_{\gamma}F^c\Vert_{L^2}.
\end{split}
\end{equation*}
Assume first that $\alpha\geq \beta, \gamma$. Using \eqref{L2AssFL}, and summing over $\beta,\gamma$,
\begin{equation*}
\begin{split}
\nu^{-s}\sum_{\beta,\gamma\le \alpha\le \nu}\,\,\sum_{t+u+v\le2s}\Vert \mathfrak{T}_\Lambda[Z^t\mathcal{P}_{\alpha} F^a,Z^{u}\mathcal{P}_{\beta}F^b,Z^v\mathcal{P}_{\gamma}F^c\Vert_{L^2}
&\lesssim \sum_{\alpha\le \nu}\left(\frac{\alpha}{\nu}\right)^{s}\Vert \mathcal{P}_\alpha Z^{s}F^a\Vert_{L^2}
\end{split}
\end{equation*}
The above sum is in $l^2_\nu$. We may proceed similarly for the case $\beta\ge \alpha,\gamma$ and the case $\gamma\ge \alpha,\beta$.

\medskip

To treat $\Sigma_{\nu,high}$, we simply use \eqref{L2AssFL} to get
\begin{equation*}
\Vert Z^s\mathcal{P}_\nu\mathfrak{T}_\Lambda[\mathcal{P}_{\ge 2\nu}F^a,F^b,F^c]\Vert_{L^2}\lesssim \nu^{s}\Vert \mathfrak{T}_\Lambda[\mathcal{P}_{\ge 2\nu}F^a,F^b,F^c]\Vert_{L^2}
\lesssim \nu^{s}\Vert \mathcal{P}_{\ge 2\nu}F^a\Vert_{L^2},
\end{equation*}
which is in $l^2_\nu$, thanks to a standard argument.

This already accounts for most of the components of the $S$-norm, except for the term involving $x$. We first remark that, 
\begin{equation*}
\begin{split}
x\mathfrak{T}_\Lambda[F,G,H]&=\mathfrak{T}_\Lambda[xF,G,H]+\mathfrak{T}_\Lambda[F,xG,H]+\mathfrak{T}_\Lambda[F,G,xH]\\
&\quad+\sum_{(A,B,C,D)\in \Lambda}[x,T_D]\mathfrak{T}[T^\prime_AF,T^{\prime\prime}_BG,T^{\prime\prime\prime}_CH]+\sum_{(A,B,C,D)\in \Lambda}T_D\mathfrak{T}[[x,T^\prime_A]F,T^{\prime\prime}_BG,T^{\prime\prime\prime}_CH]\\
&\quad+\sum_{(A,B,C,D)\in \Lambda}T_D\mathfrak{T}[T^\prime_AF,[x,T^{\prime\prime}_B]G,T^{\prime\prime\prime}_CH]+\sum_{(A,B,C,D)\in \Lambda}T_D\mathfrak{T}[T^\prime_AF,T^{\prime\prime}_BG,[x,T^{\prime\prime\prime}_C]H].\\
\end{split}
\end{equation*}
In view of \eqref{Comm}, we thus see that
\begin{equation}\label{xTDec}
\begin{split}
x\mathfrak{T}_\Lambda[F,G,H]=\mathfrak{T}_\Lambda[xF,G,H]+\mathfrak{T}_\Lambda[F,xG,H]+\mathfrak{T}_\Lambda[F,G,xH]+\mathfrak{T}_\Lambda[F,G,H].
\end{split}
\end{equation}
At this point, we see that all terms in \eqref{xTDec} are of the form already controlled before. This finishes the proof of \eqref{ControlSSS}.

\medskip

Now from \eqref{xTDec} and \eqref{ControlSSS}, we see directly that
\begin{equation*}
\begin{split}
\Vert x\mathfrak{T}_\Lambda[F^a,F^b,F^c]\Vert_S\lesssim \sup_{\sigma\in\mathfrak{S}_3}\Vert F^{\sigma(a)}\Vert_{S^+}\Vert F^{\sigma(b)}\Vert_B\Vert F^{\sigma(c)}\Vert_B
+\sup_{\sigma\in\mathfrak{S}_3}\Vert F^{\sigma(a)}\Vert_{S}\Vert xF^{\sigma(b)}\Vert_B\Vert F^{\sigma(c)}\Vert_B
\end{split}
\end{equation*}
and assuming \eqref{L2AssFL2}, we can bound this by the right-hand side of \eqref{ControlSSS+}. The term of the $S^+$ norm where $x$ is replaced by 
$(1-\partial_{xx})^4$ 
can be treated similarly to the above analysis.
This completes the proof of Lemma~\ref{ControlSS+}.
\end{proof}
%%%%%%%%%%%%%%%%%%%%%%%%%%%%%%%%%%%%%%%%%%%
We shall also need the following multilinear estimate.
\begin{lemma}\label{CM}
Let
\begin{equation*}
\frac{1}{p}=\frac{1}{q}+\frac{1}{r}+\frac{1}{s},\qquad 1\le p,q,r,s\le\infty,
\end{equation*}
then
\begin{equation*}
\begin{split}
\Vert \int_{\mathbb{R}^3}e^{ix\xi}m(\eta,\kappa)\widehat{f}(\xi-\eta)\overline{\widehat{g}}(\xi-\eta-\kappa)\widehat{h}(\xi-\kappa)d\eta d\kappa d\xi\Vert_{L^p}\lesssim\Vert\mathcal{F}^{-1}m\Vert_{L^1(\mathbb{R}^2)}\Vert f\Vert_{L^q}\Vert g\Vert_{L^r}\Vert h\Vert_{L^s}.
\end{split}
\end{equation*}
\end{lemma}
\medskip
The proof of Lemma~\ref{CM} follows from an application of the Parseval identity, the H\"older inequality and an approximation argument.
%%%%%%%%%%

\end{document}